\documentclass[UTF-8,reqno]{amsart}
\usepackage{enumerate, bbm}
\usepackage{amssymb,url,color, booktabs}
\usepackage{mathrsfs}
\usepackage{soul}
\usepackage{color}
\usepackage[colorlinks=true]{hyperref}
\hypersetup{
    linkcolor=blue,          
    citecolor=red,        
    filecolor=blue,      
    urlcolor=cyan
}
\usepackage{color}
\definecolor{MyDarkBlue}{cmyk}{0.8,0.3,0.8,0.4}
\definecolor{yellow}{rgb}{0.99,0.99,0.70}
\definecolor{white}{rgb}{1.0,1.0,1.0}
\definecolor{black}{rgb}{0.00,0.00,0.00}
\definecolor{backgroundcolor}{RGB}{199,238,206}

\numberwithin{equation}{section}

\newtheorem{theorem}{Theorem}[section]
\newtheorem{lemma}[theorem]{Lemma}
\newtheorem{remark}[theorem]{Remark}
\newtheorem{definition}[theorem]{Definition}
\newtheorem{proposition}[theorem]{Proposition}
\newtheorem{Examples}[theorem]{Example}
\newtheorem{corollary}[theorem]{Corollary}

\def\eps{\varepsilon}
\def\e{\mathrm{e}}
\def\supp{\mathrm{supp}}

\def\dif{{\mathord{{\rm d}}}}

\def\max{{\mathord{{\rm max}}}}
\def\min{{\mathord{{\rm min}}}}

\def\mR{{\boldsymbol{r}}}

\def\bbone{{\boldsymbol{1}}}
\def\bb2{{\boldsymbol{2}}}
\def\no{\nonumber}
\def\={&\!\!=\!\!&}

\def\bx{{\mathbf{x}}}

\def\bB{{\mathbf B}}
\def\bC{{\mathbf C}}

\def\b1{{\mathbbm 1}}

\def\cI{{\mathcal I}}
\def\cJ{{\mathcal J}}

\def\cQ{{\mathcal Q}}
\def\cR{{\mathcal R}}
\def\cS{{\mathcal S}}
\def\cT{{\mathcal T}}
\def\cU{{\mathcal U}}
\def\cV{{\mathcal V}}
\def\cW{{\mathcal W}}

\def\mE{{\mathbb E}}

\def\mN{{\mathbb N}}

\def\mR{{\mathbb R}}

\def\mT{{\mathbb T}}

\def\mZ{{\mathbb Z}}

\def\sS{{\mathscr S}}

\def\geq{\geqslant}
\def\leq{\leqslant}
\def\<{{\langle}}
\def\>{{\rangle}}
\def\({{\big(}}
\def\){{\big)}}
\def\[{{\Big[}}
\def\]{{\Big]}}

\def\a{\alpha}
\def\b{\beta}
\def\de{\delta}
\def\g{\gamma}

\def\l{\lambda}

\def\s{\sigma}

\def\vt{\vartheta}

\def\ff{\frac}

\def\bt{\begin{theorem}}
\def\et{\end{theorem}}
\def\bl{\begin{lemma}}
\def\el{\end{lemma}}
\def\br{\begin{remark}}
\def\er{\end{remark}}
\def\bpf{\begin{proof}}
\def\epf{\end{proof}}
\def\bx{\begin{Examples}}
\def\ex{\end{Examples}}
\def\bd{\begin{definition}}
\def\ed{\end{definition}}
\def\bp{\begin{proposition}}
\def\ep{\end{proposition}}
\def\bc{\begin{corollary}}
\def\ec{\end{corollary}}

\def\wt{\widetilde}

\allowdisplaybreaks

\begin{document}

\title[Quintic Approximation of $\Phi^4_3$ I]{Convergence to the Dynamical $\Phi^4_3$ Model under a Vanishing Quintic Perturbation I: A Paracontrolled Approach}
\author[Z. Chen]{Zikai Chen}
\author[S. Kusuoka]{Seiichiro Kusuoka}

\thanks{Zikai Chen: School of Mathematics and Statistics, Wuhan University, Wuhan, Hubei, 430072, China, and Department of Mathematics Sciences, Graduate School of Science, Kyoto University, Kyoto, 606-8502, Japan. E-mail: chenzikai@whu.edu.cn}
\thanks{Seiichiro Kusuoka: Department of Mathematics and Mathematical Sciences, Graduate School of Science, Kyoto University, Kyoto, 606-8102, Japan. E-mail:kusuoka@math.kyoto-u.ac.jp}
\thanks{Z. Chen acknowledges support from China Scholarship Council Grant Numbers 202506270058. S. Kusuoka acknowledges support from JSPS KAKENHI Grant Numbers JP22H00099 and JP23K20801.}

\begin{abstract}
We study local convergence to the dynamical $\Phi^4_3$ model under a vanishing quintic perturbation. More precisely, on the three-dimensional torus we consider
$$\partial_tu_\varepsilon=\Delta u_\varepsilon-\varepsilon^\alpha u_\varepsilon^5+\xi_\varepsilon+C_\varepsilon u_\varepsilon+\widetilde C_\varepsilon u_\varepsilon^3$$
for $\alpha\in(\ff56,1)$, where $\xi_\varepsilon$ is a spatial mollification of space-time white noise. Although the quintic coefficient vanishes, its contractions generate divergent linear and cubic contributions. We identify suitable mass and cubic counterterms that compensate these divergences. Using paracontrolled calculus, we construct the required enhanced stochastic data and prove that their renormalized higher-order components vanish, while the remaining coordinates converge to the enhanced data of the dynamical $\Phi^4_3(\lambda)$ model. We further establish local well-posedness and stability of the associated deterministic solution map. Consequently, for well-prepared initial conditions, $u_\varepsilon$ converges in probability, as a random local solution germ, to the renormalized dynamical $\Phi^4_3(\lambda)$ solution. In particular, the cubic counterterm allows convergence strictly below the threshold $\alpha=1$ arising when only linear renormalization is allowed.
\end{abstract}

\keywords{dynamical $\Phi^4_3$ model, quintic perturbation, renormalization, paracontrolled distributions}

\maketitle

\section{Introduction}
The $\Phi^4_3(\lambda)$ model ($\lambda>0$) is formally given by 
\begin{align*}
\partial_t u=\Delta u-\lambda u^3+\xi,\ (t,x)\in\mathbb R_+\times \mathbb T^3,\no
\end{align*}
where $\xi$ is the space-time white noise on $\mathbb T^3$ with formal covariance $\mathbb E[\xi(s,x)\xi(t,y)]=\delta(t-s)\delta(y-x)$. Throughout the paper, $\mathbb T^3:=(\mathbb R/\mathbb Z)^3$ denotes the three-dimensional unit torus, endowed with the Lebesgue measure of total mass one. For $k\in\mathbb Z^3$, we set $e_k(x):=e^{2\pi i k\cdot x}$, so that $-\Delta e_k=4\pi^2|k|^2e_k$.

We define the spatially mollified noise by
\begin{align}\label{mollinoise}
\xi_\varepsilon(t,x)=\sum_{k\in\mZ^3} f(\varepsilon k)\widehat\xi(t,k)e_k(x).
\end{align}
Here $\hat\xi$ is the Fourier coefficient of $\xi$ in the spatial variable:
$$
\hat \varphi(k):=\int_{\mT^3}\varphi(x)e_{-k}(x)\dif x,
$$
and $f$ is assumed to be smooth, compactly supported, real-valued and even, with $f(0)=1$.

We aim to study the limit, as $\varepsilon\to0$, of solutions to the following approximating equation:
\begin{align*}
\partial_t u_\varepsilon=\Delta u_\varepsilon-\varepsilon^\alpha u_\varepsilon^5+\xi_\varepsilon+C_\varepsilon u_\varepsilon+\wt C_\varepsilon u_\varepsilon^3,\ u_\varepsilon(0,\cdot)=u_{\varepsilon,0},\no
\end{align*}
where $C_\varepsilon$ and $\wt C_\varepsilon$ are constants depending on $\varepsilon$ whose precise form will be determined in the sequel.

For the construction of the stationary stochastic objects, we extend the space-time white noise to $\mathbb R\times\mathbb T^3$. Let $X(t)$ be the solution of the linear equation
\begin{align*}
\partial_t X=(\Delta-1) X+\xi.\no
\end{align*}
In mild form, this reads
\begin{align*}
X(t)=\int_{-\infty}^t P_{t-s}\xi_s\dif s,\no
\end{align*}
where $P_t=e^{t(\Delta-1)}$ is the massive heat semigroup. Moreover, $X$ is a Gaussian process and $X\in C([0,T];\bC^{-\frac{1}{2}-}(\mathbb T^3))$, where $\bC$ is the standard H\"older--Besov space.

\subsection{A heuristic analysis}

The analysis of the role of the parameter $\alpha$ below is purely heuristic and is not intended as a rigorous proof.

Note that the main singular part of $u_\varepsilon$ is given by $X_\varepsilon$ (see \eqref{Xmollify} below for definition). Formally, the perturbation produces the term
\begin{align}\label{quanticwickextension}
\varepsilon^\alpha X_\varepsilon^5=\varepsilon^\alpha X_\varepsilon^{\diamond 5}+10\varepsilon^\alpha C_\varepsilon^{(1)}X_\varepsilon^{\diamond 3}+15\varepsilon^\alpha (C_\varepsilon^{(1)})^2X_\varepsilon,
\end{align}
where $X_\eps^{\diamond m}:=H_m(X_\eps;C_\varepsilon^{(1)})$, $H_m$ is the $m$-order Hermite polynomial and  
\begin{align}\label{defC1eps}
C_\varepsilon^{(1)}:=\mathbb E\big[X_\varepsilon(t,x)^2\big]=\frac12\sum_{k\in\mathbb Z^3}\frac{|f(\varepsilon k)|^2}{1+4\pi^2|k|^2}
\end{align}
is independent of $(t,x)$ since $X_\varepsilon$ is stationary and spatially homogeneous. In particular,
$$C_\varepsilon^{(1)}=\frac{\sigma_f^2}{\varepsilon}+O(1),\ \sigma_f^2=\frac{1}{8\pi^2}\int_{\mathbb R^3}\frac{|f(\theta)|^2}{|\theta|^2}\dif\theta.$$
The last term in \eqref{quanticwickextension} is linear and can be absorbed into the renormalization constant $C_\varepsilon$. Therefore the important term for the nonlinear structure is $10\varepsilon^\alpha C_\varepsilon^{(1)}X_\varepsilon^{\diamond 3}$. Since $\varepsilon^\alpha C_\varepsilon^{(1)}\sim\varepsilon^{\alpha-1}$, if one only allows the usual linear counterterm, $\alpha=1$ is the critical boundary. For $\alpha>1$, the fifth-order perturbation does not change the limiting cubic coefficient. For $\alpha=1$, it may still contribute a finite correction to the effective cubic nonlinearity; see \cite{EX22}. For $\alpha<1$, the perturbation is too large to be treated as a vanishing higher-order correction in the same scaling regime. 

However, the situation is different if one allows a cubic counterterm. In that case, the divergent cubic contribution generated by the quintic term can be compensated by a suitable choice of the cubic counterterm. Hence, one can go below this threshold; that is,  treat part of the regime $\alpha<1$. The precise admissible range is determined by the remaining higher order resonant products in the enhanced stochastic objects.  These terms are not removed by the cubic counterterm and have to be controlled by the small prefactor $\varepsilon^\alpha$.  We postpone the detailed power counting and the rigorous estimates to the subsequent sections. 

\subsection{Motivation and background}

The preceding heuristic analysis isolates the role of the exponent $\alpha$. We now place this question in the broader context of the dynamical $\Phi^4_3$ model and its approximations. The Euclidean $\Phi^4_3$ field is one of the fundamental non-Gaussian scalar field theories in three dimensions and a basic example of a renormalizable Euclidean quantum field theory. Its rigorous construction was a central achievement of constructive quantum field theory; see, for instance, \cite{BFS83,GJ87}. The stochastic quantization procedure proposed in \cite{PW81} associates with this field a Langevin dynamics whose invariant distribution is formally the Euclidean $\Phi^4_3$ measure. The model therefore provides a natural meeting point of stochastic analysis, statistical mechanics and Euclidean quantum field theory.

From the analytic viewpoint, the dynamical $\Phi^4_3$ equation is a canonical singular stochastic PDE. The linear stochastic convolution has spatial regularity slightly below $-\frac12$, so its cubic power cannot be defined by classical multiplication. Consequently, smooth approximations do not converge without suitable divergent counterterms. The Da Prato--Debussche method for two-dimensional stochastic quantization equations \cite{DPD03} provided an important precursor to the modern pathwise theories. Hairer subsequently constructed the renormalized dynamical $\Phi^4_3$ model using regularity structures \cite{Hairer14}. The paracontrolled calculus introduced in \cite{GIP15} gives an alternative approach and was applied directly to the three-dimensional stochastic quantization equation in \cite{CC18}. Further approaches and systematic renormalization procedures include the renormalization-group method of \cite{Kup16} and the algebraic BPHZ theory of \cite{BHZ19}. We refer also to \cite{CW17,MWX17} for expositions of the underlying stochastic objects and diagrammatic estimates. Spatial discretization and lattice approximation were studied in \cite{HM18,ZZ18}.

The local solution theory has been complemented by substantial progress on global dynamics and equilibrium properties. Mourrat and Weber established a coming-down-from-infinity estimate on the torus, yielding global control and a dynamical construction of a finite-volume invariant measure \cite{MW17}. Global solutions and analogous a priori estimates on Euclidean space were obtained in \cite{GH19}. Variational and PDE methods have also led to new constructions and structural results for the Euclidean $\Phi^4_3$ measure \cite{BG20,GH21}. In particular, these approaches connect the singular stochastic dynamics with non-Gaussianity, reflection positivity and the Dyson--Schwinger equations. Low-temperature phase separation and the associated decay of the spectral gap were established in \cite{CGW22}.

Several approximation results for the dynamical $\Phi^4_3$ model are relevant to the present setting. Hairer and Xu proved convergence for three-dimensional continuous phase-coexistence models near criticality \cite{HX18}. Shen and Xu extended it to suitable non-Gaussian random fields \cite{SX18}, while Furlan and Gubinelli treated nonlinearities with only finite smoothness by paracontrolled methods \cite{FG19}. The whole-space setting was considered in \cite{ZZ23WU}. Erhard and Xu subsequently treated odd polynomial potentials together with general higher-order smoothing mechanisms and showed that the effective cubic coupling may depend on the full microscopic potential and smoothing profile \cite{EX22}. At the microscopic level, the convergence of the two-dimensional Ising--Kac Glauber dynamics to the dynamical $\Phi^4_2$ model was proved in \cite{MW17IK}; its three-dimensional counterpart, converging to the dynamical $\Phi^4_3$ model, was established recently in \cite{GMW25}.

The work most directly related to the present setting is \cite{EX22}. Under the canonical weakly nonlinear scaling $-\varepsilon^{-\ff32}V'(\sqrt{\varepsilon},u)$, a sextic component of the microscopic potential produces a quintic term with coefficient of order $\varepsilon$. Its leading cubic contraction is then of order $\varepsilon C_\varepsilon^{(1)}\asymp1$ and contributes a finite correction to the effective cubic coupling. In the present paper, the quintic perturbation has the stronger size $\varepsilon^\alpha u_\varepsilon^5$ with $\alpha<1$. The corresponding cubic contraction is of order $\varepsilon^{\alpha-1}$ and therefore diverges, so this approximation is not covered by the canonical scaling and cannot be treated by a linear mass counterterm alone. We compensate this new relevant contribution by introducing a cubic counterterm. For $\alpha\in(\frac56,1)$, we prove that the remaining higher-order components vanish in the enhanced-data topology and that the corresponding well-prepared solutions converge, as local solution germs, to the renormalized dynamical $\Phi^4_3(\lambda)$ solution.

\subsection{Main results}

We first introduce the notation used in the statement below. We work on a fixed probability space $(\Omega,\mathcal F,\mathbb P)$ carrying the space-time white noise $\xi$. Let $X_\varepsilon(t):=\int_{-\infty}^tP_{t-s}\xi_\varepsilon(s)\dif s$ be the stationary stochastic convolution driven by $\xi_\varepsilon$. The deterministic contraction constants $C_\varepsilon^{(1)}$, $C_\varepsilon^{(2)}$, $D_{\varepsilon,\alpha}$ and $B_{\varepsilon,\alpha}$ are defined in \eqref{defC1eps}, \eqref{eq:C2-def-sec4}, \eqref{eq:D-def} and \eqref{eq:B5-explicit}, respectively. We also set
$$\sigma_f^2:=\frac{1}{8\pi^2}\int_{\mathbb R^3}\frac{|f(\theta)|^2}{|\theta|^2}\dif\theta.$$

For $S>0$, we denote by $\mathscr L_S^{\mathrm{stat}}$ the stationary canonical lift and by $\mathfrak X_{S,\alpha}$ the product space in which the components of the resulting stationary enhanced datum are measured. We denote by $\mathfrak X_{S,\alpha}^0$ the corresponding weighted finite-time space and by $\mathscr Q_0:\mathfrak X_{S,\alpha}\to\mathfrak X_{S,\alpha}^0$ the finite-time realization map. The precise definitions of these spaces, including their dependence on the auxiliary regularity parameters, are given in Section~\ref{subsec:admissible-enhanced-data}. We write $\mathbb X^\Phi$ for the standard stationary $\Phi^4_3$ enhancement and $\boldsymbol0$ for the collection of zero entries corresponding to the additional higher-order components.

We now fix the parameters used in the first main result. Fix $\alpha\in(\frac56,1)$ and $T>0$, and set $\widehat T:=T\wedge1$. Choose $\delta,\delta',\mathfrak d,\zeta,\kappa_0>0$ such that
$$
0<4\delta'<\delta<\frac15\bigl(3\alpha-\frac52\bigr),\ 1-\alpha+\frac32\delta<\mathfrak d<\frac12\bigl(\alpha-\frac12-2\delta\bigr),$$
$$0<\zeta<\delta',\ 2(\zeta+\kappa_0)<3\alpha-\frac52-5\delta.$$
The spaces appearing below are understood with this choice of auxiliary parameters. Let $(\lambda_\varepsilon)_{\varepsilon\in(0,1]}$ be a deterministic bounded family such that $\lambda_\varepsilon\to\lambda>0$ as $\varepsilon\downarrow0$.

With this notation and choice of parameters, our first result identifies the required counterterms and establishes the convergence of the enhanced data.

\begin{theorem}\label{thm:intro-enhanced-data}
Define
$$\wt C_\varepsilon:=10\varepsilon^\alpha C_\varepsilon^{(1)}-\lambda_\varepsilon,\ C_\varepsilon:=C_\varepsilon^{\mathrm{can}}=3\lambda_\varepsilon C_\varepsilon^{(1)}-15\varepsilon^\alpha\bigl(C_\varepsilon^{(1)}\bigr)^2-9\lambda_\varepsilon^2C_\varepsilon^{(2)}-9B_{\varepsilon,\alpha}-6D_{\varepsilon,\alpha},$$
then
\begin{equation*}
C_\varepsilon^{\mathrm{can}}=-15\sigma_f^4\varepsilon^{\alpha-2}+3\lambda\sigma_f^2\varepsilon^{-1}+o(\varepsilon^{-1}).
\end{equation*}

Set $\mu_\varepsilon:=\varepsilon^\alpha$, $\mathbf c_\varepsilon:=(C_\varepsilon^{(1)},C_\varepsilon^{(2)},D_{\varepsilon,\alpha},B_{\varepsilon,\alpha})$, and let
\begin{equation*}
\mathbb Z_\varepsilon:=\mathscr L_{\widehat T}^{\mathrm{stat}}\bigl(X_\varepsilon;\lambda_\varepsilon,\mu_\varepsilon,\mathbf c_\varepsilon\bigr),\ \mathbb Z:=(\mathbb X^\Phi,\lambda,0,\boldsymbol0).
\end{equation*}
See \eqref{DefZ_eps} and \eqref{DefZ} below for the precise definitions of $\mZ_\eps$ and $\mZ$, respectively. Then, for every $p\in[1,\infty)$,
\begin{equation*}
\mathbb Z_\varepsilon\to\mathbb Z\ \text{in }L^p\bigl(\Omega;\mathfrak X_{\widehat T,\alpha}\bigr).
\end{equation*}
Moreover, their finite-time realizations satisfy
\begin{equation*}
\mathbb Z_\varepsilon^0:=\mathscr Q_0\mathbb Z_\varepsilon\to\mathbb Z^0:=\mathscr Q_0\mathbb Z\ \text{in }L^p\bigl(\Omega;\mathfrak X_{\widehat T,\alpha}^0\bigr).
\end{equation*}
\end{theorem}

\br
The parameter regime used in Theorem~\ref{thm:intro-enhanced-data} is nonempty because $\alpha>\frac56$. Indeed,
\begin{align*}
\frac12\bigl(\alpha-\frac12-2\delta\bigr)-\bigl(1-\alpha+\frac32\delta\bigr)=\frac12\Bigl(3\alpha-\frac52-5\delta\Bigr)>0.
\end{align*}
Thus one may first choose $\delta$, then $\delta'<\ff\delta4$, next $\mathfrak d$ in the resulting nonempty interval, and finally $\zeta,\kappa_0>0$ sufficiently small. These are auxiliary parameters encoding the regularity losses and time weights in the enhanced-data spaces; they impose no additional assumption on the approximating model.
\er

For the solution-level statement, we retain the hypotheses and notation of Theorem~\ref{thm:intro-enhanced-data}. Set $\sigma:=\alpha-\frac12-2\delta$ and choose $\beta$ and $\gamma$ such that
$$2\mathfrak d<\beta<\sigma,\ 2-\alpha+2\delta<\gamma<\min\{2\alpha-\frac12-3\delta,\,\beta+\alpha-\delta\}.$$
The parameter conditions in Theorem~\ref{thm:intro-enhanced-data} guarantee that these intervals are nonempty. Here $\beta$ is the regularity of the initial remainder, while $\sigma$ and $\gamma$ are the regularity exponents used in the controlled fixed-point space.

Let $v_{\varepsilon,0}$ and $v_0$ be measurable $\bC^\beta$-valued random variables. We say that the initial conditions are well prepared if
$$u_{\varepsilon,0}=X_\varepsilon(0)+v_{\varepsilon,0},\ v_{\varepsilon,0}\longrightarrow v_0\ \text{in probability in }\bC^\beta.$$
We then set $u_0:=X(0)+v_0$.

Fix $\lambda_*\in(0,\lambda)$. Since $\lambda_\varepsilon\to\lambda$, one has $\lambda_\varepsilon\geq\lambda_*$ for all sufficiently small $\varepsilon$. For $R\geq1$, define
$$\Omega_{\varepsilon,R}:=
\big\{\|\mathbb Z_\varepsilon^0\|_{\mathfrak X_{\widehat T,\alpha}^0}+\|\mathbb Z^0\|_{\mathfrak X_{\widehat T,\alpha}^0}+\|v_{\varepsilon,0}\|_\beta+\|v_0\|_\beta\leq R\big\},$$
where $\|\cdot\|_\beta$ denotes the norm of $\bC^\beta$.

By a local solution germ at time $0$, we mean an equivalence class of pairs $(\tau,u)$, where $\tau>0$ and $u$ is a local solution on $[0,\tau]$. Two pairs $(\tau_1,u_1)$ and $(\tau_2,u_2)$ are equivalent if there exists $\tau\in(0,\tau_1\wedge\tau_2]$ such that $u_1=u_2$ on $[0,\tau]$.

The following local convergence theorem is our main result.

\begin{theorem}\label{thm:intro-local-convergence}
Under the above setting, for every $R\geq1$, there exist deterministic constants $K_{R,\lambda_*}<\infty$ and $0<T_R\leq\widehat T$, independent of $\varepsilon$, such that the following assertions hold on $\Omega_{\varepsilon,R}$: the approximating equation
\begin{align*}
\partial_tu_\varepsilon=\Delta u_\varepsilon-\varepsilon^\alpha u_\varepsilon^5+\xi_\varepsilon+C_\varepsilon^{\mathrm{can}}u_\varepsilon+\wt C_\varepsilon u_\varepsilon^3,\ u_\varepsilon(0)=u_{\varepsilon,0},
\end{align*}
has a pathwise local mild solution $u_\varepsilon=X_\varepsilon+v_\varepsilon$, where
$(v_\varepsilon,v_\varepsilon)\in\mathcal D_{T_R}^{\beta,\gamma}(\mathbb Z_\varepsilon;v_{\varepsilon,0})$ and $\|v_\varepsilon\|_{\mathcal L_{T_R}^{\beta,\sigma}}\leq K_{R,\lambda_*},\ \|v_\varepsilon^\sharp\|_{\mathcal L_{T_R}^{\beta,\gamma}}\leq1$ (See Definition \ref{def:controlled-distributions} below for the precise definition). The limiting enhanced datum gives a local solution $u=X+v$ of
\begin{align}\label{eq:intro-limit-equation}
\partial_tu=\Delta u-\lambda u^3+\xi,\ u(0)=u_0,
\end{align}
where \eqref{eq:intro-limit-equation} is understood in the renormalized sense. The fixed points are unique in the above bounded class and therefore determine unique local solution germs at time $0$.

For every $\eta>0$,
\begin{equation*}
\mathbb P\big(\Omega_{\varepsilon,R}\cap\big\{\|u_\varepsilon-u\|_{C([0,T_R];\bC^{-\frac12-\delta})}>\eta\big\}\big)\to0.
\end{equation*}
The localizing events are asymptotically exhaustive:
\begin{equation*}
\lim_{R\to\infty}\limsup_{\varepsilon\to0}\mathbb P\bigl(\Omega_{\varepsilon,R}^{\mathrm c}\bigr)=0.
\end{equation*}
Consequently, the local solution germs determined by $u_\varepsilon$ converge in probability to the local solution germ of the renormalized dynamical $\Phi^4_3(\lambda)$ equation. The lifetime $T_R$ may depend on the localization radius; no positive deterministic lifetime independent of $R$ is asserted.
\end{theorem}

\br 
Our results complement the convergence results in \cite{EX22,FG19,HX18,SX18,ZZ23WU}. The main difference is that we allow the quintic perturbation to have the stronger size $\varepsilon^\alpha$, with $\alpha\in(\frac56,1)$, rather than the canonical size $\varepsilon$. This produces a divergent correction to the cubic coupling and therefore requires an additional cubic counterterm. Although the approximation considered here is less general than that of \cite{EX22}, our result permits a stronger quintic perturbation and still yields local convergence to the standard dynamical $\Phi^4_3$ model.
\er

\subsection{Structure of the paper}

The remainder of the paper is organized as follows. Section \ref{sec02} collects the analytic preliminaries and introduces the enhanced-data framework used throughout the paper. After recalling the H\"older--Besov spaces, paraproduct calculus and the relevant commutator estimate, we separate the stationary Ornstein--Uhlenbeck component from the solution and perform the Wick--Taylor expansion of the mollified equation. This identifies the singular products underlying the paracontrolled ansatz and leads to the canonical choice of the mass counterterm. We then define the space of admissible enhanced data, construct its finite-time realization and formulate the convergence of the mollified enhanced data.

Section \ref{sec03} develops the deterministic solution theory for a general admissible enhanced datum. We introduce the space of controlled distributions, construct the renormalized mild map and establish the corresponding analytic bounds and stability estimates. A scaled contraction argument then yields local well-posedness, uniqueness of the initial solution germ and quantitative stability on bounded sets. Applying these results to the convergent enhanced data gives the localized convergence of the approximating solutions to the renormalized dynamical $\Phi^4_3(\lambda)$ solution.

Section \ref{sec:rough-distribution} provides the stochastic estimates required for the construction and convergence of the enhanced data. We first establish dyadic and time-regularity estimates for the Wick powers and their stationary heat convolutions. We then analyze the higher-order resonant products through Wiener-chaos expansions, showing that the nonlocal higher-order components vanish in the required topologies and isolating the local contributions represented by the constants $D_{\varepsilon,\alpha}$ and $B_{\varepsilon,\alpha}$. These estimates complete the proof of the enhanced-data convergence used in the preceding sections.

\section{Preliminaries}\label{sec02}

This section prepares the analytic and stochastic framework used in the solution theory. We first recall the H\"older--Besov spaces, the Bony decomposition, and the commutator estimate needed below. We then rewrite the mollified equation after separating its stationary Ornstein--Uhlenbeck component and identify the singular terms which determine the paracontrolled ansatz. This expansion also leads to the canonical choice of the mass counterterm. Finally, we introduce the rough distribution space and a finite-time realization.

\subsection{H\"older-Besov space and paraproduct calculus}

Let $\chi,\theta\in C_c^\infty$ be nonnegative radial functions such that:
\begin{enumerate}
    \item $\supp \chi$ is contained in a ball, and $\supp \theta$ is contained in an annulus;
    \item for every $\xi\in\mR^d$, $\chi(\xi)+\sum_{j\geq 0}\theta(2^{-j}\xi)=1;$
    \item $\supp\chi\cap\supp\bigl(\theta(2^{-j}\cdot)\bigr)=\varnothing\ \text{for }j\geq 1$ and $\supp\bigl(\theta(2^{-j}\cdot)\bigr)\cap\supp\bigl(\theta(2^{-i}\cdot)\bigr)=\varnothing\ \text{whenever }|i-j|>1$.
\end{enumerate}
For $u\in\sS'(\mR^d)$, define the Littlewood--Paley blocks by
$$\Delta_{-1}u =\bigl(\chi\,\hat u\bigr)^{\scriptscriptstyle\vee},\ \Delta_j u = \bigl(\theta(2^{-j}\cdot)\,\hat u\bigr)^{\scriptscriptstyle\vee},\ j\geq 0,$$
where $^{\scriptscriptstyle\vee}$ denotes the inverse Fourier transform.

\begin{definition}[Besov spaces]
Let $\alpha\in\mR$ and $1\leq p,q\leq\infty$. The Besov space $\bB^{\alpha}_{p,q}$ is defined by
$$\bB^{\alpha}_{p,q}:=\Big\{u\in\sS':\|u\|^q_{\bB^{\alpha}_{p,q}}:=\sum_{j\geq -1}2^{j q\alpha}\|\Delta_j u\|_{L^p}^q<\infty\Big\}.$$
As usual, when $q=\infty$, the $\ell^q$-norm is replaced by the supremum norm. In particular, we denote $ \bC^{\alpha}:=\bB^{\alpha}_{\infty,\infty}$ and write $\|u\|_{\a}:=\|u\|_{\bC^\a}$.
\end{definition}

Then we recall the Bony paraproduct decomposition, which allows us to control the product of functions in H\"{o}lder spaces under weaker regularity assumptions. For any $k\geq0$, define the cut-off low frequency operator $S_k$ as
\begin{align*}
    S_k f:=\sum_{j=-1}^{k-1}\Delta_jf\to f\ \text{in }\sS'\ \text{as }k\to\infty.
\end{align*} 
where we take $S_k=0$ for negative $k$ by convention. For $f,g\in\sS'$, we define the following paraproducts
\begin{align*}
    f\prec g:=\sum_{k\geq-1}S_{k-1} f\Delta_k g,\ f\circ g:=\sum_{k\geq-1}\Delta_k f\wt\Delta_k g,
\end{align*}
where $\wt\Delta_j:=\sum_{|i-j|\le1}\Delta_i$. The Bony decomposition of $fg$ is formally written as (cf. \cite{BCD11})
\begin{align*}
    fg=f\prec g+f\circ g+g\prec f.\no
\end{align*} 
The essential property of the Bony decomposition is that for some $N_0\in\mN$,
\begin{align*}
    \Delta_k(S_{j-1}f\Delta_j g)=0\ \text{for}\ |k-j|>N_0.\no
\end{align*}

We recall the following standard continuity estimates for the paraproduct and resonant product; see, for instance, \cite[Themrem 2.82 and 2.85]{BCD11}.

\bl\label{bony}
Let $\a, \b\in\mR$.
\begin{enumerate}
\item[\rm\text{(i)}] If $\a<0$, then there exists a constant $C=C(d,\a,\b)>0$ such that
\begin{align}
\| f\prec g\|_{\a+\b}\lesssim_C\|f\|_{\a}\|g\|_{\b}.\no
\end{align}
For $\a=0$, we have 
\begin{align}
\| f\prec g\|_{\b}\lesssim_C\|f\|_{L^\infty}\|g\|_{\b}.\no
\end{align}

\item[\rm\text{(ii)}] If $\a+\b>0$, then there exists a constant $C=C(d,\a,\b)>0$ such that
\begin{align}\no
\|f\circ g\|_{\a+\b}\lesssim_C\|f\|_{\a}\|g\|_{\b}.
\end{align}

\item[\rm\text{(iiI)}] If $\a+\b>0$, then there exists a constant $C=C(d,\a,\b)>0$ such that
\begin{align}
\|fg\|_{\a\wedge\b}\lesssim_C \|f\|_{\a}\|g\|_{\b}.\no
\end{align}
\end{enumerate}
\el

We shall also use the following commutator estimate from \cite[Lemma 2.4]{GIP15}.

\bl\label{commutatores}
Let $\a$, $\b$, $\g\in\mR$ such that $\a<1$, $\a+\b+\g>0$ and $\b+\g<0$. Then
$$\mathfrak R(f,g,h):=(f\prec g)\circ h-f(g\circ h)$$
is well-defined for $f\in\bC^\a$, $g\in\bC^\b$ and $h\in\bC^\g$. Moreover, 
$$\|\mathfrak R(f,g,h)\|_{\a+\b+\g}\lesssim\|f\|_\a\|g\|_\b\|h\|_\g.$$
\el

\subsection{Paracontrolled distribution analysis}

We consider
\begin{align*}
\partial_t u_\varepsilon=\Delta u_\varepsilon-\varepsilon^\alpha u_\varepsilon^5+\xi_\varepsilon+C_\varepsilon u_\varepsilon+\wt C_\varepsilon u_\varepsilon^3,\ (t,x)\in \mathbb R_+\times \mathbb T^3,\no
\end{align*}
where $\a\in(\ff56,1)$. We use the zero-mean stationary Ornstein–Uhlenbeck process $X$ in the construction of the stochastic objects. Then
\begin{align*}
\mE\big[\widehat X_t(k)\widehat X_s(\ell)\big]=\mathbf 1_{\{k+\ell=0\}}\frac{e^{-(1+4\pi^2|k|^2)|t-s|}}{2(1+4\pi^2|k|^2)}.
\end{align*}
Define its spatial mollification by
\begin{align}\label{Xmollify}
X_\eps(t)=\sum_{k\in \mZ^3}f(\eps k)\widehat X(t,k)e_k.
\end{align}
For $m\in\mathbb N_+$, we define the Wick-product notation
\begin{align*}
X_\varepsilon^{\diamond m}:=H_m(X_\varepsilon; \mathbb E[X_\varepsilon^2]):=H_m(X_\varepsilon;C_\varepsilon^{(1)}),\no
\end{align*}
where $H_m$ is the $m$-th Hermite polynomial. In particular, we have 
\begin{align*}
X_\varepsilon^3=X_\varepsilon^{\diamond3}+3C_\varepsilon^{(1)} X_\varepsilon,\ X_\varepsilon^5=X_\varepsilon^{\diamond5}+10C_\varepsilon^{(1)} X_\varepsilon^{\diamond3}+15(C_\varepsilon^{(1)})^2 X_\varepsilon.\no
\end{align*}

We denote by
$$I_0(F)(t):=\int_0^t P_{t-s}F(s)\,\dif s$$
the finite-time integration operator. By the Duhamel formula, $I_0(F)$ is the contribution of the forcing term $F$ to the mild solution on $[0,t]$ with zero initial condition. We also define the stationary integration operator
$$\cI(F)(t):=\int_{-\infty}^t P_{t-s}F(s)\,ds$$
used in Section~\ref{subsec:admissible-enhanced-data} to construct the stationary stochastic coordinates entering the enhanced data $\mathbb Z_\varepsilon$ and $\mathbb Z$; see  \eqref{DefZ_eps} and \eqref{DefZ} below. For stationary stochastic objects, $I_0(F)(t)=\cI(F)(t)-P_t\cI(F)(0)$. The stationary enhanced datum is converted into its finite-time realization through the map $\mathscr Q_0$ defined in Section~\ref{subsec:admissible-enhanced-data}.

Then the equation for $u_\varepsilon:=X_\varepsilon+v_\varepsilon$ is written in the
mild form
\begin{align*}
v_\varepsilon=P_t v_{\eps,0}-I_0(\varepsilon^\alpha u_\varepsilon^5-(C_\varepsilon +1)u_\varepsilon-\wt C_\varepsilon u_\varepsilon^3),\ v_{\eps,0}=u_{\varepsilon,0}-X_\varepsilon(0).\no
\end{align*}
Denote $\lambda_\eps:=10\eps^\a C_\eps^{(1)}-\wt C_\eps\to\lambda$. Then the local Wick counterterm generated by $\eps^\a u^5-\wt C_\eps u^3$ is
$$
A_{\eps,\a}^{\rm loc}:=15\eps^\a\big(C_\eps^{(1)}\big)^2-3\wt C_\eps C_\eps^{(1)}=3\lambda_\eps C_\eps^{(1)}-15\eps^\a\big(C_\eps^{(1)}\big)^2 .
$$
A direct Wick--Taylor expansion gives the exact identity
$$
\begin{aligned}
\eps^\a (X_\eps+v)^5-\wt C_\eps (X_\eps+v)^3-A_{\eps,\a}^{\rm loc}(X_\eps+v)=&\lambda_\eps X_\eps^{\diamond3}+\eps^\a X_\eps^{\diamond5}+3\lambda_\eps vX_\eps^{\diamond2}\\
&+5\eps^\a vX_\eps^{\diamond4}+3\lambda_\eps v^2X_\eps+10\eps^\a v^2X_\eps^{\diamond3}\\
&+\lambda_\eps v^3+10\eps^\a v^3X_\eps^{\diamond2}+5\eps^\a v^4X_\eps+\eps^\a v^5 .
\end{aligned}
$$
Therefore $v_\eps$ solves
$$
\begin{aligned}
v_\eps={}&P_t v_{\eps,0}-I_0\Big(\lambda_\eps X_\eps^{\diamond3}+\eps^\a X_\eps^{\diamond5}+3\lambda_\eps v_\eps X_\eps^{\diamond2}+5\eps^\a v_\eps X_\eps^{\diamond4}+3\lambda_\eps v_\eps^2X_\eps\\
&+10\eps^\a v_\eps^2X_\eps^{\diamond3}+\lambda_\eps v_\eps^3+10\eps^\a v_\eps^3X_\eps^{\diamond2}+5\eps^\a v_\eps^4X_\eps+\eps^\a v_\eps^5-\big(C_\eps+1-A_{\eps,\a}^{\rm loc}\big)(X_\eps+v_\eps)\Big).
\end{aligned}
$$

The leading singular terms suggest the following one-level paracontrolled ansatz:
$$
\begin{aligned}
v_\eps=P_t v_{\eps,0}-\underbrace{\eps^\a I_0\big(X_\eps^{\diamond5}\big)}_{\bC^{\a-\ff12-}}-\underbrace{\lambda_\eps I_0\big(X_\eps^{\diamond3}\big)}_{\bC^{\ff12-}}-\underbrace{5\eps^\a I_0\big(v_\eps'\prec X_\eps^{\diamond4}\big)}_{\bC^{\a-}}-\underbrace{3\lambda_\eps I_0\big(v_\eps'\prec X_\eps^{\diamond2}\big)}_{\bC^{1-}}+\underbrace{v_\eps^\sharp}_{\bC^{2\a-\ff12-}},\ v_\eps'=v_\eps.
\end{aligned}
$$
This is the analogue of the Catellier–Chouk ansatz (see \cite{CC18}). The term $\eps^\a I_0(X_\eps^{\diamond5})$ is kept explicitly in the ansatz, since it does not vanish in the $\bC^{\ff12-}$ topology, hence cannot be absorbed into the smoother remainder $v_\eps^\sharp$.

\subsection{Mass counterterm}

We keep the usual scalar counterterm from the $\Phi^4_3$ theory.  With
$a_k:=1+4\pi^2|k|^2$, set
\begin{align}\label{eq:C2-def-sec4}
C_\eps^{(2)}:=\frac12\sum_{k_1,k_2\in\mZ^3}\frac{|f(\eps k_1)|^2|f(\eps k_2)|^2}{a_{k_1}a_{k_2}\big(a_{k_1+k_2}+a_{k_1}+a_{k_2}\big)}.
\end{align}
An explicit calculation yields
$$C_\eps^{(2)}=\frac{1}{96\pi^2}\log\frac1\eps+O(1).$$
The two extra local constants produced by the quintic terms are exactly the constants
$D_{\eps,\a}$ and $B_{\eps,\a}$ defined in \eqref{eq:D-def} and \eqref{eq:B5-explicit} below. Notice that, in the present range $\a<1$, these constants does not converge as $\eps\downarrow0$; rather, $D_{\eps,\a}=O(\eps^{\a-1}),\ B_{\eps,\a}=O(\eps^{2\a-2})$.

\begin{proposition}\label{prop:finite-counterterm}
Assume $\l_\eps\to\l$ with $\sup_{\eps\in(0,1]}|\lambda_\eps|<\infty$. The scalar mass counterterm compatible with the renormalizations of Section~\ref{sec:rough-distribution} is
\begin{equation}\label{eq:canonical-mass-counterterm}
C_\eps^{\rm can}:=A_{\eps,\a}^{\rm loc}-9\lambda_\eps^2C_\eps^{(2)}-9B_{\eps,\a}-6D_{\eps,\a}.
\end{equation}
Moreover, if we set
$$
\sigma_f^2:=\frac{1}{8\pi^2}\int_{\mR^3}\frac{|f(\theta)|^2}{|\theta|^2}\,\dif\theta,
$$
then
\begin{equation}\label{eq:Ceps-leading-subcritical}
C_\eps^{\rm can}=-15\sigma_f^4\eps^{\a-2}+3\lambda\sigma_f^2\eps^{-1}+o(\eps^{-1}).
\end{equation}
The lower order divergent terms are still contained in the exact expression \eqref{eq:canonical-mass-counterterm}, namely in $C_\eps^{(2)}$, $B_{\eps,\a}$, $D_{\eps,\a}$ and the subleading part of $A_{\eps,\a}^{\rm loc}$.
\end{proposition}

\begin{proof}
The local Wick expansion above gives the scalar linear term $A_{\eps,\a}^{\rm loc}$. The usual second-order $\Phi^4_3$ resonance, arising from $\cI(X_\eps^{\diamond3})\circ X_\eps^{\diamond2}$, contributes the mass correction $9\lambda_\eps^2C_\eps^{(2)}$, while the two new local pieces isolated in Section~\ref{sec:rough-distribution} contribute $9B_{\eps,\a}$ and $6D_{\eps,\a}$.  Subtracting these scalar terms gives \eqref{eq:canonical-mass-counterterm}.  The additional $+1$ coming from the massive semigroup $P_t=e^{t(\Delta-1)}$ is a regular deterministic term and is not included in $C_\eps^{\rm can}$.

Finally,
$$
C_\eps^{(1)}=\frac{\sigma_f^2}{\eps}+O(1),
$$
so, since $\lambda_\eps\to\lambda$,
$$
A_{\eps,\a}^{\rm loc}=3\lambda_\eps C_\eps^{(1)}-15\eps^\a\big(C_\eps^{(1)}\big)^2=-15\sigma_f^4\eps^{\a-2}+\frac{3\lambda\sigma_f^2}{\eps}+o(\eps^{-1}).
$$
The remaining terms in \eqref{eq:canonical-mass-counterterm} are all $o(\eps^{-1})$: indeed $C_\eps^{(2)}=O(\log(\ff1\eps))$, $D_{\eps,\a}=O(\eps^{\a-1})$, and $B_{\eps,\a}=O(\eps^{2\a-2})$.  This proves \eqref{eq:Ceps-leading-subcritical}.
\end{proof}

\subsection{Rough distribution and finite-time realizations}\label{subsec:admissible-enhanced-data}

We now collect the stochastic objects required by the paracontrolled expansion into a deterministic enhanced datum.

For $\rho\in\mR$ and $\zeta\in(0,1)$, let $C_T^{\zeta,\rho}$ be the space equipped with the norm
\begin{equation*}
\|F\|_{C_T^{\zeta,\rho}}:=\sup_{0\leq t\leq T}\|F_t\|_{\bC^\rho}+\sup_{0\leq s<t\leq T}\ff{\|F_t-F_s\|_{\bC^\rho}}{|t-s|^{\zeta}}.
\end{equation*}

Fix $\alpha\in(\ff56,1)$ and choose $\delta>0$ such that
\begin{equation}\label{eq:model-parameter-choice}
0<\delta<\frac15\Big(3\alpha-\frac52\Big).
\end{equation}
Choose a parabolic time exponent $\mathfrak d$ and an auxiliary exponent $\delta'$ so that
\begin{equation}\label{eq:model-parabolic-time-choice}
1-\alpha+\frac32\delta<\mathfrak d<\frac12\Big(\alpha-\frac12-2\delta\Big),\ 0<4\delta'<\delta.
\end{equation}
The interval for $\mathfrak d$ is not empty in view of \eqref{eq:model-parameter-choice}.  The exponent $\mathfrak d$ is the time regularity which will be used in the controlled derivative $v'$. 

For $\rho\in\mR$, define
\begin{align*}
\|F\|_{\mathcal E_T^{\mathfrak d;\rho}}:=\sup_{0\leq t\leq T}\|F_t\|_{\bC^\rho}+\sup_{0\leq s<t\leq T}\frac{\|F_t-F_s\|_{\bC^{\rho-2\mathfrak d}}}{|t-s|^{\mathfrak d}}.
\end{align*}
We let $\mathcal E_T^{\mathfrak d;\rho}$ be the completion of $C^\infty([0,T]\times\mT^3)$ with respect to this norm, and let $\mathcal E_{T,0}^{\mathfrak d;\rho}$ be its closed subspace consisting of elements with zero trace at $t=0$. Thus the fixed-time regularity is $\rho$, while the $\mathfrak d$-H\"older time increment is measured in the parabolically lower space $\bC^{\rho-2\mathfrak d}$.

We introduce the following collection of model distributions, which slightly enlarges the usual collection associated with the $\Phi^4_3$ model:
\begin{align*}
\mathfrak X_T^\Phi:=C_T^{\delta',-\ff12-\delta}\times C_T^{\delta',-1-\delta}\times\mathcal E_T^{\mathfrak d;1-\delta}\times\mathcal E_T^{\mathfrak d;\ff12-\delta}\times C_T^{\delta',-\delta}\times C_T^{\delta',-\delta}\times C_T^{\delta',-\ff12-\delta}.
\end{align*}
Except for the third coordinate, these are the usual stationary $\Phi^4_3$ coordinates constructed by the standard Wiener-chaos estimates; see, for example, \cite{CC18}. The additional third coordinate records $\cI(X^{\diamond2})$.  It is needed to recover the finite-time realization of the resonant product involving $I_0(X^{\diamond2})$. Its convergence in $\mathcal E_T^{\mathfrak d;1-\delta}$, namely, uniformly in time in $\bC^{1-\delta}$ with $\mathfrak d$-H\"older time increments in $\bC^{1-\delta-2\mathfrak d}$, is proved by the same integrated second-chaos estimate as for the other second-chaos components of the enhancement.

We define the space of higher-order components by
\begin{align}\label{eq:stationary-higher-sector}
\begin{split}
\mathfrak V_{T,\alpha}:=&C_T^{\delta',\alpha-1-\delta}\times C_T^{\delta',\alpha-\ff32-\delta}\times C_T^{\delta',\alpha-2-\delta}\times\mathcal E_T^{\mathfrak d;\alpha+1-\delta}\times\mathcal E_T^{\mathfrak d;\alpha+\ff12-\delta}\times\mathcal E_T^{\mathfrak d;\alpha-\delta}\\
&\times\mathcal E_T^{\mathfrak d;\alpha-\ff12-\delta}\times C_T^{\delta',2\alpha-2-\delta}\times C_T^{\delta',\alpha-\ff32-\delta}\times C_T^{\delta',\alpha-1-\delta}\times C_T^{\delta',\alpha-1-\delta}\\
&\times C_T^{\delta',\alpha-1-\delta}\times C_T^{\delta',\alpha-1-\delta}\times C_T^{\delta',2\alpha-2-\delta}\times C_T^{\delta',\alpha-\ff32-\delta}\times C_T^{\delta',2\alpha-\ff52-\delta}.
\end{split}
\end{align}
We set
\begin{equation*}
\mathfrak X_{T,\alpha}:=\mathfrak X_T^\Phi\times\mR^2\times\mathfrak V_{T,\alpha},
\end{equation*}
and equip it with the product norm.

We first define the canonical lift on a smooth core. 
Let $X\in C_c^\infty(\mR\times\mT^3)$, $(\lambda,\mu)\in\mR^2$, and $\mathbf c=(c_1,c_2,d,b)\in\mR^4$. Define
\begin{equation*}
X^{\diamond m}:=H_m(X;c_1),\ m=2,3,4,5.
\end{equation*}
The standard part of the smooth canonical lift is defined by
\begin{equation*}
\mathbb X^\Phi(X;\mathbf c):=\big(X,X^{\diamond2},\cI(X^{\diamond2}),\cI(X^{\diamond3}),\cI(X^{\diamond3})\circ X,\cI(X^{\diamond2})\circ X^{\diamond2}-c_2,\cI(X^{\diamond3})\circ X^{\diamond2}-3c_2X\big).
\end{equation*}
The first higher-order block is defined by
\begin{equation*}
\mathbb V^1(X;\mu,\mathbf c):=\big(\mu X^{\diamond2},\mu X^{\diamond3},\mu X^{\diamond4},\mu\cI(X^{\diamond2}),\mu\cI(X^{\diamond3}),\mu\cI(X^{\diamond4}),\mu\cI(X^{\diamond5})\big),
\end{equation*}
and the second one is defined by
\begin{equation}\label{eq:second-higher-order-block}
\begin{aligned}
\mathbb V^2(X;\lambda,\mu,\mathbf c):=\big(&\mu^2\cI(X^{\diamond5})\circ X^{\diamond3},\lambda\mu\cI(X^{\diamond5})\circ X^{\diamond2},\lambda\mu\cI(X^{\diamond5})\circ X,\mu\cI(X^{\diamond2})\circ X^{\diamond4},\\
&\mu\cI(X^{\diamond4})\circ X^{\diamond2},\ff{10}{3}\lambda\mu\cI(X^{\diamond3})\circ X^{\diamond3}-d,\ff{25}{9}\mu^2\cI(X^{\diamond4})\circ X^{\diamond4}-b,\\
&\ff53\lambda\mu\cI(X^{\diamond3})\circ X^{\diamond4}-2dX,\ff53\mu^2\cI(X^{\diamond5})\circ X^{\diamond4}-3bX\big).
\end{aligned}
\end{equation}
We write $\mathbb V(X;\lambda,\mu,\mathbf c):=(\mathbb V^1,\mathbb V^2)$ and define the smooth canonical lift by
\begin{equation*}
\mathscr L_T^{\rm stat}(X;\lambda,\mu,\mathbf c):=\big(\mathbb X^\Phi(X;\mathbf c),\lambda,\mu,\mathbb V(X;\lambda,\mu,\mathbf c)\big).
\end{equation*}
The superscript ``stat'' here indicates that the coordinates are constructed using the stationary integration operator $\mathcal I$; it does not impose a probabilistic stationarity assumption on the deterministic smooth field $X$.

\begin{definition}\label{def:admissible-enhanced-data}
Let
\begin{equation*}
\mathscr X_{T,\alpha}^{\infty}:=\left\{\mathscr L_T^{\rm stat}(X;\lambda,\mu,\mathbf c):X\in C_c^\infty(\mR\times\mT^3),(\lambda,\mu)\in\mR^2,\mathbf c\in\mR^4\right\}.
\end{equation*}
The space of admissible enhanced data is defined by
\begin{equation*}
\mathscr X_{T,\alpha}:=\overline{\mathscr X_{T,\alpha}^{\infty}}^{\,\mathfrak X_{T,\alpha}}.
\end{equation*}
We equip $\mathscr X_{T,\alpha}$ with the metric inherited from $\mathfrak X_{T,\alpha}$,
\begin{equation*}
d_{\mathscr X_{T,\alpha}}(\mathbb Z,\overline{\mathbb Z}):=\|\mathbb Z-\overline{\mathbb Z}\|_{\mathfrak X_{T,\alpha}}.
\end{equation*}
\end{definition}

We write $\mathbb Z=(\mathbb X^\Phi,\lambda,\mu,\mathbb V)\in\mathscr X_{T,\alpha}$.
For a general admissible datum, all composite entries are joint limits of smooth canonical lifts and are not reconstructed from the first coordinate $X$ and the scalar $\mu$. We construct stochastic lifts only for $\mu_\varepsilon=\varepsilon^\alpha\to0$, and make no such claim for a fixed $\mu\neq0$.

With
\begin{equation*}
\mathbf c_\varepsilon:=\bigl(C_\varepsilon^{(1)},C_\varepsilon^{(2)},D_{\varepsilon,\alpha},B_{\varepsilon,\alpha}\bigr),
\end{equation*}
we use
\begin{equation}\label{DefZ_eps}
\mathbb Z_\varepsilon:=\mathscr L_T^{\rm stat}\bigl(X_\varepsilon;\lambda_\varepsilon,\varepsilon^\alpha,\mathbf c_\varepsilon\bigr)
\end{equation}
as shorthand for the limit in $\mathfrak X_{T,\alpha}$ of the canonical lifts obtained by multiplying $X_\varepsilon$ by smooth cutoffs on expanding time intervals and then mollifying in time. The exponential decay of the massive semigroup and the classical nature of all spatial products for fixed $\varepsilon$ imply that this limit exists. Hence $\mathbb Z_\varepsilon\in\mathscr X_{T,\alpha}$ almost surely.

The limiting datum is
\begin{equation}\label{DefZ}
\mathbb Z=(\mathbb X^\Phi,\lambda,0,\boldsymbol0),
\end{equation}
where
\begin{equation*}
\mathbb X^\Phi:=\big(X,X^{\diamond2},\cI(X^{\diamond2}),\cI(X^{\diamond3}),\cI(X^{\diamond3})\circ X,(\cI(X^{\diamond2})\circ X^{\diamond2})^\diamond,(\cI(X^{\diamond3})\circ X^{\diamond2})^\diamond\big),
\end{equation*}
and $\boldsymbol0$ denotes the limit of the scaled higher-order coordinates.

The estimates in Section~\ref{sec:rough-distribution}, together with the standard construction of the $\Phi^4_3$ model distributions, yield the following convergence of the admissible enhanced data.

\begin{theorem}\label{thm:enhanced-data-convergence}
Assume that $\lambda_\varepsilon\to\lambda$.  For every $p\in[1,\infty)$ and every choice of parameters satisfying \eqref{eq:model-parameter-choice}--\eqref{eq:model-parabolic-time-choice},
\begin{equation*}
\mathbb Z_\varepsilon\to\mathbb Z\ \text{in}\ L^p(\Omega;\mathfrak X_{T,\alpha}).
\end{equation*}
In particular, $\mathbb Z\in\mathscr X_{T,\alpha}$ almost surely.
\end{theorem}

\begin{proof}

The convergence of the usual standard $\Phi^4_3$ coordinates follows from the standard Wiener-chaos construction; see \cite[Section~4]{CC18}. 

Set $\bar\delta=\delta-2\delta'$. Then $\bar\delta>0$ and, since $4\delta'<\delta$, whenever an estimate in Section~\ref{sec:rough-distribution} requires an auxiliary loss parameter larger than $2\delta'$, it can be chosen strictly below $\ff\delta2$. Apply Proposition~\ref{prop:eps-wick-vanishing} and Lemmas~\ref{lem:X5-X3-vanishing}--\ref{lem:X5-X4-B-localization} below with $\bar\delta$ in place of the spatial loss parameter denoted by $\delta$ in Section~\ref{sec:rough-distribution}. Then all the resulting spatial indices agree exactly with those in \eqref{eq:stationary-higher-sector}. This proves convergence of all components in $\mathfrak V_{T,\alpha}$ except $\varepsilon^\alpha\cI(X_\varepsilon^{\diamond m})$, $m=2,3,4,5$, whose convergence in the corresponding parabolic path spaces is verified below.

For $m\in\{2,3,4,5\}$, set $F_{\varepsilon,m}:=\varepsilon^\alpha\cI(X_\varepsilon^{\diamond m})$ and $\rho_m:=\alpha+2-\frac m2-\delta$.
The fixed-time part of $\|F_{\varepsilon,m}\|_{\mathcal E_T^{\mathfrak d;\rho_m}}$ converges to zero by Proposition~\ref{prop:eps-wick-vanishing} below. For the time-increment part, take $\zeta=\mathfrak d$ in \eqref{wick-dyadic-02}.  Since $\mathfrak d<\frac14$, this choice is admissible.  For sufficiently small $\eta>0$, dyadic summation gives
\begin{align*}
\left\|\sup_{0\leq s<t\leq T}\frac{\|F_{\varepsilon,m}(t)-F_{\varepsilon,m}(s)\|_{\bC^{\rho_m-2\mathfrak d}}}{|t-s|^{\mathfrak d}}\right\|_{L^p(\Omega)}\lesssim\varepsilon^{\alpha-\eta}\sup_{1\leq K\lesssim\varepsilon^{-1}}K^{\rho_m-2\mathfrak d+\frac m2-2+2\mathfrak d+\eta}\lesssim\varepsilon^{\delta-2\eta}\to0.
\end{align*}
Thus for $m=2,3,4,5$,
\begin{equation*}
\varepsilon^\alpha\cI(X_\varepsilon^{\diamond m})\to0\ \text{in}\ L^p\bigl(\Omega;\mathcal E_T^{\mathfrak d;\rho_m}\bigr).
\end{equation*}
Similarly, for $m=2,3$ we have
\begin{equation*}
\cI(X_\varepsilon^{\diamond m})\to\cI(X^{\diamond m})\ \text{in}\ L^p\bigl(\Omega;\mathcal E_T^{\mathfrak d;2-\frac m2-\delta}\bigr),
\end{equation*}
This proves convergence in their prescribed $\mathcal E_T^{\mathfrak d;\rho}$ spaces.
\end{proof}

We next construct the finite-time realization used in the mild equation. Its integrated coordinates vanish at the initial time, while some resonant coordinates need not converge in their prescribed spatial spaces as $t\downarrow0$. This motivates the weighted path spaces introduced below.

Choose $\zeta,\kappa_0>0$ such that
\begin{equation}\label{eq:finite-time-weight-choice}
\zeta<\delta',\ 2(\zeta+\kappa_0)<3\alpha-\frac52-5\delta.
\end{equation}
For $\omega>0$ and $\rho\in\mR$, set
\begin{align}\label{eq:weighted-finite-time-norm}
\|F\|_{\mathcal W_T^{\omega,\zeta;\rho}}:=\sup_{0<t\leq T}t^\omega\|F_t\|_{\bC^\rho}+\sup_{0<s<t\leq T}s^\omega\frac{\|F_t-F_s\|_{\bC^\rho}}{|t-s|^\zeta}.
\end{align}
We define $\mathcal W_T^{\omega,\zeta;\rho}$ as the completion of $C^\infty([0,T]\times\mT^3)$ with respect to \eqref{eq:weighted-finite-time-norm}.  Elements of this space are identified with paths on $(0,T]$; no value at $t=0$ is prescribed.  Since
$0<\zeta<\delta'$, every smooth path satisfies
\begin{equation}\label{eq:regular-path-weighted-embedding}
\|F\|_{\mathcal W_T^{\omega,\zeta;\rho}}\leq T^\omega(1+T^{\delta'-\zeta})\|F\|_{C_T^{\delta',\rho}}.
\end{equation}
Consequently, the identity map extends continuously from the closure of smooth paths in $C_T^{\delta',\rho}$ into $\mathcal W_T^{\omega,\zeta;\rho}$.

For a smooth path $F$, set
\begin{equation*}
(Q_0F)_t:=F_t-P_tF_0.
\end{equation*}

\begin{lemma}\label{lem:finite-time-primitives}
For every $\rho\in\mR$ and every smooth path $F$, one has
\begin{equation}\label{eq:Q0-primitive-estimate}
\|Q_0F\|_{\mathcal E_{T,0}^{\mathfrak d;\rho}}\lesssim\|F\|_{\mathcal E_T^{\mathfrak d;\rho}}.
\end{equation}
Consequently, $Q_0$ extends uniquely to a bounded linear map
\begin{equation*}
Q_0:\mathcal E_T^{\mathfrak d;\rho}\to\mathcal E_{T,0}^{\mathfrak d;\rho}.
\end{equation*}
\end{lemma}

\begin{proof}
The uniform estimate follows from the boundedness of $P_t$ on $\bC^\rho$. Moreover,
\begin{equation*}
(Q_0F)_t-(Q_0F)_s=F_t-F_s-(P_t-P_s)F_0.
\end{equation*}
The first difference is controlled by the $\mathcal E_T^{\mathfrak d;\rho}$ norm, while
\begin{equation*}
\|(P_t-P_s)F_0\|_{\bC^{\rho-2\mathfrak d}}\lesssim|t-s|^{\mathfrak d}\|F_0\|_{\bC^\rho}.
\end{equation*}
This proves \eqref{eq:Q0-primitive-estimate}. The extension follows by completion.
\end{proof}

The next estimate controls the initial-time corrections in the resonant coordinates.

\begin{lemma}\label{lem:initial-layer-resonance}
Let $r_1,r_2,\rho\in\mR$, let
$A\in\mathcal E_T^{\mathfrak d;r_1}$ and $B\in C_T^{\delta',r_2}$ be smooth paths, and set
\begin{equation}\label{eq:initial-layer-threshold}
\theta_*(r_1,r_2,\rho):=\max\Big\{0,-\frac{r_1+r_2}{2},\frac{\rho-r_1-r_2}{2}\Big\}.
\end{equation}
For $t>0$, define
\begin{equation*}
\Gamma(A,B)_t:=(P_tA_0)\circ B_t.
\end{equation*}
Assume that
\begin{equation}\label{eq:initial-layer-weight-condition}
\omega>\zeta+\theta_*(r_1,r_2,\rho).
\end{equation}
Then $\Gamma(A,B)\in\mathcal W_T^{\omega,\zeta;\rho}$.  More precisely, for every $\theta$ satisfying
\begin{equation*}
\theta_*(r_1,r_2,\rho)<\theta<\omega-\zeta,
\end{equation*}
one has
\begin{equation}\label{eq:initial-layer-estimate}
\|\Gamma(A,B)\|_{\mathcal W_T^{\omega,\zeta;\rho}}\lesssim T^{\omega-\theta-\zeta}(1+T^\zeta+T^{\delta'})\|A\|_{\mathcal E_T^{\mathfrak d;r_1}}\|B\|_{C_T^{\delta',r_2}}.
\end{equation}
For smooth $\bar A\in\mathcal E_T^{\mathfrak d;r_1}$ and
$\bar B\in C_T^{\delta',r_2}$, the corresponding difference estimate is
\begin{align}\label{eq:initial-layer-difference}
\begin{split}
&\|\Gamma(A,B)-\Gamma(\bar A,\bar B)\|_{\mathcal W_T^{\omega,\zeta;\rho}}\\
&\qquad\lesssim T^{\omega-\theta-\zeta}(1+T^\zeta+T^{\delta'})\big(\|A-\bar A\|_{\mathcal E_T^{\mathfrak d;r_1}}\|B\|_{C_T^{\delta',r_2}}+\|\bar A\|_{\mathcal E_T^{\mathfrak d;r_1}}\|B-\bar B\|_{C_T^{\delta',r_2}}\big).
\end{split}
\end{align}
\end{lemma}

\begin{proof}
The choice of $\theta$ implies $r_1+r_2+2\theta>\rho\vee0$. The heat-flow and resonant-product estimates therefore give
\begin{equation*}
\|\Gamma(A,B)_t\|_{\bC^\rho}\lesssim t^{-\theta}\|A_0\|_{\bC^{r_1}}\|B_t\|_{\bC^{r_2}}.
\end{equation*}
For $0<s<t\leq T$, write
\begin{align*}
\Gamma(A,B)_t-\Gamma(A,B)_s=\bigl((P_t-P_s)A_0\bigr)\circ B_t+(P_sA_0)\circ(B_t-B_s).
\end{align*}
By the standard Schauder estimates for the massive heat semigroup $P_t=e^{t(\Delta-1)}$ (see \cite[Lemmas~A.6 and A.7]{EX22}) and the identity $P_t-P_s=(P_{t-s}-\mathrm{Id})P_s$, we have
\begin{align*}
\|(P_t-P_s)A_0\|_{\bC^{r_1+2\theta}}\lesssim|t-s|^\zeta s^{-\theta-\zeta}\|A_0\|_{\bC^{r_1}},\ \|P_sA_0\|_{\bC^{r_1+2\theta}}\lesssim s^{-\theta}\|A_0\|_{\bC^{r_1}}.
\end{align*}
Consequently,
\begin{align*}
\|\Gamma(A,B)_t-\Gamma(A,B)_s\|_{\bC^\rho}\lesssim(|t-s|^\zeta s^{-\theta-\zeta}+|t-s|^{\delta'}s^{-\theta})\|A\|_{\mathcal E_T^{\mathfrak d;r_1}}\|B\|_{C_T^{\delta',r_2}}.
\end{align*}
Multiplying by $s^\omega|t-s|^{-\zeta}$ and using $\theta<\omega-\zeta$ proves \eqref{eq:initial-layer-estimate}.  The difference estimate follows from
\begin{equation*}
\Gamma(A,B)-\Gamma(\bar A,\bar B)=\Gamma(A-\bar A,B)+\Gamma(\bar A,B-\bar B).
\end{equation*}
\end{proof}

We now specify the weights used by the finite-time resonant coordinates:
\begin{align}\label{eq:finite-time-coordinate-weights}
\omega_{31}=\omega_{22}:=&\delta+\zeta+\kappa_0,&\omega_{32}:=&\frac14+\delta+\zeta+\kappa_0,\no\\
\omega_{51}=\omega_{24}=\omega_{42}=\omega_{33}:=&\frac{1-\alpha}{2}+\delta+\zeta+\kappa_0,&\omega_{53}=\omega_{44}:=&1-\alpha+\delta+\zeta+\kappa_0,\\
\omega_{52}=\omega_{34}:=&\frac34-\frac\alpha2+\delta+\zeta+\kappa_0,&\omega_{54}:=&\frac54-\alpha+\delta+\zeta+\kappa_0.\no
\end{align}
For the correction defining the coordinate indexed by $\tau$, let $\theta_\tau^*$ denote the threshold in \eqref{eq:initial-layer-threshold}.  A direct calculation gives $\omega_\tau=\zeta+\theta_\tau^*+\kappa_0$ for every resonant coordinate $\tau$.  By \eqref{eq:finite-time-weight-choice}, all these weights lie in $(0,1)$.

The finite-time standard model distribution space is defined by
\begin{align*}
\mathfrak X_{T,\alpha}^{0,\Phi}:=C_T^{\delta',-\ff12-\delta}\times C_T^{\delta',-1-\delta}\times\mathcal E_{T,0}^{\mathfrak d;1-\delta}\times\mathcal E_{T,0}^{\mathfrak d;\ff12-\delta}\times\mathcal W_T^{\omega_{31},\zeta;-\delta}\times\mathcal W_T^{\omega_{22},\zeta;-\delta}\times\mathcal W_T^{\omega_{32},\zeta;-\ff12-\delta}.
\end{align*}
The product space for the unchanged higher-order stochastic coordinates and the finite-time integrated Wick-power components is defined by
\begin{align*}
\mathfrak V_{T,\alpha}^{0,1}:=C_T^{\delta',\alpha-1-\delta}\times C_T^{\delta',\alpha-\ff32-\delta}\times C_T^{\delta',\alpha-2-\delta}\times\mathcal E_{T,0}^{\mathfrak d;\alpha+1-\delta}\times\mathcal E_{T,0}^{\mathfrak d;\alpha+\ff12-\delta}\times\mathcal E_{T,0}^{\mathfrak d;\alpha-\delta}\times\mathcal E_{T,0}^{\mathfrak d;\alpha-\ff12-\delta},
\end{align*}
and the resonant higher-order block is defined by
\begin{align*}
\mathfrak V_{T,\alpha}^{0,2}:=&\mathcal W_T^{\omega_{53},\zeta;2\alpha-2-\delta}\times\mathcal W_T^{\omega_{52},\zeta;\alpha-\ff32-\delta}\times\mathcal W_T^{\omega_{51},\zeta;\alpha-1-\delta}\times\mathcal W_T^{\omega_{24},\zeta;\alpha-1-\delta}\times\mathcal W_T^{\omega_{42},\zeta;\alpha-1-\delta}\\
&\qquad\times\mathcal W_T^{\omega_{33},\zeta;\alpha-1-\delta}\times\mathcal W_T^{\omega_{44},\zeta;2\alpha-2-\delta}\times\mathcal W_T^{\omega_{34},\zeta;\alpha-\ff32-\delta}\times\mathcal W_T^{\omega_{54},\zeta;2\alpha-\ff52-\delta}.
\end{align*}
We set
\begin{equation*}
\mathfrak X_{T,\alpha}^0:=\mathfrak X_{T,\alpha}^{0,\Phi}\times\mR^2\times\mathfrak V_{T,\alpha}^{0,1}\times\mathfrak V_{T,\alpha}^{0,2},
\end{equation*}
and equip it with the product norm.

Let
\begin{equation*}
\mathbb Z=(\mathbb X^\Phi,\lambda,\mu,\mathbb V)=\mathscr L_T^{\rm stat}(X;\lambda,\mu,\mathbf c)\in\mathscr X_{T,\alpha}^{\infty}
\end{equation*}
be a smooth canonical lift. Writing $\mathbb V=(\mathbb V^1,\mathbb V^2)$, we denote its coordinates by the three displays below. We first define all finite-time coordinates on this smooth canonical core.
$$\mathbb X^\Phi=\bigl(X,X^{\diamond2},\mathbf X_2,\mathbf X_3,\mathbf X_{31},\mathbf X_{22},\mathbf X_{32}\bigr),$$
$$\mathbb V^1=\bigl(Y_2,Y_3,Y_4,\mathbf Y_2,\mathbf Y_3,\mathbf Y_4,\mathbf Y_5\bigr),$$
$$\mathbb V^2=\bigl(\mathbf R_{53},\mathbf R_{52},\mathbf R_{51},\mathbf R_{24},\mathbf R_{42},\mathbf R_{33},\mathbf R_{44},\mathbf R_{34},\mathbf R_{54}\bigr).$$
Define the finite-time integrated coordinates by
\begin{equation}\label{eq:finite-time-primitives-definition}
\mathbf X_2^0:=Q_0\mathbf X_2,\ \mathbf X_3^0:=Q_0\mathbf X_3,\ \mathbf Y_j^0:=Q_0\mathbf Y_j,\ j=2,3,4,5.
\end{equation}

The standard finite-time resonances are
\begin{align}\label{eq:finite-time-standard-resonances}
\mathbf X_{31}^0:=\mathbf X_{31}-(P_\cdot\mathbf X_3(0))\circ X,\ \mathbf X_{22}^0:=\mathbf X_{22}-(P_\cdot\mathbf X_2(0))\circ X^{\diamond2},\ \mathbf X_{32}^0:=\mathbf X_{32}-(P_\cdot\mathbf X_3(0))\circ X^{\diamond2}.
\end{align}
The higher-order finite-time resonances are defined by
\begin{align}\label{eq:finite-time-higher-resonances}
\begin{split}
\mathbf R_{53}^0:=\mathbf R_{53}-(P_\cdot\mathbf Y_5(0))\circ Y_3,\ \mathbf R_{52}^0:=\mathbf R_{52}-\lambda(P_\cdot\mathbf Y_5(0))\circ X^{\diamond2},\ \mathbf R_{51}^0:=\mathbf R_{51}-\lambda(P_\cdot\mathbf Y_5(0))\circ X,\\
\mathbf R_{24}^0:=\mathbf R_{24}-(P_\cdot\mathbf X_2(0))\circ Y_4,\ \mathbf R_{42}^0:=\mathbf R_{42}-(P_\cdot\mathbf Y_4(0))\circ X^{\diamond2},\ \mathbf R_{33}^0:=\mathbf R_{33}-\frac{10}{3}\lambda(P_\cdot\mathbf X_3(0))\circ Y_3,\\
\mathbf R_{44}^0:=\mathbf R_{44}-\frac{25}{9}(P_\cdot\mathbf Y_4(0))\circ Y_4,\ \mathbf R_{34}^0:=\mathbf R_{34}-\frac53\lambda(P_\cdot\mathbf X_3(0))\circ Y_4,\ \mathbf R_{54}^0:=\mathbf R_{54}-\frac53(P_\cdot\mathbf Y_5(0))\circ Y_4.
\end{split}
\end{align}
We then set
$$\mathbb X^{0,\Phi}:=\bigl(X,X^{\diamond2},\mathbf X_2^0,\mathbf X_3^0,\mathbf X_{31}^0,\mathbf X_{22}^0,\mathbf X_{32}^0\bigr),$$
$$\mathbb V^{0,1}:=\bigl(Y_2,Y_3,Y_4,\mathbf Y_2^0,\mathbf Y_3^0,\mathbf Y_4^0,\mathbf Y_5^0\bigr),$$
$$\mathbb V^{0,2}:=\bigl(\mathbf R_{53}^0,\mathbf R_{52}^0,\mathbf R_{51}^0,\mathbf R_{24}^0,\mathbf R_{42}^0,\mathbf R_{33}^0,\mathbf R_{44}^0,\mathbf R_{34}^0,\mathbf R_{54}^0\bigr),$$
$$\mathscr Q_0\mathbb Z=\mathbb Z^0:=\bigl(\mathbb X^{0,\Phi},\lambda,\mu,\mathbb V^{0,1},\mathbb V^{0,2}\bigr).$$
On the smooth canonical core, these definitions amount to replacing $\mathcal I$ by $I_0$ in the integrated and resonant coordinates, while leaving the counterterms unchanged.

Define the finite-time admissible space by
\begin{equation}\label{eq:finite-time-admissible-space}
\mathscr X_{T,\alpha}^0
:=\overline{\Big\{\mathscr Q_0\mathbb Z:
\mathbb Z\in\mathscr X_{T,\alpha}^{\infty}\Big\}}^{\,\mathfrak X_{T,\alpha}^0}.
\end{equation}

\begin{proposition}[Finite-time realization]\label{prop:finite-time-realization}
The map $\mathscr Q_0$, defined on $\mathscr X_{T,\alpha}^{\infty}$ by \eqref{eq:finite-time-primitives-definition}, \eqref{eq:finite-time-standard-resonances}, and \eqref{eq:finite-time-higher-resonances}, extends uniquely to a locally Lipschitz map
\begin{equation*}
\mathscr Q_0:\mathscr X_{T,\alpha}\to\mathscr X_{T,\alpha}^0.
\end{equation*}
More precisely,
\begin{align}\label{eq:finite-time-realization-stability}
\|\mathscr Q_0\mathbb Z-\mathscr Q_0\overline{\mathbb Z}\|_{\mathfrak X_{T,\alpha}^0}\leq C_T(1+\|\mathbb Z\|_{\mathfrak X_{T,\alpha}}+\|\overline{\mathbb Z}\|_{\mathfrak X_{T,\alpha}})^2d_{\mathscr X_{T,\alpha}}(\mathbb Z,\overline{\mathbb Z}).
\end{align}
\end{proposition}

\begin{proof}
The unchanged coordinates are controlled in their original $C_T^{\delta',\rho}$ spaces.  The integrated Wick-power coordinates are controlled by Lemma~\ref{lem:finite-time-primitives}.  Each stationary resonant coordinate, viewed as a path in the finite-time weighted space, is controlled by \eqref{eq:regular-path-weighted-embedding}.  Its initial correction is of the form treated in Lemma~\ref{lem:initial-layer-resonance}, up to a fixed numerical factor and, in some cases, multiplication by the scalar coordinate $\lambda$. Here the weights in \eqref{eq:finite-time-coordinate-weights} were chosen so that \eqref{eq:initial-layer-weight-condition} holds strictly.

The estimates \eqref{eq:regular-path-weighted-embedding}, \eqref{eq:Q0-primitive-estimate}, \eqref{eq:initial-layer-estimate} and \eqref{eq:initial-layer-difference} imply \eqref{eq:finite-time-realization-stability} on smooth canonical lifts.  Since $\mathscr X_{T,\alpha}$ is the closure of the smooth lifts and $\mathfrak X_{T,\alpha}^0$ is complete, the map extends uniquely.  Its image belongs to the closure in \eqref{eq:finite-time-admissible-space}.
\end{proof}

As a consequence of Theorem~\ref{thm:enhanced-data-convergence} and Proposition~\ref{prop:finite-time-realization}, for every $p\in[1,\infty)$,
\begin{equation}\label{eq:finite-time-model-convergence}
\mathbb Z_\varepsilon^0:=\mathscr Q_0\mathbb Z_\varepsilon\to\mathbb Z^0:=\mathscr Q_0\mathbb Z\ \text{in}\ L^p(\Omega;\mathfrak X_{T,\alpha}^0).
\end{equation}
Indeed, \eqref{eq:finite-time-realization-stability} and the H\"older inequality give
\begin{align*}
\|\mathbb Z_\varepsilon^0-\mathbb Z^0\|_{L^p(\Omega;\mathfrak X_{T,\alpha}^0)}\leq C_T\left\|1+\|\mathbb Z_\varepsilon\|_{\mathfrak X_{T,\alpha}}+\|\mathbb Z\|_{\mathfrak X_{T,\alpha}}\right\|_{L^{4p}(\Omega)}^2\left\|d_{\mathscr X_{T,\alpha}}(\mathbb Z_\varepsilon,\mathbb Z)\right\|_{L^{2p}(\Omega)},
\end{align*}
where the first factor is uniformly bounded and the second factor converges to zero by Theorem~\ref{thm:enhanced-data-convergence}. For $0<S\leq T$, let $\operatorname{Res}_{S,T}$ denote the componentwise restriction to $[0,S]$. Both admissible spaces are stable under restriction, and the corresponding restriction maps are contractive:
$$\operatorname{Res}_{S,T}\mathscr X_{T,\alpha}\subseteq\mathscr X_{S,\alpha},\ 
\|\operatorname{Res}_{S,T}\mathbb Z\|_{\mathfrak X_{S,\alpha}}\leq\|\mathbb Z\|_{\mathfrak X_{T,\alpha}};$$
$$\operatorname{Res}_{S,T}\mathscr X_{T,\alpha}^0\subseteq\mathscr X_{S,\alpha}^0,\ \|\operatorname{Res}_{S,T}\mathbb Z^0\|_{\mathfrak X_{S,\alpha}^0}\leq\|\mathbb Z^0\|_{\mathfrak X_{T,\alpha}^0}.$$
Thus both the stationary datum and its finite-time realization can be used on every shorter time interval. In the fixed-point argument below, all stochastic terms are interpreted through $\mathbb Z^0$.

\section{Solution theory}\label{sec03}

This section develops the solution theory needed to transfer the convergence of the enhanced stochastic data to the convergence of the corresponding solutions.  We retain the parameter choices of Section~\ref{subsec:admissible-enhanced-data}, in particular $\alpha\in(\frac56,1)$ and the exponents $\delta,\delta'$ and $\mathfrak d$ satisfying \eqref{eq:model-parameter-choice} and \eqref{eq:model-parabolic-time-choice}.  In Section~\ref{subsec:controlled-distributions}, we introduce controlled pairs $(v,v')$ relative to the finite-time realization $\mathbb Z^0=\mathscr Q_0\mathbb Z$.  The two components are initially independent, and the scaled distance \eqref{eq:controlled-scaled-distance} supplies the small factor needed to handle the derivative swap in the fixed-point map.  In Section~\ref{subsec:renormalized-mild-map}, we construct the renormalized mild map and establish its local bounds and stability with respect to the controlled pair, the enhanced datum and the initial condition.  Finally, Section~\ref{subsec:local-well-posedness} combines these estimates with restriction consistency and the scaled contraction argument to obtain a bounded local fixed point, uniqueness of the associated initial solution germ and quantitative stability on bounded data sets.  Applying this theory to the convergent finite-time models from Section~\ref{subsec:admissible-enhanced-data} yields deterministic convergence of solutions and, under the well-prepared initial condition \eqref{eq:well-prepared-convergence-assumption}, convergence of the approximating fields to the renormalized dynamical $\Phi^4_3(\lambda)$ solution in the localized solution-germ sense.

\subsection{The space of controlled distributions}\label{subsec:controlled-distributions}

We fix an admissible stationary enhanced datum $\mathbb Z\in\mathscr X_{T,\alpha}$ and write $\mathbb Z^0=\mathscr Q_0\mathbb Z\in\mathscr X_{T,\alpha}^0$ for its finite-time realization.  All stochastic coordinates which enter the mild formulation are read from $\mathbb Z^0$.  The positive weights $\omega_\tau$ introduced in Section~\ref{subsec:admissible-enhanced-data} belong only to the finite-time resonant model coordinates.  The controlled solution itself is required to remain uniformly bounded in its base $\bC^\beta$ norm; only its higher spatial norm is allowed to have the standard heat-flow singularity encoded below. Thus no analogue of the model weight $\omega_\tau$ is introduced in the solution space.

Following the pair formulation of \cite[Definition~3.3]{CC18}, a controlled object is a pair $(v,v')$.  The two components are kept independent at this stage; the identity $v'=v$ will follow only at a fixed point of the solution map.

We retain the parameters $\alpha,\delta,\delta',\mathfrak d,\zeta$ fixed in Section~\ref{subsec:admissible-enhanced-data} and set
\begin{equation*}
\sigma:=\alpha-\frac12-2\delta.
\end{equation*}
Choose $\beta$ and $\gamma$ such that
\begin{align}\label{eq:controlled-spatial-exponents}
2\mathfrak d<\beta<\sigma,\ 2-\alpha+2\delta<\gamma<\min\Big\{2\alpha-\frac12-3\delta,\beta+\alpha-\delta\Big\}.
\end{align}
These intervals are nonempty. In particular,
\begin{equation}\label{eq:controlled-ordering}
0<2\mathfrak d<\beta<\sigma<\gamma.
\end{equation}
Here $\beta$ describes the spatial regularity of the initial remainder $v_0$, while $\sigma$ and $\gamma$ describe the spatial regularities of $(v,v')$ and $v^\sharp$, respectively, in the controlled fixed-point space (see Definition \ref{def:controlled-distributions} below). The lower bound on $\mathfrak d$ is precisely what makes the time-increment part of the integrated diagonal operator locally integrable. Here, the diagonal terms refer to the resonant interactions $I_0\bigl(I_0(f\prec A_i)\circ A_j\bigr)$, $i,j\in\{2,4\}$, where $A_2=X^{\diamond2}$ and $A_4=Y_4$. Their commutator expansion defines the integrated operator $\mathscr J_{ij}$ in \eqref{eq:J-definition} below. The upper bound $\gamma<\beta+\alpha-\delta$ is needed to propagate the rough initial condition through the terms containing $Y_4$. 

For $\rho_0\leq\rho_1$ and a path $f:[0,T]\to\bC^{\rho_0}$ whose restriction to $(0,T]$ takes values in $\bC^{\rho_1}$, define
\begin{align}\label{eq:controlled-path-norm}
\begin{split}
\|f\|_{\mathcal L_T^{\rho_0,\rho_1}}:=&\sup_{0\leq t\leq T}\|f_t\|_{\rho_0}+\sup_{0<t\leq T}t^{\frac{\rho_1-\rho_0}{2}}\|f_t\|_{\rho_1}\\
&+\sup_{0\leq s<t\leq T}\frac{\|f_t-f_s\|_{\rho_0-2\mathfrak d}}{|t-s|^\mathfrak d}+\sup_{0<s<t\leq T}s^{\frac{\rho_1-\rho_0}{2}}\frac{\|f_t-f_s\|_{\rho_1-2\mathfrak d}}{|t-s|^\mathfrak d}.
\end{split}
\end{align}
We denote by $\mathcal L_T^{\rho_0,\rho_1}$ the Banach space of paths with finite norm \eqref{eq:controlled-path-norm}.  Notice that continuity at the initial time is required only in the weaker space $\bC^{\rho_0-2\mathfrak d}$; this convention allows the heat extension of every $a\in\bC^{\rho_0}$, without imposing strong continuity of the heat semigroup on the whole Besov space $B_{\infty,\infty}^{\rho_0}$.  For $a\in\bC^{\rho_0}$, set
\begin{equation*}
\mathcal L_{T,a}^{\rho_0,\rho_1}:=\left\{f\in\mathcal L_T^{\rho_0,\rho_1}:f_0=a\right\}.
\end{equation*}
Thus $\mathcal L_{T,a}^{\rho_0,\rho_1}$ is a closed affine subspace and $\mathcal L_{T,0}^{\rho_0,\rho_1}$ is a Banach space.  Restriction to a shorter interval is contractive in these norms. Moreover, we have the following elementary embedding property.

\begin{lemma}\label{lem:controlled-path-embeddings}
Let $0<T\leq1$.
\begin{enumerate}[]
\item (i) For every $a\in\bC^{\rho_0}$,
\begin{equation}
\|P_\cdot a\|_{\mathcal L_T^{\rho_0,\rho_1}}\lesssim\|a\|_{\rho_0}.\no
\end{equation}
\item (ii) If $\rho_0\leq\rho_1\leq\rho_2$, then
\begin{equation}
\mathcal L_T^{\rho_0,\rho_2}\hookrightarrow\mathcal L_T^{\rho_0,\rho_1}.\no
\end{equation}
\item (iii) If $\rho_0\leq\rho_1\leq\rho$, then
$$\mathcal E_{T,0}^{\mathfrak d;\rho}\hookrightarrow\mathcal L_{T,0}^{\rho_0,\rho_1},\ \|F\|_{\mathcal L_T^{\rho_0,\rho_1}}\lesssim\|F\|_{\mathcal E_T^{\mathfrak d;\rho}}.$$
\end{enumerate}
The implicit constants are uniform for $0<T\leq1$.
\end{lemma}

\begin{proof}
The first statement follows the heat-flow estimates: for $0<s<t$,
$$\|P_ta\|_{\rho_1}\lesssim t^{-\frac{\rho_1-\rho_0}{2}}\|a\|_{\rho_0},\ \|(P_t-P_s)a\|_{\rho_0-2\mathfrak d}\lesssim|t-s|^\mathfrak d\|a\|_{\rho_0},$$
$$\|(P_t-P_s)a\|_{\rho_1-2\mathfrak d}\lesssim|t-s|^\mathfrak d s^{-\frac{\rho_1-\rho_0}{2}}\|a\|_{\rho_0}.$$

For the second statement, if $\rho_2=\rho_0$, there is nothing to prove.  Suppose that
$\rho_2>\rho_0$ and set $\theta:=\frac{\rho_1-\rho_0}{\rho_2-\rho_0}\in[0,1]$. Then the spatial interpolation gives
$$t^{\frac{\rho_1-\rho_0}{2}}\|f_t\|_{\rho_1}\lesssim\|f_t\|_{\rho_0}^{1-\theta}(t^{\frac{\rho_2-\rho_0}{2}}\|f_t\|_{\rho_2})^\theta,$$
$$s^{\frac{\rho_1-\rho_0}{2}}\frac{\|f_t-f_s\|_{\rho_1-2\mathfrak d}}{|t-s|^\mathfrak d}\lesssim\Big(\frac{\|f_t-f_s\|_{\rho_0-2\mathfrak d}}{|t-s|^\mathfrak d}\Big)^{1-\theta}\Big(s^{\frac{\rho_2-\rho_0}{2}}\frac{\|f_t-f_s\|_{\rho_2-2\mathfrak d}}{|t-s|^\mathfrak d}\Big)^\theta.$$
The remaining two terms in the
$\mathcal L_T^{\rho_0,\rho_1}$ norm are directly controlled by their
counterparts in $\mathcal L_T^{\rho_0,\rho_2}$. 

For the third statement, it follows directly from the spatial embeddings $\bC^\rho\hookrightarrow\bC^{\rho_1}\hookrightarrow\bC^{\rho_0}$. The proof is complete.
\end{proof}

The finite-time integrated Wick-power coordinates which enter the controlled ansatz satisfy 
$$\mathbf X_3^0\in\mathcal E_{T,0}^{\mathfrak d;\frac12-\delta},\ \mathbf Y_5^0\in\mathcal E_{T,0}^{\mathfrak d;\alpha-\frac12-\delta}$$
Lemma~\ref{lem:controlled-path-embeddings} yields
\begin{equation*}
\|\mathbf X_3^0\|_{\mathcal L_T^{\beta,\sigma}}+\|\mathbf Y_5^0\|_{\mathcal L_T^{\beta,\sigma}}\lesssim\|\mathbb Z^0\|_{\mathfrak X_{T,\alpha}^0}.
\end{equation*}
We also record the deterministic estimate for the two integrated paraproducts appearing in the ansatz.

\begin{lemma}\label{lem:controlled-integrated-paraproduct}
Let $r\in\mR$ satisfy $ r+2\mathfrak d<\beta<\sigma<r+2$, $A\in C_T^{\delta',r}$ and $f\in\mathcal L_T^{\beta,\sigma}$.  Define
\begin{equation*}
\mathcal B_A(f):=I_0(f\prec A).
\end{equation*}
Then $\mathcal B_A(f)\in\mathcal L_{T,0}^{\beta,\sigma}$ and, uniformly for $0<T\leq1$,
\begin{equation}\label{eq:integrated-paraproduct-estimate}
\|\mathcal B_A(f)\|_{\mathcal L_T^{\beta,\sigma}}\lesssim T^{1-\frac{\beta-r}{2}}\|A\|_{C_T^{\delta',r}}\|f\|_{\mathcal L_T^{\beta,\sigma}}.
\end{equation}
The same estimate also holds for differences.
\end{lemma}

\begin{proof}
Since $\beta>2\mathfrak d>0$, 
\begin{equation*}
\|(f\prec A)_t\|_r\lesssim\|f_t\|_\beta\|A_t\|_r.
\end{equation*}
Since $\beta<\sigma<r+2$, the Schauder estimate gives
\begin{align}\label{eq:integrated-paraproduct-pointwise-proof}
\|\mathcal B_A(f)_t\|_\beta\lesssim t^{1-\frac{\beta-r}{2}}\sup_{0\leq u\leq T}\|(f\prec A)_u\|_r,\ t^{\ff{\s-\b}2}\|\mathcal B_A(f)_t\|_\sigma\lesssim t^{1-\frac{\beta-r}{2}}\sup_{0\leq u\leq T}\|(f\prec A)_u\|_r.
\end{align}
We next estimate the time increments.  Let $0\leq s<t\leq T$ and set
$h:=t-s$.  The semigroup property gives
\begin{equation*}
\mathcal B_A(f)_t-\mathcal B_A(f)_s=(P_h-\mathrm{Id})\mathcal B_A(f)_s+\int_s^tP_{t-u}(f\prec A)_u\dif u.
\end{equation*}
For the low increment norm, note $\beta-2\mathfrak d>r$.  Therefore,
$$\|(P_h-\mathrm{Id})\mathcal B_A(f)_s\|_{\beta-2\mathfrak d}\lesssim h^\mathfrak d\|\mathcal B_A(f)_s\|_\beta\lesssim h^\mathfrak d s^{1-\frac{\beta-r}{2}}\sup_{0\leq u\leq T}\|(f\prec A)_u\|_r,$$
$$\Big\|\int_s^tP_{t-u}(f\prec A)_u\dif u\Big\|_{\beta-2\mathfrak d}\lesssim h^{1-\frac{\beta-r}{2}+\mathfrak d}\sup_{0\leq u\leq T}\|(f\prec A)_u\|_r.$$
Consequently,
\begin{equation}\label{eq:integrated-paraproduct-low-increment}
\|\mathcal B_A(f)_t-\mathcal B_A(f)_s\|_{\beta-2\mathfrak d}\lesssim h^\mathfrak d T^{1-\frac{\beta-r}{2}}\sup_{0\leq u\leq T}\|(f\prec A)_u\|_r.
\end{equation}
Similarly,
\begin{equation}\label{eq:integrated-paraproduct-high-increment}
s^{\ff{\s-\b}2}\|\mathcal B_A(f)_t-\mathcal B_A(f)_s\|_{\sigma-2\mathfrak d}\lesssim h^\mathfrak d T^{1-\frac{\beta-r}{2}}\sup_{0\leq u\leq T}\|(f\prec A)_u\|_r.
\end{equation}
Combining \eqref{eq:integrated-paraproduct-pointwise-proof}, \eqref{eq:integrated-paraproduct-low-increment} and \eqref{eq:integrated-paraproduct-high-increment} proves
\eqref{eq:integrated-paraproduct-estimate}.  The difference estimate follows from the bilinear identity
\begin{equation*}
f\prec A-\bar f\prec\bar A=(f-\bar f)\prec A+\bar f\prec(A-\bar A).
\end{equation*}
\end{proof}

For the two singular directions, set $r_2:=-1-\delta$ and $r_4:=\alpha-2-\delta$. The parameter choices imply
\begin{align}\label{eq:diagonal-integrability-margins}
\begin{split}
\min_{i,j\in\{2,4\}}(2\mathfrak d+r_i+r_j+2)&=2\mathfrak d+2\alpha-2-2\delta>0,\\
\min_{i,j\in\{2,4\}}(\beta+r_i+r_j+2)&=\beta+2\alpha-2-2\delta>0.
\end{split}
\end{align}
Moreover,
\begin{equation*}
\gamma<\beta+r_i+2,\ i\in\{2,4\}.
\end{equation*}
The first line of \eqref{eq:diagonal-integrability-margins} is the condition needed for the time-increment part of the integrated diagonal operator; the second is the spatial commutator condition.

Since $\alpha<1$, one has $r_4<r_2<0$.  The choice of $\beta$ and $\sigma$ therefore implies
\begin{equation*}
r_i<\beta<\sigma<r_i+2,\ \min_{i\in\{2,4\}}(\beta-r_i)=\beta+1+\delta>2\mathfrak d.
\end{equation*}
Hence all the assumptions of Lemma~\ref{lem:controlled-integrated-paraproduct} are satisfied for $A_2=X^{\diamond2}$ and $A_4=Y_4$.  The smaller of the two time powers is
\begin{equation*}
\theta_0:=\min_{i\in\{2,4\}}\big(1-\frac{\beta-r_i}{2}\big)=\frac{\alpha-\delta-\beta}{2}>0.
\end{equation*}

\begin{remark}\label{rem:well-prepared-initial-data}
The deterministic fixed-point theory below is formulated for a remainder initial condition $v_0\in\bC^\beta$.  For the approximating equation this corresponds to the well-prepared decomposition
\begin{equation*}
u_{\varepsilon,0}=X_\varepsilon(0)+v_{\varepsilon,0},\ v_{\varepsilon,0}\to v_0\ \text{in}\ \bC^\beta.
\end{equation*}
If $u_{\varepsilon,0}$ is prescribed independently of $X_\varepsilon(0)$, then $v_{\varepsilon,0}=u_{\varepsilon,0}-X_\varepsilon(0)$ has negative regularity.  In the presence of the regular quintic term, that situation requires a separate initial-layer analysis and is not covered here.
\end{remark}

Write the relevant finite-time coordinates as in Section~\ref{subsec:admissible-enhanced-data}:
$$\mathbb X^{0,\Phi}=\bigl(X,X^{\diamond2},\mathbf X_2^0,\mathbf X_3^0,\mathbf X_{31}^0,\mathbf X_{22}^0,\mathbf X_{32}^0\bigr),$$
$$\mathbb V^{0,1}=\bigl(Y_2,Y_3,Y_4,\mathbf Y_2^0,\mathbf Y_3^0,\mathbf Y_4^0,\mathbf Y_5^0\bigr).$$
On a smooth canonical lift,
\begin{equation*}
Y_4=\mu X^{\diamond4},\ \mathbf X_3^0=I_0(X^{\diamond3}),\ \mathbf Y_5^0=\mu I_0(X^{\diamond5}).
\end{equation*}
For a general admissible datum, the same symbols denote the corresponding coordinates of $\mathbb Z^0$.

Let $v_0\in\bC^\beta$ and let
\begin{equation*}
\mathbf v=(v,v')\in\bigl(\mathcal L_{T,v_0}^{\beta,\sigma}\bigr)^2.
\end{equation*}
We define the controlled remainder by
\begin{align}\label{eq:controlled-remainder-definition}
v^\sharp:=v-P_\cdot v_0+\mathbf Y_5^0+\lambda\mathbf X_3^0+5I_0(v'\prec Y_4)+3\lambda I_0(v'\prec X^{\diamond2}).
\end{align}
The two integrated paraproducts identify the singular directions of the equation: $Y_4$ is generated by the quintic perturbation, while $X^{\diamond2}$ is the usual cubic direction.

\begin{definition}[Controlled distributions]\label{def:controlled-distributions}
Let $\mathbb Z\in\mathscr X_{T,\alpha}$, $\mathbb Z^0=\mathscr Q_0\mathbb Z$ and $v_0\in\bC^\beta$.  We say that $\mathbf v=(v,v')$ is controlled by $\mathbb Z$ with initial condition $v_0$ if
\begin{equation*}
(v,v')\in\bigl(\mathcal L_{T,v_0}^{\beta,\sigma}\bigr)^2,\ v^\sharp\in\mathcal L_{T,0}^{\beta,\gamma},
\end{equation*}
where $v^\sharp$ is defined by \eqref{eq:controlled-remainder-definition}.  The space of such pairs is denoted by $\mathcal D_T^{\beta,\gamma}(\mathbb Z;v_0)$.  The size functional is defined by
\begin{equation*}
\|\mathbf v\|_{\mathcal D_T^{\beta,\gamma}(\mathbb Z;v_0)}:=\|v'\|_{\mathcal L_T^{\beta,\sigma}}+\|v^\sharp\|_{\mathcal L_T^{\beta,\gamma}}.
\end{equation*}
For two controlled pairs associated with the same enhanced datum and the same initial condition, define
\begin{equation}\label{eq:controlled-distance}
d_{\mathcal D_T}(\mathbf v,\overline{\mathbf v}):=\|v'-\bar v'\|_{\mathcal L_T^{\beta,\sigma}}+\|v^\sharp-\bar v^\sharp\|_{\mathcal L_T^{\beta,\gamma}}.
\end{equation}
For $\chi>0$, we shall also use the equivalent $T$-dependent metric
\begin{equation}\label{eq:controlled-scaled-distance}
d_{\mathcal D_T}^{(\chi)}(\mathbf v,\overline{\mathbf v}):=T^\chi\|v'-\bar v'\|_{\mathcal L_T^{\beta,\sigma}}+\|v^\sharp-\bar v^\sharp\|_{\mathcal L_T^{\beta,\gamma}}.
\end{equation}
The scaled metric will be used only in the contraction argument; it is not an initial-time weight in the definition of the solution.
\end{definition}

The first component $v$ is reconstructed continuously from $(v',v^\sharp)$ and the fixed data $(\mathbb Z^0,v_0)$.

\begin{lemma}\label{lem:controlled-reconstruction}
Uniformly for $0<T\leq1$, every $\mathbf v\in\mathcal D_T^{\beta,\gamma}(\mathbb Z;v_0)$ satisfies
\begin{align}\label{eq:controlled-reconstruction-estimate}
\begin{split}
\|v\|_{\mathcal L_T^{\beta,\sigma}}\lesssim&\|v_0\|_\beta+\|\mathbf Y_5^0\|_{\mathcal E_T^{\mathfrak d;\alpha-\frac12-\delta}}+|\lambda|\|\mathbf X_3^0\|_{\mathcal E_T^{\mathfrak d;\frac12-\delta}}+\|v^\sharp\|_{\mathcal L_T^{\beta,\gamma}}\\
&+T^{\theta_0}\big(5\|Y_4\|_{C_T^{\delta',\alpha-2-\delta}}+3|\lambda|\|X^{\diamond2}\|_{C_T^{\delta',-1-\delta}}\big)\|v'\|_{\mathcal L_T^{\beta,\sigma}}.
\end{split}
\end{align}
In particular, for every $R>0$ there exists a constant $C_R<\infty$ such that
\begin{equation*}
\|v\|_{\mathcal L_T^{\beta,\sigma}}\leq C_R\big(1+\|\mathbf v\|_{\mathcal D_T^{\beta,\gamma}(\mathbb Z;v_0)}\big)
\end{equation*}
whenever $\|\mathbb Z^0\|_{\mathfrak X_{T,\alpha}^0}+\|v_0\|_\beta\leq R$.

If $\mathbf v,\overline{\mathbf v}$ are controlled by the same datum and have the same initial condition, then
\begin{align}\label{eq:controlled-reconstruction-same-data-difference}
\|v-\bar v\|_{\mathcal L_T^{\beta,\sigma}}\lesssim\|v^\sharp-\bar v^\sharp\|_{\mathcal L_T^{\beta,\gamma}}+T^{\theta_0}\big(\|Y_4\|_{C_T^{\delta',\alpha-2-\delta}}+|\lambda|\|X^{\diamond2}\|_{C_T^{\delta',-1-\delta}}\big)\|v'-\bar v'\|_{\mathcal L_T^{\beta,\sigma}}.
\end{align}
Consequently, on bounded sets and for every $\chi>0$,
\begin{equation}\label{eq:controlled-derivative-swap-smallness}
T^\chi\|v-\bar v\|_{\mathcal L_T^{\beta,\sigma}}\lesssim_R\bigl(T^\chi+T^{\theta_0}\bigr)d_{\mathcal D_T}^{(\chi)}(\mathbf v,\overline{\mathbf v}).
\end{equation}

More generally, if $\mathbf v\in\mathcal D_T^{\beta,\gamma}(\mathbb Z;v_0)$ and
$\overline{\mathbf v}\in\mathcal D_T^{\beta,\gamma}(\overline{\mathbb Z};\bar v_0)$, then, on every set on which the two model norms, the two initial norms and the two controlled sizes are bounded by $R$,
\begin{align}\label{eq:controlled-reconstruction-different-data}
\|v-\bar v\|_{\mathcal L_T^{\beta,\sigma}}\lesssim_R\|v_0-\bar v_0\|_\beta+d_{\mathcal D_T}(\mathbf v,\overline{\mathbf v})+\|\mathbb Z^0-\overline{\mathbb Z}^0\|_{\mathfrak X_{T,\alpha}^0}.
\end{align}
Here the term $d_{\mathcal D_T}(\mathbf v,\overline{\mathbf v})$ is still defined by \eqref{eq:controlled-distance}.
\end{lemma}

\begin{proof}
Lemma~\ref{lem:controlled-path-embeddings} controls the heat extension, the two finite-time integrated Wick-power coordinates and the embedding $\mathcal L_{T,0}^{\beta,\gamma} \hookrightarrow\mathcal L_{T,0}^{\beta,\sigma}$. Lemma~\ref{lem:controlled-integrated-paraproduct}, applied with $r=r_2$ and $r=r_4$, controls the two integrated paraproducts.  Substitution in \eqref{eq:controlled-remainder-definition} proves \eqref{eq:controlled-reconstruction-estimate}.

For two pairs over the same data, subtract their reconstruction formulas and apply the difference part of Lemma~\ref{lem:controlled-integrated-paraproduct}; this proves \eqref{eq:controlled-reconstruction-same-data-difference} and then \eqref{eq:controlled-derivative-swap-smallness}.  For different data, expand each bilinear term by a telescoping identity.  The finite-time integrated Wick-power differences and the differences of $Y_4$, $X^{\diamond2}$ and $\lambda$ are components of $\mathbb Z^0-\overline{\mathbb Z}^0$.  This gives \eqref{eq:controlled-reconstruction-different-data}.
\end{proof}

\begin{proposition}\label{prop:controlled-space-complete}
For fixed $(\mathbb Z,v_0)$, the metric space $\mathcal D_T^{\beta,\gamma}(\mathbb Z;v_0)$ is complete for $d_{\mathcal D_T}$.  It is also complete for each equivalent metric $d_{\mathcal D_T}^{(\chi)}$ with $\chi>0$.
\end{proposition}

\begin{proof}
The map $\mathbf v=(v,v')\longmapsto(v',v^\sharp)$ is an affine bijection from $\mathcal D_T^{\beta,\gamma}(\mathbb Z;v_0)$ onto $\mathcal L_{T,v_0}^{\beta,\sigma}\times\mathcal L_{T,0}^{\beta,\gamma}$. Its inverse is the reconstruction formula \eqref{eq:controlled-remainder-definition}, which is continuous by Lemma~\ref{lem:controlled-reconstruction}.  The two path spaces are complete, and the scaled metric is equivalent to the unscaled one for every fixed $T>0$.
\end{proof}

\begin{remark}[Role of $\alpha>\frac56$]\label{rem:alpha-threshold-controlled-space}
The most singular classical product involving the controlled remainder is $v^\sharp\circ Y_4$.  Since $v^\sharp_t\in\bC^\gamma$ and $Y_{4,t}\in\bC^{\alpha-2-\delta}$, it is well defined provided $\gamma+\alpha-2-\delta>0$. On the other hand, the paracontrolled expansion produces at most the remainder regularity $2\alpha-\frac12-$, which explains the upper bound imposed on $\gamma$.  Ignoring arbitrarily small losses, the interval
\begin{equation*}
2-\alpha<\gamma<2\alpha-\frac12
\end{equation*}
is nonempty exactly when $\alpha>\frac56$.
\end{remark}

The fixed-point map constructed below has the form
\begin{equation*}
\mathcal M_{\mathbb Z,v_0}(v,v')=\bigl(\Gamma_{\mathbb Z,v_0}(v,v'),v\bigr).
\end{equation*}
Therefore every fixed point satisfies $v'=v$.  In the contraction argument we shall use the scaled metric \eqref{eq:controlled-scaled-distance}; the smallness of the derivative swap in the second component is precisely \eqref{eq:controlled-derivative-swap-smallness}.

\subsection{The renormalized mild map}\label{subsec:renormalized-mild-map}

We now define the nonlinear mild map for a generic admissible enhanced datum. The only genuinely diagonal terms are treated after the outer heat integration.

Fix $\lambda_*>0$ and let
\begin{equation*}
\mathbb Z=(\mathbb X^\Phi,\lambda,\mu,\mathbb V)\in\mathscr X_{T,\alpha},\ \lambda\geq\lambda_*.
\end{equation*}
Write $\mathbb Z^0=\mathscr Q_0\mathbb Z$ for the finite-time realization and use the coordinate notation
$$\mathbb X^{0,\Phi}:=\bigl(X,X^{\diamond2},\mathbf X_2^0,\mathbf X_3^0,\mathbf X_{31}^0,\mathbf X_{22}^0,\mathbf X_{32}^0\bigr),$$
$$\mathbb V^{0,1}:=\bigl(Y_2,Y_3,Y_4,\mathbf Y_2^0,\mathbf Y_3^0,\mathbf Y_4^0,\mathbf Y_5^0\bigr),$$
$$\mathbb V^{0,2}:=\bigl(\mathbf R_{53}^0,\mathbf R_{52}^0,\mathbf R_{51}^0,\mathbf R_{24}^0,\mathbf R_{42}^0,\mathbf R_{33}^0,\mathbf R_{44}^0,\mathbf R_{34}^0,\mathbf R_{54}^0\bigr),$$
On a smooth canonical lift,
\begin{align*}
Y_j=\mu X^{\diamond j},\ j=2,3,4;\ \mathbf Y_j^0=\mu I_0(X^{\diamond j}),\ j=2,3,4,5.
\end{align*}
Set
\begin{equation*}
A_2:=X^{\diamond2},\ A_4:=Y_4,\ \mathbf A_2:=\mathbf X_2^0,\ \mathbf A_4:=\mathbf Y_4^0.
\end{equation*}
Thus $A_i\in C_T^{\delta',r_i}$ and $\mathbf A_i\in\mathcal E_{T,0}^{\mathfrak d;r_i+2}$, where $r_2=-1-\delta$ and $r_4=\alpha-2-\delta$.

For $t>0$ and smooth spatial distributions, define the commutator
\begin{align*}
\mathfrak R_t^{\mathrm h}(f,g):=P_t(f\prec g)-f\prec P_t g,\ \mathfrak R(f,g,h):=(f\prec g)\circ h-f(g\circ h).
\end{align*}
For a smooth path $f$ and $i,j\in\{2,4\}$, set
\begin{align}\label{eq:J-time-definition}
\mathscr J_{ij}^{\mathrm{time}}(f)(t):=\int_0^t P_{t-s}\int_0^s(f_r-f_s)(P_{s-r}A_{i,r}\circ A_{j,s})\dif r\dif s,
\end{align}
\begin{align}
\mathscr J_{ij}^{\mathrm{heat}}(f)(t):=\int_0^t P_{t-s}\int_0^s\bigl(\mathfrak R_{s-r}^{\mathrm h}(f_r,A_{i,r})\circ A_{j,s}\bigr)\dif r\dif s,
\end{align}
\begin{align}\label{eq:J-bony-definition}
\mathscr J_{ij}^{\mathrm{Bony}}(f)(t):=\int_0^t P_{t-s}\int_0^s\mathfrak R(f_r,P_{s-r}A_{i,r},A_{j,s})\dif r\dif s,
\end{align}
and
\begin{equation}\label{eq:J-definition}
\mathscr J_{ij}(f):=\mathscr J_{ij}^{\mathrm{time}}(f)+\mathscr J_{ij}^{\mathrm{heat}}(f)+\mathscr J_{ij}^{\mathrm{Bony}}(f).
\end{equation}

Choose $\kappa_{\mathrm J}>0$ so small that
$$0<\kappa_{\mathrm J}<\min\{\beta-2\mathfrak d,\,2\mathfrak d+2\alpha-2-2\delta\},$$
and define
$$\eta_{\mathrm J}:=\min\{\beta-2\mathfrak d,\,2\mathfrak d+2\alpha-2-2\delta\}-\kappa_{\mathrm J}>0.$$

\begin{lemma}[Integrated diagonal estimate]\label{lem:integrated-diagonal-estimate}
For every $i,j\in\{2,4\}$, the three expressions \eqref{eq:J-time-definition}--\eqref{eq:J-bony-definition} admit canonical analytic interpretations for $f\in\mathcal L_T^{\beta,\sigma}$, $A_i\in C_T^{\delta',r_i}$ and $A_j\in C_T^{\delta',r_j}$. Their sum defines a jointly continuous map
\begin{equation*}
\mathscr J_{ij}:
\mathcal L_T^{\beta,\sigma}
\times C_T^{\delta',r_i}
\times C_T^{\delta',r_j}
\longrightarrow
\mathcal L_{T,0}^{\beta,\gamma}.
\end{equation*}
There exist $\theta_{\mathrm J}>0$ and a polynomial $\mathcal Q$ such that, uniformly for $0<T\leq1$,
\begin{equation}\label{eq:J-estimate}
\|\mathscr J_{ij}(f)\|_{\mathcal L_T^{\beta,\gamma}}\leq T^{\theta_{\mathrm J}}\mathcal Q\bigl(\|A_i\|_{C_T^{\delta',r_i}},\|A_j\|_{C_T^{\delta',r_j}}\bigr)\|f\|_{\mathcal L_T^{\beta,\sigma}}.
\end{equation}
On bounded sets, the corresponding difference estimate is
\begin{align}\label{eq:J-difference-estimate}
\|\mathscr J_{ij}^{A_i,A_j}(f)-\mathscr J_{ij}^{\bar A_i,\bar A_j}(\bar f)\|_{\mathcal L_T^{\beta,\gamma}}\lesssim_R T^{\theta_{\mathrm J}}\bigl(\|f-\bar f\|_{\mathcal L_T^{\beta,\sigma}}+\|A_i-\bar A_i\|_{C_T^{\delta',r_i}}+\|A_j-\bar A_j\|_{C_T^{\delta',r_j}}\bigr).
\end{align}
For smooth data one has the exact identity
\begin{align}\label{eq:integrated-master-identity}
I_0(I_0(f\prec A_i)\circ A_j)=I_0(f(\mathbf A_i\circ A_j))+\mathscr J_{ij}(f).
\end{align}
\end{lemma}

\begin{proof}
We indicate the estimates which are responsible for the parameter conditions. Let $\tau=s-r$.  By the time increment in \eqref{eq:controlled-path-norm},
\begin{equation*}\|f_r-f_s\|_{\beta-2\mathfrak d}\lesssim\tau^{\mathfrak d}\|f\|_{\mathcal L_T^{\beta,\sigma}}.
\end{equation*}
Set $\eta_{ij}^{\mathrm{time}}:=2\mathfrak d+r_i+r_j+2-\kappa_{\mathrm J}>0$ by the first line of \eqref{eq:diagonal-integrability-margins}. The heat-flow and resonant-product estimates give
\begin{equation*}
\|P_\tau A_{i,r}\circ A_{j,s}\|_{\eta_{ij}^{\mathrm{time}}}\lesssim\tau^{-1-\mathfrak d+\ff{\kappa_{\mathrm J}}2}\|A_i\|_{C_T^{\delta',r_i}}\|A_j\|_{C_T^{\delta',r_j}}.
\end{equation*}
Since $\beta-2\mathfrak d>0$, multiplication by $f_r-f_s$ is classical and the integrand in \eqref{eq:J-time-definition} belongs to $\bC^{\eta_{\mathrm J}}$ with size bounded by $\tau^{-1+\ff{\kappa_{\mathrm J}}2}\|f\|_{\mathcal L_T^{\beta,\sigma}}\|A_i\|_{C_T^{\delta',r_i}}\|A_j\|_{C_T^{\delta',r_j}}$. The singularity is integrable at $r=s$.

For the heat commutator, the standard estimate gives
\begin{equation*}
\|\mathfrak R_\tau^{\mathrm h}(f_r,A_{i,r})\|_{\beta+r_i+2-\kappa_{\mathrm J}}\lesssim\tau^{-1+\ff{\kappa_{\mathrm J}}2}\|f_r\|_\beta\|A_{i,r}\|_{r_i}.
\end{equation*}
After taking the resonance with $A_{j,s}$, the spatial regularity is $\beta+r_i+r_j+2-\kappa_{\mathrm J}>0$ by the second line of \eqref{eq:diagonal-integrability-margins}.

Finally, for $i,j\in\{2,4\}$, $r_i+r_j+2-\kappa_{\mathrm J}<0$, while $\beta+r_i+r_j+2-\kappa_{\mathrm J}>0$.  Hence the commutator estimate applied to $\mathfrak R(f_r,P_\tau A_{i,r},A_{j,s})$ yields the same $\tau^{-1+\ff{\kappa_{\mathrm J}}2}$ bound in a space of regularity at least $\eta_{\mathrm J}$.

Since $\eta_{\mathrm J}>0$ and $\gamma<2<\eta_{\mathrm J}+2$, the Schauder estimate, including its time-increment version with exponent $\mathfrak d$, maps the outer integral into $\mathcal L_{T,0}^{\beta,\gamma}$ and yields a positive power of $T$. This proves \eqref{eq:J-estimate}.  The estimates are multilinear, so telescoping gives \eqref{eq:J-difference-estimate}.

For the algebraic identity, write
$$P_{s-r}(f_r\prec A_{i,r})=f_r\prec P_{s-r}A_{i,r}+\mathfrak R_{s-r}^{\mathrm h}(f_r,A_{i,r}),$$
$$(f_r\prec P_{s-r}A_{i,r})\circ A_{j,s}=f_r(P_{s-r}A_{i,r}\circ A_{j,s})+\mathfrak R(f_r,P_{s-r}A_{i,r},A_{j,s}).$$
Splitting $f_r=f_s+(f_r-f_s)$ and integrating first in $r$ gives \eqref{eq:integrated-master-identity}.  The estimates above define $\mathscr J_{ij}$ without any density argument.  For smooth inputs, the preceding algebraic computation proves \eqref{eq:integrated-master-identity}.
\end{proof}

The finite-time resonant coordinates may blow up at $t=0$.  The following weighted Schauder estimate is used whenever such a coordinate occurs under the outer integral.

\begin{lemma}\label{lem:weighted-forcing}
Let $\rho<\beta\leq\gamma$ and $0\leq\omega<1$ satisfy $
\rho\leq\beta-2\mathfrak d$, $\gamma<\rho+2$ and $\Theta(\rho,\omega):=1-\omega-\frac{\beta-\rho}{2}>0$.
If $G:(0,T]\to\bC^\rho$ satisfies
\begin{equation*}
M_{\rho,\omega}(G):=\sup_{0<t\leq T}t^\omega\|G_t\|_\rho<\infty,
\end{equation*}
then $I_0(G)\in\mathcal L_{T,0}^{\beta,\gamma}$ and
\begin{equation}\label{eq:weighted-forcing-estimate}
\|I_0(G)\|_{\mathcal L_T^{\beta,\gamma}}\lesssim T^{\Theta(\rho,\omega)}M_{\rho,\omega}(G).
\end{equation}
If $F\in\mathcal W_T^{\omega,\zeta;\rho}$ and $f\in\mathcal L_T^{\beta,\sigma}$ with $\beta+\rho>0$, then $f_tF_t$ is defined for $t>0$ and
\begin{equation}\label{eq:weighted-product-forcing-estimate}
\|I_0(fF)\|_{\mathcal L_T^{\beta,\gamma}}\lesssim T^{\Theta(\rho,\omega)}\|f\|_{\mathcal L_T^{\beta,\sigma}}\|F\|_{\mathcal W_T^{\omega,\zeta;\rho}}.
\end{equation}
Both estimates have locally Lipschitz difference versions.
\end{lemma}

\begin{proof}
The pointwise Schauder estimate show that
\begin{align}\label{eq:weighted-forcing-pointwise-proof}
\|(I_0(G))_t\|_\beta\lesssim M_{\rho,\omega}(G) t^{\Theta(\rho,\omega)},\ t^{\ff{\g-\b}2}\|(I_0(G))_t\|_\gamma\lesssim M_{\rho,\omega}(G) t^{\Theta(\rho,\omega)}.
\end{align}

Let $0\leq s<t\leq T$ and set $h:=t-s$.  We use
\begin{equation*}
(I_0(G))_t-(I_0(G))_s=(P_h-\mathrm{Id})(I_0(G))_s+\int_s^tP_{t-u}G_u\dif u.
\end{equation*}
The heat-flow increment estimate and \eqref{eq:weighted-forcing-pointwise-proof} give
\begin{align*}
h^{-\mathfrak d}\|(P_h-\mathrm{Id})(I_0(G))_s\|_{\beta-2\mathfrak d}\lesssim M_{\rho,\omega}(G) T^{\Theta(\rho,\omega)},\ s^{\ff{\g-\b}2} h^{-\mathfrak d}\|(P_h-\mathrm{Id})(I_0(G))_s\|_{\gamma-2\mathfrak d}\lesssim M_{\rho,\omega}(G) T^{\Theta(\rho,\omega)}.
\end{align*}
The second estimate is used only for $s>0$.

It remains to treat the short integral.  The condition $\rho\leq\beta-2\mathfrak d$ ensures that the required smoothing exponents are nonnegative.  If $h\leq s$, then
\begin{align*}
h^{-\mathfrak d}\Big\|\int_s^tP_{t-u}G_u\dif u\Big\|_{\beta-2\mathfrak d}&\lesssim M_{\rho,\omega}(G) s^{-\omega}h^{1-\ff{\beta-\rho}2} \lesssim M_{\rho,\omega}(G) T^{\Theta(\rho,\omega)},\no\\
s^{\ff{\g-\b}2} h^{-\mathfrak d}\Big\|\int_s^tP_{t-u}G_u\dif u\Big\|_{\gamma-2\mathfrak d}&\lesssim M_{\rho,\omega}(G) s^{\ff{\g-\b}2-\omega}h^{1-\ff{\gamma-\rho}2}\lesssim M_{\rho,\omega}(G) T^{\Theta(\rho,\omega)}.
\end{align*}
If $h>s$, then $t<2h$, and a direct beta-integral gives
\begin{align}
h^{-\mathfrak d}\Big\|\int_s^tP_{t-u}G_u\dif u\Big\|_{\beta-2\mathfrak d}\lesssim M_{\rho,\omega}(G) h^{\Theta(\rho,\omega)},\ s^{\ff{\g-\b}2} h^{-\mathfrak d}\Big\|\int_s^tP_{t-u}G_u\dif u\Big\|_{\gamma-2\mathfrak d}\lesssim M_{\rho,\omega}(G) h^{\Theta(\rho,\omega)}.\no
\end{align}
The first estimate also covers $s=0$, while the second is only needed for $s>0$.  Combining the estimates above proves \eqref{eq:weighted-forcing-estimate}.

If $\beta+\rho>0$, the Bony estimate yields
\begin{equation*}
M_{\rho,\omega}(fF)\lesssim\|f\|_{\mathcal L_T^{\beta,\sigma}}\|F\|_{\mathcal W_T^{\omega,\zeta;\rho}},
\end{equation*}
and \eqref{eq:weighted-product-forcing-estimate} follows directly. The difference estimates follow by telescoping the linear and bilinear expressions.
\end{proof}

For each weighted coordinate $\tau$, denote its spatial regularity by
$\rho_\tau$.  In the order used in Section~\ref{subsec:admissible-enhanced-data},
\begin{align*}
\rho_{31}=\rho_{22}:=&-\delta,&\rho_{32}:=&-\frac12-\delta,\\
\rho_{51}=\rho_{24}=\rho_{42}=\rho_{33}:=&\alpha-1-\delta,&\rho_{53}=\rho_{44}:=&2\alpha-2-\delta,\\
\rho_{52}=\rho_{34}:=&\alpha-\frac32-\delta,&\rho_{54}:=&2\alpha-\frac52-\delta.
\end{align*}
All the exponents above are strictly negative.  Hence $\rho_\tau\leq\beta-2\mathfrak d$ for every weighted coordinate. Let $\mathcal R_0$ be the set of these indices and define
\begin{equation*}
\theta_{\mathrm{mod}}:=\min_{\tau\in\mathcal R_0}\big(1-\omega_\tau-\frac{\beta-\rho_\tau}{2}\big).
\end{equation*}
Then for every $\tau\in\cR_0$,
\begin{equation}\label{eq:weighted-model-time-margin-positive}
\theta_{\mathrm{mod}}>0,\ \gamma<\rho_\tau+2.
\end{equation}
The critical coordinate is $\mathbf R_{54}^0$.  Indeed,
\begin{align*}
1-\omega_{54}-\frac{\beta-\rho_{54}}{2}=\frac{4\alpha-3-\beta-3\delta-2\zeta-2\kappa_0}{2}>\frac{3\alpha-\frac52-\delta-2\zeta-2\kappa_0}{2}>0.
\end{align*}
All other coordinates have a larger margin.

Define the normalized diagonal model coordinates by
\begin{equation*}
\mathbf M_{22}:=\mathbf X_{22}^0,\ \mathbf M_{42}:=\mathbf R_{42}^0,\ \mathbf M_{24}:=\mathbf R_{24}^0,\ \mathbf M_{44}:=\frac9{25}\mathbf R_{44}^0.
\end{equation*}
For $i,j\in\{2,4\}$, set
\begin{equation*}
\mathscr D_{ij}^{\mathbb Z}(f):=I_0(f\mathbf M_{ij})+\mathscr J_{ij}(f).
\end{equation*}
The first term is interpreted by Lemma~\ref{lem:weighted-forcing} and the second by Lemma~\ref{lem:integrated-diagonal-estimate}, which implies the following result.

\begin{lemma}\label{lem:renormalized-diagonal-operators}
For every $i,j\in\{2,4\}$, $\mathscr D_{ij}^{\mathbb Z}:\mathcal L_T^{\beta,\sigma}\to\mathcal L_{T,0}^{\beta,\gamma}$ is jointly continuous in $f$ and $\mathbb Z^0$.  On bounded sets,
\begin{equation*}
\|\mathscr D_{ij}^{\mathbb Z}(f)\|_{\mathcal L_T^{\beta,\gamma}}\lesssim_RT^{\theta_{\mathrm D}}\|f\|_{\mathcal L_T^{\beta,\sigma}},
\end{equation*}
where $\theta_{\mathrm D}:=\min\{\theta_{\mathrm J},\theta_{\mathrm{mod}}\}>0$. The corresponding locally Lipschitz difference estimate holds.

If $\mathbb Z=\mathscr L_T^{\mathrm{stat}}(X;\lambda,\mu,\mathbf c)$, $\mathbf c=(c_1,c_2,d,b)$ is a smooth canonical lift, then
\begin{align}\label{eq:renormalized-diagonal-smooth-identities}
\begin{split}
\mathscr D_{22}^{\mathbb Z}(f)=&I_0(I_0(f\prec A_2)\circ A_2)-c_2I_0(f),\\
\mathscr D_{42}^{\mathbb Z}(f)=&I_0(I_0(f\prec A_4)\circ A_2),\\
\mathscr D_{24}^{\mathbb Z}(f)=&I_0(I_0(f\prec A_2)\circ A_4),\\
\mathscr D_{44}^{\mathbb Z}(f)=&I_0(I_0(f\prec A_4)\circ A_4)-\frac9{25}bI_0(f).
\end{split}
\end{align}
\end{lemma}

Let $v_0\in\bC^\beta$ and $\mathbf v=(v,v')\in\mathcal D_T^{\beta,\gamma}(\mathbb Z;v_0)$.  Set
\begin{equation*}
H_{\mathbf v}:=P_\cdot v_0+v^\sharp,\ \Psi:=\mathbf Y_5^0+\lambda\mathbf X_3^0,\ W:=v+\Psi.
\end{equation*}
By \eqref{eq:controlled-remainder-definition}, we obtain
\begin{equation}\label{eq:W-controlled-decomposition}
W=H_{\mathbf v}-5I_0(v'\prec A_4)-3\lambda I_0(v'\prec A_2).
\end{equation}
Choose $0<\kappa_W<\min\{\a-\ff12-2\de,2\a-\ff32-2\de\}$ and set $\rho_W:=\alpha-\delta-\kappa_W$. Then
\begin{equation}\label{eq:W-regularity-margins}
\beta<\rho_W<\gamma,\ \rho_W-\frac12-\delta>0,\ \rho_W+\alpha-\frac32-\delta>0.
\end{equation}

\begin{lemma}\label{lem:W-improved-regularity}
It holds that $W\in\mathcal L_T^{\beta,\rho_W}$. Moreover, on bounded sets of data,
\begin{equation*}
\|W\|_{\mathcal L_T^{\beta,\rho_W}}\lesssim_R1+\|\mathbf v\|_{\mathcal D_T^{\beta,\gamma}(\mathbb Z;v_0)}.
\end{equation*}
The corresponding difference estimate holds.  In particular, $W\circ X$ and $W\circ Y_3$ are classically defined for every positive time.
\end{lemma}

\begin{proof}
The term $H_{\mathbf v}$ belongs to $\mathcal L_T^{\beta,\gamma}$ and hence to $\mathcal L_T^{\beta,\rho_W}$.  Since $\rho_W<r_i+2$ for $i\in\{2,4\}$, the proof of Lemma~\ref{lem:controlled-integrated-paraproduct} applies with target regularity $\rho_W$ to the two integrated paraproducts in \eqref{eq:W-controlled-decomposition}.  This proves the estimate.  The last statement follows from \eqref{eq:W-regularity-margins}.
\end{proof}

Define the two principal product remainders by
\begin{align}\label{eq:P12-sharp-definition}
\mathcal P_{\mathbb Z}^{12,\sharp}(\mathbf v):=3\lambda I_0(A_2\prec v+H_{\mathbf v}\circ A_2)-3I_0(\mathbf R_{52}^0)-3\lambda^2I_0(\mathbf X_{32}^0)-15\lambda\,\mathscr D_{42}^{\mathbb Z}(v')-9\lambda^2\mathscr D_{22}^{\mathbb Z}(v'),
\end{align}
\begin{align}\label{eq:P14-sharp-definition}
\mathcal P_{\mathbb Z}^{14,\sharp}(\mathbf v):=5I_0(A_4\prec v+H_{\mathbf v}\circ A_4)-3I_0(\mathbf R_{54}^0)-3I_0(\mathbf R_{34}^0)-15\lambda\mathscr D_{24}^{\mathbb Z}(v')-25\mathscr D_{44}^{\mathbb Z}(v').
\end{align}
Set
\begin{align}\label{eq:P12-P14-definition}
\mathcal P_{\mathbb Z}^{12}(\mathbf v):=3\lambda I_0(v\prec A_2)+\mathcal P_{\mathbb Z}^{12,\sharp}(\mathbf v),\ \mathcal P_{\mathbb Z}^{14}(\mathbf v):=5I_0(v\prec A_4)+\mathcal P_{\mathbb Z}^{14,\sharp}(\mathbf v).
\end{align}

We next define the quadratic blocks. By Lemma~\ref{lem:W-improved-regularity}, $W_t\in\bC^{\rho_W}$ for every $t>0$.
Since $X_t\in\bC^{-\frac12-\delta}$ and $Y_{3,t}\in\bC^{\alpha-\frac32-\delta}$, \eqref{eq:W-regularity-margins} and the resonant-product estimate in Lemma \ref{bony} show that $W\circ X$ and $W\circ Y_3$ are well defined on $(0,T]$. We define $\circ_{\mathbb Z}$ on $\{W,\Psi\}\times\{X,Y_3\}$ by
\begin{align}\label{eq:prescribed-Psi-resonances}
\Psi\circ_{\mathbb Z}X:=\lambda^{-1}\mathbf R_{51}^0+\lambda\mathbf X_{31}^0,\ W\circ_{\mathbb Z}X:=W\circ X,\ \Psi\circ_{\mathbb Z}Y_3:=\mathbf R_{53}^0+\frac3{10}\mathbf R_{33}^0,\ W\circ_{\mathbb Z}Y_3:=W\circ Y_3.
\end{align}
Here $\circ_{\mathbb Z}$ denotes the possibly renormalized resonant product encoded by the enhanced datum $\mathbb Z$, which agrees with the usual Bony resonant product when no renormalization is needed. For $B\in\{X,Y_3\}$ and $(F,G)\in\{(W,\Psi),(\Psi,\Psi)\}$, set
\begin{align}\label{eq:symmetric-triple-product}
\begin{split}
\mathfrak T_B^{\mathbb Z}(F,G):=&(F\circ G)B+(F\prec G+G\prec F)\prec B+B\prec(F\prec G+G\prec F)\\
&+F(G\circ_{\mathbb Z}B)+G(F\circ_{\mathbb Z}B)+\mathfrak R(F,G,B)+\mathfrak R(G,F,B).
\end{split}
\end{align}
Here $\mathfrak R(W,\Psi,B)$ and $\mathfrak R(\Psi,\Psi,B)$ are given by Lemma~\ref{commutatores}. Since the regularities of $W$ and $B$ sum to a positive number, we define directly $\mathfrak R(\Psi,W,B):=(\Psi\prec W)\circ B-\Psi(W\circ B)$. Thus $\mathfrak T_B^{\mathbb Z}(F,G)$ is an extension of the ordinary triple product $FGB$: if $F,G,B$ are smooth, $F\circ_{\mathbb Z}B=F\circ B$ and $G\circ_{\mathbb Z}B=G\circ B$, then $\mathfrak T_B^{\mathbb Z}(F,G)=FGB$.

The inequalities
\begin{align*}
\rho_W-\frac12-\delta>0,\ \rho_W+\alpha-\frac32-\delta>0,\ 2\alpha-\frac32-3\delta>0,\ 3\alpha-\frac52-3\delta&>0
\end{align*}
ensure that all ordinary resonant products and commutators arising in the four cases above are well defined. The terms involving $\Psi\circ_{\mathbb Z}B$ occur only under the outer operator $I_0$; there, $I_0$ is applied to each summand in \eqref{eq:symmetric-triple-product}, and Lemma~\ref{lem:weighted-forcing} applies.

Define
\begin{align}\label{eq:P21-definition}
\mathcal P_{\mathbb Z}^{21}(\mathbf v):=3\lambda I_0(W^2X-2\mathfrak T_X^{\mathbb Z}(W,\Psi)+\mathfrak T_X^{\mathbb Z}(\Psi,\Psi)),
\end{align}
\begin{align}\label{eq:P23-definition}
\mathcal P_{\mathbb Z}^{23}(\mathbf v):=10I_0(W^2Y_3-2\mathfrak T_{Y_3}^{\mathbb Z}(W,\Psi)+\mathfrak T_{Y_3}^{\mathbb Z}(\Psi,\Psi)).
\end{align}

Finally, set
\begin{equation*}
v\circ_{\mathbb Z}X:=W\circ X-\Psi\circ_{\mathbb Z}X
\end{equation*}
and introduce the paralinearization remainder
\begin{equation*}
\operatorname{Rem}_4(v):=v^4-4v^3\prec v.
\end{equation*}
The last singular block is
\begin{align}\label{eq:P41-definition}
\mathcal P_{\mathbb Z}^{41}(\mathbf v):=5\mu I_0(v^4\prec X+X\prec v^4+4v^3(v\circ_{\mathbb Z}X)+4\mathfrak R(v^3,v,X)+\operatorname{Rem}_4(v)\circ X).
\end{align}

\begin{lemma}\label{lem:smooth-consistency-product-blocks}
Let $\mathbb Z=\mathscr L_T^{\mathrm{stat}}(X;\lambda,\mu,\mathbf c)$, $\mathbf c=(c_1,c_2,d,b)$ be a smooth canonical lift, and $\mathbf v=(v,v')$ be a smooth controlled pair. Then
$$\mathcal P_{\mathbb Z}^{12}(v,v')=3\lambda I_0(vX^{\diamond2})+9\lambda^2c_2I_0(X+v'),$$
$$\mathcal P_{\mathbb Z}^{14}(v,v')=5\mu I_0(vX^{\diamond4})+6dI_0(X)+9bI_0(X+v'),$$
$$\mathcal P_{\mathbb Z}^{21}(v,v')=3\lambda I_0(v^2X),$$
$$\mathcal P_{\mathbb Z}^{23}(v,v')=10\mu I_0(v^2X^{\diamond3})+6dI_0(v),$$
$$\mathcal P_{\mathbb Z}^{41}(v,v')=5\mu I_0(v^4X).$$
\end{lemma}

\begin{proof}
The controlled reconstruction can be written as
\begin{equation*}
v=H_{\mathbf v}-\mathbf Y_5^0-\lambda\mathbf X_3^0-5I_0(v'\prec A_4)-3\lambda I_0(v'\prec A_2).
\end{equation*}
Taking the resonant product with $A_2$, applying the outer integration, and using \eqref{eq:renormalized-diagonal-smooth-identities} gives
\begin{align*}
I_0(v\circ A_2)=I_0(H_{\mathbf v}\circ A_2)-I_0(\mathbf Y_5^0\circ A_2)-\lambda I_0(\mathbf X_3^0\circ A_2)-5\mathscr D_{42}^{\mathbb Z}(v')-3\lambda\mathscr D_{22}^{\mathbb Z}(v')-3\lambda c_2I_0(v').
\end{align*}
On a smooth canonical lift,
\begin{align*}
\lambda(\mathbf Y_5^0\circ A_2)=\mathbf R_{52}^0,\ \mathbf X_3^0\circ A_2=\mathbf X_{32}^0+3c_2X.
\end{align*}
Adding the two paraproduct pieces of $vA_2$ and comparing with \eqref{eq:P12-sharp-definition}--\eqref{eq:P12-P14-definition} proves
\begin{equation*}
\mathcal P_{\mathbb Z}^{12}(v,v')=3\lambda I_0(vX^{\diamond2})+9\lambda^2c_2I_0(X+v').
\end{equation*}

The same calculation with $A_4$ gives
\begin{align*}
I_0(v\circ A_4)=I_0(H_{\mathbf v}\circ A_4)-I_0(\mathbf Y_5^0\circ A_4)-\lambda I_0(\mathbf X_3^0\circ A_4)-5\mathscr D_{44}^{\mathbb Z}(v')-3\lambda\mathscr D_{24}^{\mathbb Z}(v')-\frac95bI_0(v').
\end{align*}
The model identities are
\begin{align*}
\mathbf Y_5^0\circ A_4=\frac35\mathbf R_{54}^0+\frac95bX,\ \lambda\mathbf X_3^0\circ A_4=\frac35\mathbf R_{34}^0+\frac65dX.
\end{align*}
This gives
\begin{equation*}
\mathcal P_{\mathbb Z}^{14}(v,v')=5\mu I_0(vX^{\diamond4})+6dI_0(X)+9bI_0(X+v').
\end{equation*}

On a smooth canonical lift, $\Psi\circ_{\mathbb Z}X=\Psi\circ X$. Since $v=W-\Psi$,
\begin{equation*}
W^2X-2\mathfrak T_X^{\mathbb Z}(W,\Psi)+\mathfrak T_X^{\mathbb Z}(\Psi,\Psi)=v^2X.
\end{equation*}
Moreover, $v\circ_{\mathbb Z}X=v\circ X$, and
\begin{equation*}
v^4\circ X=4v^3(v\circ X)+4R(v^3,v,X)+\operatorname{Rem}_4(v)\circ X.
\end{equation*}
This proves the identities for $\mathcal P_{\mathbb Z}^{21}$ and $\mathcal P_{\mathbb Z}^{41}$.

For the $Y_3$ block,
\begin{equation*}
\Psi\circ Y_3=\Psi\circ_{\mathbb Z}Y_3+\frac3{10}d.
\end{equation*}
Consequently,
\begin{equation*}
W^2Y_3-2\mathfrak T_{Y_3}^{\mathbb Z}(W,\Psi)+\mathfrak T_{Y_3}^{\mathbb Z}(\Psi,\Psi)=v^2Y_3+\frac35dv.
\end{equation*}
Since $Y_3=\mu X^{\diamond3}$, this proves the identity for $\mathcal P_{\mathbb Z}^{23}$.
\end{proof}

For brevity, write
\begin{align}\label{eq:product-block-vector}
\mathbf P_{\mathbb Z}^{\sharp}(\mathbf v):=\bigl(\mathcal P_{\mathbb Z}^{12,\sharp}(\mathbf v),\mathcal P_{\mathbb Z}^{14,\sharp}(\mathbf v),\mathcal P_{\mathbb Z}^{21}(\mathbf v),\mathcal P_{\mathbb Z}^{23}(\mathbf v),\mathcal P_{\mathbb Z}^{41}(\mathbf v)\bigr).
\end{align}

\begin{proposition}\label{prop:renormalized-product-blocks}
Let $v_0\in\bC^\beta$ and $\mathbf v\in\mathcal D_T^{\beta,\gamma}(\mathbb Z;v_0)$.  The blocks \eqref{eq:P12-sharp-definition}--\eqref{eq:P41-definition}, initially defined on smooth canonical lifts, admit unique jointly continuous extensions to every admissible enhanced datum.  They satisfy
\begin{equation*}
\mathcal P_{\mathbb Z}^{12}(\mathbf v),\mathcal P_{\mathbb Z}^{14}(\mathbf v)\in\mathcal L_{T,0}^{\beta,\sigma},\ \mathbf P_{\mathbb Z}^{\sharp}(\mathbf v)\in(\mathcal L_{T,0}^{\beta,\gamma})^5.
\end{equation*}
For every $R>0$, there exist $C_{R,\lambda_*}<\infty$ and $\theta_1>0$ such that, whenever $\|\mathbb Z^0\|_{\mathfrak X_{T,\alpha}^0}+\|v_0\|_\beta\leq R$ and $\lambda\geq\lambda_*$, one has
\begin{equation}\label{eq:renormalized-product-block-bound}
\|\mathbf P_{\mathbb Z}^{\sharp}(\mathbf v)\|_{(\mathcal L_T^{\beta,\gamma})^5}\leq C_{R,\lambda_*}T^{\theta_1}\big(1+\|\mathbf v\|_{\mathcal D_T^{\beta,\gamma}(\mathbb Z;v_0)}\big)^4.
\end{equation}
\end{proposition}

\begin{proof}
We first treat the two principal blocks.  For $i\in\{2,4\}$, the paraproduct estimate gives
\begin{equation}\label{eq:principal-left-paraproduct-bound}
\sup_{0<t\leq T}
\|A_{i,t}\prec v_t\|_{\beta+r_i}\lesssim\|A_i\|_{C_T^{\delta',r_i}}\|v\|_{\mathcal L_T^{\beta,\sigma}}.
\end{equation}
Since $\beta+r_i\leq\beta-2\mathfrak d$ and $\gamma<\beta+r_i+2$, Lemma~\ref{lem:weighted-forcing}, with $\rho=\beta+r_i$ and $\omega=0$, yields the positive time power 
$$\theta_i:=1-\frac{\beta-(\beta+r_i)}2=1+\frac{r_i}{2}>0.$$

The resonance $H_{\mathbf v}\circ A_i$ is estimated separately.  Choose
$\varepsilon_i>0$ so small that $0<\varepsilon_i<\min\{\beta-2\mathfrak d,\gamma+r_i\}$, $a_i:=-r_i+\varepsilon_i<\gamma$, and set $\omega_i:=\frac{a_i-\beta}{2}$. Hence $0\leq\omega_i<1$, $\varepsilon_i<\beta$ and $\gamma<\varepsilon_i+2$. The path-space embeddings give
\begin{equation*}
\sup_{0<t\leq T}t^{\omega_i}\|H_{\mathbf v,t}\|_{a_i}\lesssim\|v_0\|_\beta+\|v^\sharp\|_{\mathcal L_T^{\beta,\gamma}}.
\end{equation*}
Since $a_i+r_i=\varepsilon_i>0$, the resonant-product estimate implies
\begin{equation}\label{eq:H-Ai-weighted-bound}
\sup_{0<t\leq T}t^{\omega_i}\|H_{\mathbf v,t}\circ A_{i,t}\|_{\varepsilon_i}\lesssim\|A_i\|_{C_T^{\delta',r_i}}\big(\|v_0\|_\beta+\|v^\sharp\|_{\mathcal L_T^{\beta,\gamma}}\big).
\end{equation}
Moreover, $1-\omega_i-\frac{\beta-\varepsilon_i}{2}=1+\frac{r_i}{2}=\theta_i$. Thus Lemma~\ref{lem:weighted-forcing}, now with $\rho=\varepsilon_i$ and $\omega=\omega_i$, controls $I_0(H_{\mathbf v}\circ A_i)$ in $\mathcal L_{T,0}^{\beta,\gamma}$.  This proves the required estimates for $A_i\prec v$ and $H_{\mathbf v}\circ A_i$.  The diagonal terms are controlled by Lemma~\ref{lem:renormalized-diagonal-operators}, while the explicit weighted coordinates in \eqref{eq:P12-sharp-definition} and \eqref{eq:P14-sharp-definition} are controlled by Lemma~\ref{lem:weighted-forcing}.  Finally, Lemma~\ref{lem:controlled-integrated-paraproduct} applied to $I_0(v\prec A_i)$ shows that the two full blocks $\mathcal P_{\mathbb Z}^{12}$ and $\mathcal P_{\mathbb Z}^{14}$ belong to $\mathcal L_{T,0}^{\beta,\sigma}$.

We next turn to the quadratic blocks.  Set
\begin{equation*}
q_W:=\frac{\rho_W-\beta}{2},\ r_X:=-\frac12-\delta,\ r_3:=\alpha-\frac32-\delta.
\end{equation*}
By Lemma~\ref{lem:W-improved-regularity}, on bounded sets,
\begin{equation*}
\sup_{0<t\leq T}\left(\|W_t\|_\beta+t^{q_W}\|W_t\|_{\rho_W}\right)\lesssim_R 1+\|\mathbf v\|_{\mathcal D_T^{\beta,\gamma}(\mathbb Z;v_0)}.
\end{equation*}
Using \eqref{eq:W-regularity-margins}, the regularities of the two integrated Wick-power  components of $\Psi$, and the multilinear paraproduct and commutator estimates, every term in the $X$-quadratic block which does not contain a prescribed model resonance satisfies
\begin{equation}\label{eq:X-quadratic-regular-part-bound}
\sup_{0<t\leq T}t^{q_W}\|F_t\|_{r_X}\lesssim_R\big(1+\|\mathbf v\|_{\mathcal D_T^{\beta,\gamma}(\mathbb Z;v_0)}\big)^2,
\end{equation}
while every such term in the $Y_3$-quadratic block satisfies
\begin{equation}\label{eq:Y3-quadratic-regular-part-bound}
\sup_{0<t\leq T}t^{q_W}\|F_t\|_{r_3}\lesssim_R\big(1+\|\mathbf v\|_{\mathcal D_T^{\beta,\gamma}(\mathbb Z;v_0)}\big)^2.
\end{equation}
Here $F$ denotes any one of the corresponding nonprescribed summands in \eqref{eq:P21-definition} or \eqref{eq:P23-definition}.  Notice that
\begin{equation*}
q_W<1,\ r_X,r_3\leq\beta-2\mathfrak d,\  \gamma<r_X+2,\ \gamma<r_3+2.
\end{equation*}
The two time powers furnished by Lemma~\ref{lem:weighted-forcing} are
\begin{align*}
\theta_X:=1-q_W-\frac{\beta-r_X}{2}=\frac34-\frac\alpha2+\frac{\kappa_W}{2}>0,\ \theta_3:=1-q_W-\frac{\beta-r_3}{2}=\frac14+\frac{\kappa_W}{2}>0.
\end{align*}

Every remaining quadratic term contains one of the prescribed resonances in \eqref{eq:prescribed-Psi-resonances}.  Its other factor is either $W$ or $\Psi$, both of which are bounded in $\mathcal L_T^{\beta,\sigma}$.  Moreover, $\beta+\rho_\tau>0$ for all prescribed coordinates occurring here. Consequently, the weighted product estimate \eqref{eq:weighted-product-forcing-estimate} applies term by term, and \eqref{eq:weighted-model-time-margin-positive} gives the time power $\theta_{\mathrm{mod}}$.  The coefficient $\lambda^{-1}$ is uniformly bounded for $\lambda\geq\lambda_*$.

It remains to check the fourth-order block.  After separating the prescribed part of $v\circ_{\mathbb Z}X$, the terms $v^4\prec X$, $X\prec v^4$, and $v^3(W\circ X)$ satisfy the same weighted $\bC^{r_X}$ estimate as in \eqref{eq:X-quadratic-regular-part-bound}, with the right-hand side replaced by a polynomial of degree at most four.  The term containing $v^3(\Psi\circ_{\mathbb Z}X)$ is covered by \eqref{eq:weighted-product-forcing-estimate}.  For the two remaining terms, the commutator and paralinearization estimates give
\begin{align}\label{eq:P41-remainder-weighted-bound}
t^{\sigma-\beta}(\|\mathfrak R(v_t^3,v_t,X_t)\|_{2\sigma+r_X}+\|\operatorname{Rem}_4(v_t)\circ X_t\|_{2\sigma+r_X})\lesssim_R\big(1+\|\mathbf v\|_{\mathcal D_T^{\beta,\gamma}(\mathbb Z;v_0)}\big)^4.
\end{align}
Since $2\sigma+r_X=2\alpha-\frac32-5\delta>0$, these terms may in particular be measured in $\bC^{r_X}$.  Applying Lemma~\ref{lem:weighted-forcing} with $\rho=r_X$ and $\omega=\sigma-\beta$ gives
\begin{equation*}
\theta_{41}:=1-(\sigma-\beta)-\frac{\beta-r_X}{2}=\frac54-\alpha+\frac\beta2+\frac{3\delta}{2}>0.
\end{equation*}

Taking
\begin{equation*}
\theta_1:=\min\left\{\theta_{\mathrm D},\theta_{\mathrm{mod}},\theta_2,\theta_4,\theta_X,\theta_3,\theta_{41}\right\}>0
\end{equation*}
proves \eqref{eq:renormalized-product-block-bound}.  All expressions are at most quartic in the controlled pair.  The preceding analytic estimates define every displayed term directly for an admissible datum and an arbitrary controlled pair.  Their multilinear difference versions imply joint continuity, and the definitions agree with the original formulas on smooth canonical lifts.  This gives the claimed extensions.
\end{proof}

\begin{proposition}\label{prop:renormalized-product-blocks-stability}
Let $\mathbf v\in\mathcal D_T^{\beta,\gamma}(\mathbb Z;v_0)$  and $\overline{\mathbf v}\in\mathcal D_T^{\beta,\gamma}(\overline{\mathbb Z};\bar v_0)$. Assume that $\lambda\wedge\bar\lambda\geq\lambda_*$ and
\begin{align*}
\|\mathbb Z^0\|_{\mathfrak X_{T,\alpha}^0}+\|\overline{\mathbb Z}^0\|_{\mathfrak X_{T,\alpha}^0}+\|v_0\|_\beta+\|\bar v_0\|_\beta+\|\mathbf v\|_{\mathcal D_T}+\|\overline{\mathbf v}\|_{\mathcal D_T}\leq R.
\end{align*}
Then
\begin{align}\label{eq:renormalized-product-block-stability}
\|\mathbf P_{\mathbb Z}^{\sharp}(\mathbf v)-\mathbf P_{\overline{\mathbb Z}}^{\sharp}(\overline{\mathbf v})\|_{(\mathcal L_T^{\beta,\gamma})^5}\leq C_{R,\lambda_*}\bigl(\|\mathbb Z^0-\overline{\mathbb Z}^0\|_{\mathfrak X_{T,\alpha}^0}+\|v_0-\bar v_0\|_\beta+T^{\theta_1}d_{\mathcal D_T}(\mathbf v,\overline{\mathbf v})\bigr).
\end{align}
On bounded sets, Proposition~\ref{prop:finite-time-realization} allows the first term on the right-hand side to be replaced by $d_{\mathscr X_{T,\alpha}}(\mathbb Z,\overline{\mathbb Z})$.
\end{proposition}

\begin{proof}
Expand every multilinear difference by a telescoping identity.  The estimates \eqref{eq:principal-left-paraproduct-bound} and \eqref{eq:H-Ai-weighted-bound}, together with their bilinear difference versions, give the time powers $\theta_i$ for the two principal blocks.  The diagonal differences are controlled by \eqref{eq:J-difference-estimate} and the difference version of Lemma~\ref{lem:weighted-forcing}.

The difference versions of Lemmas~\ref{lem:controlled-reconstruction} and \ref{lem:W-improved-regularity} give, on the bounded set under consideration,
\begin{align*}
\sup_{0<t\leq T}\left(\|W_t-\overline W_t\|_\beta+t^{q_W}\|W_t-\overline W_t\|_{\rho_W}\right)\lesssim_R\|\mathbb Z^0-\overline{\mathbb Z}^0\|_{\mathfrak X_{T,\alpha}^0}+\|v_0-\bar v_0\|_\beta+d_{\mathcal D_T}(\mathbf v,\overline{\mathbf v}).
\end{align*}
Applying the same telescoping expansion to \eqref{eq:X-quadratic-regular-part-bound}, \eqref{eq:Y3-quadratic-regular-part-bound}, and \eqref{eq:P41-remainder-weighted-bound} yields respectively the positive time powers $\theta_X$, $\theta_3$, and $\theta_{41}$.  Every difference of a prescribed weighted resonance is controlled by the difference version of \eqref{eq:weighted-product-forcing-estimate} and carries the time power $\theta_{\mathrm{mod}}$.  Finally, the map $(\lambda,F)\longmapsto\lambda^{-1}F$ is locally Lipschitz on $[\lambda_*,\infty)\times\mathcal W_T^{\omega_{51},\zeta;\rho_{51}}$.

Thus every difference involving the controlled pair is multiplied by at least $T^{\theta_1}$. Differences of the model and of the initial condition may also carry a positive time power, which can be discarded since $T\leq1$. Taking the minimum of the powers listed in the proof of Proposition~\ref{prop:renormalized-product-blocks} gives \eqref{eq:renormalized-product-block-stability}.
\end{proof}

The remaining mixed term involving $Y_2$ requires no renormalization. Indeed, since $\beta+\alpha-1-\delta>0$, $v_t^3Y_{2,t}\in\mathcal C^{\alpha-1-\delta}$ for every $t>0$. Moreover, $\gamma<\alpha+1-\delta$, so the outer Schauder estimate gives
\begin{equation}\label{eq:v3Y2-estimate}
\|I_0(v^3Y_2)\|_{\mathcal L_T^{\beta,\gamma}}\lesssim T^{\theta_{Y_2}}\|Y_2\|_{C_T^{\delta',\alpha-1-\delta}}\|v\|_{\mathcal L_T^{\beta,\sigma}}^3
\end{equation}
for some $\theta_{Y_2}>0$.  The purely deterministic terms $I_0(v^3)$ and $I_0(v^5)$ are defined in the usual way.

\begin{definition}\label{def:renormalized-mild-map}
For $\mathbf v=(v,v')\in\mathcal D_T^{\beta,\gamma}(\mathbb Z;v_0)$,
define
\begin{align}\label{eq:renormalized-mild-map}
\begin{split}
\mathscr G_{\mathbb Z}(\mathbf v):=&-\mathbf Y_5^0-\lambda\mathbf X_3^0-\mathcal P_{\mathbb Z}^{12}(\mathbf v)-\mathcal P_{\mathbb Z}^{14}(\mathbf v)-\mathcal P_{\mathbb Z}^{21}(\mathbf v)-\mathcal P_{\mathbb Z}^{23}(\mathbf v)\\
&-\lambda I_0(v^3)-10I_0(v^3Y_2)-\mathcal P_{\mathbb Z}^{41}(\mathbf v)-\mu I_0(v^5)+I_0(X+v).
\end{split}
\end{align}
The first component of the fixed-point map is
\begin{equation}\label{eq:Gamma-definition}
\Gamma_{\mathbb Z,v_0}(\mathbf v):=P_\cdot v_0+\mathscr G_{\mathbb Z}(\mathbf v).
\end{equation}
\end{definition}

\begin{remark}
The final term $I_0(X+v)$ is caused by the choice $P_t=\exp(t(\Delta-1))$.  It is not part of the mass renormalization and must be retained in \eqref{eq:renormalized-mild-map}.
\end{remark}

By \eqref{eq:P12-P14-definition}, the mild map has the controlled form
\begin{align}\label{eq:mild-map-controlled-form}
\mathscr G_{\mathbb Z}(\mathbf v)=-\mathbf Y_5^0-\lambda\mathbf X_3^0-5I_0(v\prec Y_4)-3\lambda I_0(v\prec X^{\diamond2})+\mathscr G_{\mathbb Z}^{\sharp}(\mathbf v),
\end{align}
where
\begin{align*}
\mathscr G_{\mathbb Z}^{\sharp}(\mathbf v):=&-\mathcal P_{\mathbb Z}^{12,\sharp}(\mathbf v)-\mathcal P_{\mathbb Z}^{14,\sharp}(\mathbf v)-\mathcal P_{\mathbb Z}^{21}(\mathbf v)-\mathcal P_{\mathbb Z}^{23}(\mathbf v)\\
&-\lambda I_0(v^3)-10I_0(v^3Y_2)-\mathcal P_{\mathbb Z}^{41}(\mathbf v)-\mu I_0(v^5)+I_0(X+v).
\end{align*}
In particular,
\begin{equation*}
\mathscr G_{\mathbb Z}^{\sharp}(\mathbf v)\in\mathcal L_{T,0}^{\beta,\gamma}.
\end{equation*}
Consequently,
\begin{equation*}
\bigl(\Gamma_{\mathbb Z,v_0}(\mathbf v),v\bigr)\in\mathcal D_T^{\beta,\gamma}(\mathbb Z;v_0).
\end{equation*}
The controlled derivative of the output is the first input component $v$.

\begin{proposition}\label{prop:renormalized-mild-map-bounds}
For every $R>0$, there exist $C_{R,\lambda_*}<\infty$ and $\theta_{\mathrm G}>0$ such that
\begin{equation}\label{eq:renormalized-mild-map-bound}
\|\mathscr G_{\mathbb Z}^{\sharp}(\mathbf v)\|_{\mathcal L_T^{\beta,\gamma}}\leq C_{R,\lambda_*}T^{\theta_{\mathrm G}}\big(1+\|\mathbf v\|_{\mathcal D_T^{\beta,\gamma}(\mathbb Z;v_0)}\big)^5
\end{equation}
whenever $\|\mathbb Z^0\|_{\mathfrak X_{T,\alpha}^0}+\|v_0\|_\beta\leq R$ and $\lambda\geq\lambda_*$. If $\overline{\mathbf v}\in\mathcal D_T^{\beta,\gamma}(\overline{\mathbb Z};\bar v_0)$, $\lambda\wedge\bar\lambda\geq\lambda_*$, and $\|\mathbb Z^0\|_{\mathfrak X_{T,\alpha}^0}+\|\overline{\mathbb Z}^0\|_{\mathfrak X_{T,\alpha}^0}+\|v_0\|_\beta+\|\bar v_0\|_\beta+\|\mathbf v\|_{\mathcal D_T}+\|\overline{\mathbf v}\|_{\mathcal D_T}\leq R$, then
\begin{align}\label{eq:renormalized-mild-map-stability}
\|\mathscr G_{\mathbb Z}^{\sharp}(\mathbf v)-\mathscr G_{\overline{\mathbb Z}}^{\sharp}(\overline{\mathbf v})\|_{\mathcal L_T^{\beta,\gamma}}\leq C_{R,\lambda_*}\bigl(\|\mathbb Z^0-\overline{\mathbb Z}^0\|_{\mathfrak X_{T,\alpha}^0}+\|v_0-\bar v_0\|_\beta+T^{\theta_{\mathrm G}}d_{\mathcal D_T}(\mathbf v,\overline{\mathbf v})\bigr).
\end{align}
On bounded sets, the finite-time model difference may be replaced by $d_{\mathscr X_{T,\alpha}}(\mathbb Z,\overline{\mathbb Z})$.
\end{proposition}

\begin{proof}
Propositions~\ref{prop:renormalized-product-blocks} and \ref{prop:renormalized-product-blocks-stability} control the five renormalized product blocks collected in $\mathbf P_{\mathbb Z}^{\sharp}(\mathbf v)$ in \eqref{eq:product-block-vector}. Estimate \eqref{eq:v3Y2-estimate} treats the remaining mixed Wick term.  The ordinary Schauder estimates give positive powers of $T$ for $I_0(v^3)$, $I_0(v^5)$ and $I_0(X+v)$; the relevant inequalities are $\gamma<\beta+2$ and $\gamma<\frac32-\delta$. Lemma~\ref{lem:controlled-reconstruction} expresses the path norm of $v$ in terms of the controlled size and the fixed data.  Taking the minimum of all positive time powers proves \eqref{eq:renormalized-mild-map-bound}. Polynomial differences are expanded by telescoping, and the same estimates, together with \eqref{eq:controlled-reconstruction-different-data}, prove \eqref{eq:renormalized-mild-map-stability}.
\end{proof}

\begin{proposition}\label{prop:smooth-consistency-mild-map}
Let $\mathbb Z=\mathscr L_T^{\mathrm{stat}}(X;\lambda,\mu,\mathbf c)$, $\mathbf c=(c_1,c_2,d,b)$ be a smooth canonical lift.  Define
\begin{align}\label{eq:canonical-mass-constants}
\widetilde C:=10\mu c_1-\lambda,\ A^{\mathrm{loc}}:=3\lambda c_1-15\mu c_1^2,\ C^{\mathrm{can}}:=A^{\mathrm{loc}}-9\lambda^2c_2-9b-6d.
\end{align}
Then, for every smooth $v$,
\begin{align}\label{eq:smooth-consistency-mild-map}
\mathscr G_{\mathbb Z}(v,v)=-I_0\bigl(\mu(X+v)^5-\widetilde C(X+v)^3-(C^{\mathrm{can}}+1)(X+v)\bigr).
\end{align}
Consequently,
\begin{align*}
\Gamma_{\mathbb Z,v_0}(v,v)=P_\cdot v_0-I_0\bigl(\mu(X+v)^5-\widetilde C(X+v)^3-(C^{\mathrm{can}}+1)(X+v)\bigr).
\end{align*}
\end{proposition}

\begin{proof}
At $v'=v$, Lemma~\ref{lem:smooth-consistency-product-blocks} shows that the sum of all mixed product blocks equals
\begin{align*}
I_0(3\lambda vX^{\diamond2}+5\mu vX^{\diamond4}+3\lambda v^2X+10\mu v^2X^{\diamond3}+10\mu v^3X^{\diamond2}+5\mu v^4X)+(9\lambda^2c_2+9b+6d)I_0(X+v).
\end{align*}
Together with $\lambda I_0(X^{\diamond3})$, $\mu I_0(X^{\diamond5})$, $\lambda I_0(v^3)$ and $\mu I_0(v^5)$, the Wick--Taylor identity gives
\begin{align*}
\mu(X+v)^5-\widetilde C(X+v)^3-A^{\mathrm{loc}}(X+v)=&\lambda X^{\diamond3}+\mu X^{\diamond5}+3\lambda vX^{\diamond2}+5\mu vX^{\diamond4}+3\lambda v^2X\\
&+10\mu v^2X^{\diamond3}+\lambda v^3+10\mu v^3X^{\diamond2}+5\mu v^4X+\mu v^5.
\end{align*}
Substitution into \eqref{eq:renormalized-mild-map}, followed by \eqref{eq:canonical-mass-constants}, proves \eqref{eq:smooth-consistency-mild-map}.
\end{proof}

\begin{remark}\label{rem:lambda-zero-enhancement}
The restriction $\lambda\geq\lambda_*$ is harmless when $\lambda_\varepsilon\to\lambda>0$.  A deterministic theory uniform down to $\lambda=0$ requires one additional enhanced coordinate $\mathbf S_{51}:=\mu\cI(X^{\diamond5})\circ X$. Its finite-time realization is
\begin{equation*}
\mathbf S_{51}^0:=\mathbf S_{51}-(P_\cdot\mathbf Y_5(0))\circ X.
\end{equation*}
One may then prescribe
\begin{equation*}
\Psi\circ_{\mathbb Z}X:=\mathbf S_{51}^0+\lambda\mathbf X_{31}^0,
\end{equation*}
without dividing by $\lambda$.  This coordinate has the same spatial regularity and finite-time weight as $\mathbf R_{51}^0$.  Its stochastic convergence follows from the same dyadic argument used for $\lambda\mu\cI(X^{\diamond5})\circ X$.
\end{remark}

\subsection{Local well-posedness, stability and convergence}\label{subsec:local-well-posedness}

We retain all the parameters fixed in Sections~\ref{subsec:admissible-enhanced-data} and \ref{subsec:controlled-distributions}; in particular, \eqref{eq:model-parameter-choice}, \eqref{eq:model-parabolic-time-choice} and \eqref{eq:controlled-spatial-exponents} are assumed throughout this section.  Fix $\lambda_*>0$.

We first record that both the finite-time realization and the mild map are compatible with restriction to a shorter time interval.

\begin{lemma}\label{lem:fixed-point-restriction-consistency}
Let $0<S\leq T\leq1$.  Then
\begin{equation}\label{eq:restriction-finite-time-realization}
\operatorname{Res}_{S,T}(\mathscr Q_0\mathbb Z)=\mathscr Q_0(\operatorname{Res}_{S,T}\mathbb Z)
\end{equation}
for every $\mathbb Z\in\mathscr X_{T,\alpha}$.  Moreover, if $\mathbf v\in\mathcal D_T^{\beta,\gamma}(\mathbb Z;v_0)$, then
\begin{align}\label{eq:mild-map-restriction-consistency}
\operatorname{Res}_{S,T}\mathbf v\in\mathcal D_S^{\beta,\gamma}(\operatorname{Res}_{S,T}\mathbb Z;v_0),\ \operatorname{Res}_{S,T}(\mathcal M_{\mathbb Z,v_0}^{T}(\mathbf v))=\mathcal M_{\operatorname{Res}_{S,T}\mathbb Z,v_0}^{S}(\operatorname{Res}_{S,T}\mathbf v).
\end{align}
Here the superscript indicates the time interval on which the map is defined.
\end{lemma}

\begin{proof}
For a smooth stationary canonical lift, \eqref{eq:restriction-finite-time-realization} follows directly from $(Q_0F)_t=F_t-P_tF_0$ and from the definitions of the finite-time resonant coordinates.  The identity extends to every admissible datum because restriction is continuous in all factors of $\mathfrak X_{T,\alpha}$ and $\mathfrak X_{T,\alpha}^0$.

The controlled remainder is compatible with restriction since every occurrence of $I_0$ is an integral from $0$ to the current time.  The same observation applies to the integrated diagonal operators and to all the renormalized product blocks.  Consequently, \eqref{eq:renormalized-mild-map} commutes with restriction, which proves \eqref{eq:mild-map-restriction-consistency}.
\end{proof}

Then we derive the following local well-posedness and stability result.

\begin{theorem}\label{thm:deterministic-fixed-point}
Let $0<T\leq1$.  For every $R\geq1$, there exist $K_{R,\lambda_*}<\infty$, $0<\tau_{R,\lambda_*}\leq1$, $0<\chi<\min\{\theta_0,\theta_{\mathrm G}\}$ depending only on $R$, $\lambda_*$ and the fixed exponents, with the following property. Set $T_R:=T\wedge\tau_{R,\lambda_*}$. Let $\mathbb Z=(\mathbb X^\Phi,\lambda,\mu,\mathbb V)\in\mathscr X_{T,\alpha}$, $\mathbb Z^0=\mathscr Q_0\mathbb Z$, $v_0\in\bC^\beta$, and assume that
\begin{equation}\label{eq:fixed-point-data-bound}
\|\mathbb Z^0\|_{\mathfrak X_{T,\alpha}^0}+\|v_0\|_\beta\leq R,\ \lambda\geq\lambda_*.
\end{equation}
Then the map $\mathcal M_{\mathbb Z,v_0}(v,v')=(\Gamma_{\mathbb Z,v_0}(v,v'),v)$ has a fixed point
\begin{equation*}
\mathbf v=(v,v)\in\mathcal D_{T_R}^{\beta,\gamma}(\mathbb Z;v_0)
\end{equation*}
satisfying
\begin{equation}\label{eq:fixed-point-solution-bound}
\|v\|_{\mathcal L_{T_R}^{\beta,\sigma}}\leq K_{R,\lambda_*},\ \|v^\sharp\|_{\mathcal L_{T_R}^{\beta,\gamma}}\leq1.
\end{equation}
It is the unique fixed point satisfying \eqref{eq:fixed-point-solution-bound}.  In particular,
\begin{equation}\label{eq:abstract-renormalized-mild-equation}
v=P_\cdot v_0+\mathscr G_{\mathbb Z}(v,v).
\end{equation}

Let $\overline{\mathbb Z}\in\mathscr X_{T,\alpha}$ and $\bar v_0\in\bC^\beta$ satisfy the same assumptions with $\lambda\wedge\bar\lambda\geq\lambda_*$, and let $\overline{\mathbf v}=(\bar v,\bar v)$ be the corresponding fixed point. If both data sets satisfy \eqref{eq:fixed-point-data-bound} with the same $R$, then
\begin{align}\label{eq:deterministic-solution-stability}
d_{\mathcal D_{T_R}}(\mathbf v,\overline{\mathbf v})+\|v-\bar v\|_{\mathcal L_{T_R}^{\beta,\sigma}}\leq C_{R,\lambda_*}\big(\|\mathbb Z^0-\overline{\mathbb Z}^0\|_{\mathfrak X_{T_R,\alpha}^0}+\|v_0-\bar v_0\|_\beta\big).
\end{align}
For pairs associated with different data, the two sharp remainders in $d_{\mathcal D_{T_R}}$ are computed relative to their respective enhanced data and initial conditions.  On bounded subsets of the stationary admissible space, the finite-time model difference on the right-hand side of \eqref{eq:deterministic-solution-stability} may be replaced by $d_{\mathscr X_{T,\alpha}}(\mathbb Z,\overline{\mathbb Z})$.
\end{theorem}

\begin{proof}
For $0<S\leq T$, all data and paths below are restricted to $[0,S]$. Choose $0<\chi<\min\{\theta_0,\theta_{\mathrm G}\}$. For $K>0$, introduce the closed subset
\begin{align*}
\mathfrak B_{K,S}(\mathbb Z;v_0):=\bigl\{\mathbf w=(w,w')\in\mathcal D_S^{\beta,\gamma}(\mathbb Z;v_0):\ \|w'\|_{\mathcal L_S^{\beta,\sigma}}\leq K,\ \|w^\sharp\|_{\mathcal L_S^{\beta,\gamma}}\leq1\bigr\}.
\end{align*}
This set is nonempty.  Indeed, one may take $w'=P_\cdot v_0$, set $w^\sharp=0$, and define $w$ through \eqref{eq:controlled-remainder-definition}.  It is complete for $d_{\mathcal D_S}^{(\chi)}$ by Proposition~\ref{prop:controlled-space-complete}.

Retaining the small time power in \eqref{eq:controlled-reconstruction-estimate}, the data bound \eqref{eq:fixed-point-data-bound} gives
\begin{equation*}
\|w\|_{\mathcal L_S^{\beta,\sigma}}\leq C_R+C\|w^\sharp\|_{\mathcal L_S^{\beta,\gamma}}+C_RS^{\theta_0}\|w'\|_{\mathcal L_S^{\beta,\sigma}}.
\end{equation*}
Choose $K=K_{R,\lambda_*}$ large enough that the first two terms are at most $\ff K2$, and then choose $S$ sufficiently small so that
\begin{equation*}
C_RS^{\theta_0}K\leq\frac K2.
\end{equation*}
Thus the second component of $\mathcal M_{\mathbb Z,v_0}(\mathbf w)$ has $\mathcal L_S^{\beta,\sigma}$ norm at most $K$.

By \eqref{eq:mild-map-controlled-form}, the controlled remainder of the first component of the image is $\mathscr G_{\mathbb Z}^{\sharp}(\mathbf w)$.  Proposition \ref{prop:renormalized-mild-map-bounds} gives
\begin{equation*}
\|\mathscr G_{\mathbb Z}^{\sharp}(\mathbf w)\|_{\mathcal L_S^{\beta,\gamma}}\leq C_{R,K,\lambda_*}S^{\theta_{\mathrm G}}(2+K)^5.
\end{equation*}
After decreasing $S$, this quantity is at most $1$.  Hence $\mathcal M_{\mathbb Z,v_0}$ maps $\mathfrak B_{K,S}(\mathbb Z;v_0)$ into itself.

For two elements of the preceding closed set, \eqref{eq:controlled-derivative-swap-smallness} and \eqref{eq:renormalized-mild-map-stability} yield
\begin{align}\label{eq:fixed-point-contraction}
d_{\mathcal D_S}^{(\chi)}\left(\mathcal M_{\mathbb Z,v_0}(\mathbf w),\mathcal M_{\mathbb Z,v_0}(\overline{\mathbf w})\right)\leq C_{R,K,\lambda_*}\left(S^\chi+S^{\theta_0}+S^{\theta_{\mathrm G}-\chi}\right)d_{\mathcal D_S}^{(\chi)}(\mathbf w,\overline{\mathbf w}).
\end{align}
Here we used
\begin{equation*}
d_{\mathcal D_S}(\mathbf w,\overline{\mathbf w})\leq2S^{-\chi}d_{\mathcal D_S}^{(\chi)}(\mathbf w,\overline{\mathbf w}).
\end{equation*}
We may therefore choose $\tau_{R,\lambda_*}\leq1$ so that the coefficient in \eqref{eq:fixed-point-contraction} is at most $\ff12$ whenever $S\leq\tau_{R,\lambda_*}$, and then set $T_R$. Banach's fixed-point theorem gives a fixed point in $\mathfrak B_{K,T_R}(\mathbb Z;v_0)$.  The form of the second component of $\mathcal M_{\mathbb Z,v_0}$ implies that the fixed point satisfies $v'=v$, and \eqref{eq:abstract-renormalized-mild-equation} follows from \eqref{eq:Gamma-definition}.  The contraction also proves uniqueness in the class \eqref{eq:fixed-point-solution-bound}.

We next prove stability.  Set $D_0:=\|\mathbb Z^0-\overline{\mathbb Z}^0\|_{\mathfrak X_{T_R,\alpha}^0}+\|v_0-\bar v_0\|_\beta$. Since both fixed points satisfy $v'=v$ and $\bar v'=\bar v$, repeating the proof of Lemma~\ref{lem:controlled-reconstruction} while retaining the small time power gives
\begin{align*}
\|v-\bar v\|_{\mathcal L_{T_R}^{\beta,\sigma}}\leq C_{R,\lambda_*}(D_0+\|v^\sharp-\bar v^\sharp\|_{\mathcal L_{T_R}^{\beta,\gamma}}+T_R^{\theta_0}\|v-\bar v\|_{\mathcal L_{T_R}^{\beta,\sigma}}).
\end{align*}
On the other hand, since the two controlled sizes are bounded in terms of $R$ and $\lambda_*$, \eqref{eq:renormalized-mild-map-stability} yields
\begin{align*}
\|v^\sharp-\bar v^\sharp\|_{\mathcal L_{T_R}^{\beta,\gamma}}\leq C_{R,\lambda_*}\big(D_0+T_R^{\theta_{\mathrm G}}(\|v-\bar v\|_{\mathcal L_{T_R}^{\beta,\sigma}}+\|v^\sharp-\bar v^\sharp\|_{\mathcal L_{T_R}^{\beta,\gamma}})\big).
\end{align*}
Here we used the fixed-point identities $v^\sharp=\mathscr G_{\mathbb Z}^{\sharp}(\mathbf v)$ and $\bar v^\sharp=\mathscr G_{\overline{\mathbb Z}}^{\sharp}(\overline{\mathbf v})$. Insert the second estimate into the first one and then add the two inequalities.  After decreasing $\tau_{R,\lambda_*}$ once more if necessary, the terms multiplied by $T_R^{\theta_0}$ and $T_R^{\theta_{\mathrm G}}$ can be absorbed into the resulting left-hand side.  Since
\begin{equation*}
d_{\mathcal D_{T_R}}(\mathbf v,\overline{\mathbf v})=\|v-\bar v\|_{\mathcal L_{T_R}^{\beta,\sigma}}+\|v^\sharp-\bar v^\sharp\|_{\mathcal L_{T_R}^{\beta,\gamma}},
\end{equation*}
this proves \eqref{eq:deterministic-solution-stability} with a constant independent of $T$.  Proposition \ref{prop:finite-time-realization} gives the corresponding estimate in terms of the stationary model distance.

Finally, Lemma~\ref{lem:fixed-point-restriction-consistency} shows that, for $0<S\leq T_R$, the restriction of the constructed solution is the fixed point obtained by applying the same construction directly on $[0,S]$. The same contraction argument applied to any two solutions lying in a common bounded fixed-point class gives uniqueness on a sufficiently short initial interval.  Since every controlled fixed point has finite controlled size, any two such fixed points belong, after restriction to a sufficiently short initial interval, to a common bounded class.  Thus the corresponding local solution germ is unique.  No shifted-interval or global uniqueness assertion is needed here.
\end{proof}

\begin{remark}
The preceding theorem deliberately states uniqueness in the bounded class \eqref{eq:fixed-point-solution-bound}.  The causal continuation observation at the end of the proof gives uniqueness of the associated initial solution germ, but it does not by itself provide a shifted-interval continuation, global existence or global uniqueness statement.
\end{remark}

\begin{corollary}[Deterministic convergence of solutions]
\label{cor:deterministic-solution-convergence}
Let $\mathbb Z_n=(\mathbb X_n^\Phi,\lambda_n,\mu_n,\mathbb V_n)\in\mathscr X_{T,\alpha}$, $\mathbb Z=(\mathbb X^\Phi,\lambda,\mu,\mathbb V)\in\mathscr X_{T,\alpha}$, and let $v_{0,n},v_0\in\bC^\beta$. Set $\mathbb Z_n^0:=\mathscr Q_0\mathbb Z_n$ and  $\mathbb Z^0:=\mathscr Q_0\mathbb Z$. Assume that for some $R<\infty$ and $\lambda_*>0$, $\lambda_n\geq\lambda_*$ for $n\geq1$, $\lambda\geq\lambda_*$, and
\begin{equation}\label{eq:deterministic-data-convergence}
\sup_{n\geq1}(\|\mathbb Z_n^0\|_{\mathfrak X_{T,\alpha}^0}+\|v_{0,n}\|_\beta)+\|\mathbb Z^0\|_{\mathfrak X_{T,\alpha}^0}+\|v_0\|_\beta\leq R,\ \|\mathbb Z_n^0-\mathbb Z^0\|_{\mathfrak X_{T,\alpha}^0}+\|v_{0,n}-v_0\|_\beta\to0.
\end{equation}
Then the corresponding fixed points are defined on the common interval $[0,T_R]$ supplied by Theorem~\ref{thm:deterministic-fixed-point}, and
\begin{align}\label{eq:deterministic-solution-convergence}
d_{\mathcal D_{T_R}}\bigl((v_n,v_n),(v,v)\bigr)+\|v_n-v\|_{\mathcal L_{T_R}^{\beta,\sigma}}\to0.
\end{align}
It is enough in \eqref{eq:deterministic-data-convergence} to assume
\begin{equation*}
d_{\mathscr X_{T,\alpha}}(\mathbb Z_n,\mathbb Z)+\|v_{0,n}-v_0\|_\beta\to0
\end{equation*}
together with a uniform bound on the stationary model norms.

If $X_n$ and $X$ denote the first coordinates of $\mathbb Z_n$ and $\mathbb Z$, respectively, and $u_n:=X_n+v_n$, $u:=X+v$, then
\begin{equation}\label{eq:reconstructed-field-convergence}
u_n\to u\ \text{in }C\bigl([0,T_R];\bC^{-\frac12-\delta}\bigr).
\end{equation}
\end{corollary}

\begin{proof}
Estimate \eqref{eq:deterministic-solution-stability} proves \eqref{eq:deterministic-solution-convergence}.  The stationary formulation follows from Proposition~\ref{prop:finite-time-realization}.  Finally, $X_n\to X$ in $C_T^{\delta',-\frac12-\delta}$, while $\mathcal L_{T_R}^{\beta,\sigma}$ embeds continuously into $C([0,T_R];\bC^{-\frac12-\delta})$.  This proves \eqref{eq:reconstructed-field-convergence}.
\end{proof}

We now apply the preceding deterministic theory to the stochastic approximations. Assume that $\lambda_\varepsilon\to\lambda>0$ and fix $\lambda_*\in(0,\lambda)$. For all sufficiently small $\varepsilon$, one has $\lambda_\varepsilon\geq\lambda_*$. Set $\mu_\varepsilon:=\varepsilon^\alpha$, $\mathbf c_\varepsilon:=(C_\varepsilon^{(1)},C_\varepsilon^{(2)},D_{\varepsilon,\alpha},B_{\varepsilon,\alpha})$. Let $\mathbb Z_\varepsilon=\mathscr L_T^{\mathrm{stat}}(X_\varepsilon;\lambda_\varepsilon,\mu_\varepsilon,\mathbf c_\varepsilon)$, $\mathbb Z=(\mathbb X^\Phi,\lambda,0,\boldsymbol0)$ and write $\mathbb Z_\varepsilon^0:=\mathscr Q_0\mathbb Z_\varepsilon$, $\mathbb Z^0:=\mathscr Q_0\mathbb Z$. Choose $\lambda_*\in(0,\lambda)$.  For all sufficiently small $\varepsilon$, one then has $\lambda_\varepsilon\geq\lambda_*$.  Let $v_{\varepsilon,0}$ and $v_0$ be measurable $\bC^\beta$-valued random variables, and assume the well-prepared initial conditions 
\begin{equation}\label{eq:well-prepared-convergence-assumption}
u_{\varepsilon,0}=X_\varepsilon(0)+v_{\varepsilon,0},\ v_{\varepsilon,0}\to v_0\ \text{in probability in }\bC^\beta,
\end{equation}
and put $u_0=X(0)+v_0$. With this notation, the convergence of the approximating solutions is stated as follows.

\begin{corollary}\label{cor:approximating-solutions-convergence}
For $R\geq1$, define
\begin{align}\label{eq:localized-data-event}
\Omega_{\varepsilon,R}:=\bigl\{\|\mathbb Z_\varepsilon^0\|_{\mathfrak X_{T,\alpha}^0}+\|\mathbb Z^0\|_{\mathfrak X_{T,\alpha}^0}+\|v_{\varepsilon,0}\|_\beta+\|v_0\|_\beta\leq R\bigr\}.
\end{align}
Let $T_R$ be the common existence time from Theorem~\ref{thm:deterministic-fixed-point}. On  $\Omega_{\varepsilon,R}$, let $\mathbf v_\varepsilon=(v_\varepsilon,v_\varepsilon)$ and $\mathbf v=(v,v)$ be the corresponding fixed points driven by $\mathbb Z_\varepsilon$ and $\mathbb Z$, respectively. Then, for every $\eta>0$,
\begin{align}\label{eq:localized-remainder-convergence}
\mathbb P\bigl(\Omega_{\varepsilon,R}\cap\bigl\{d_{\mathcal D_{T_R}}(\mathbf v_\varepsilon,\mathbf v)+\|v_\varepsilon-v\|_{\mathcal L_{T_R}^{\beta,\sigma}}>\eta\bigr\}\bigr)\to0.
\end{align}
Moreover, with $u_\varepsilon:=X_\varepsilon+v_\varepsilon$ and $u:=X+v$, one has
\begin{equation}\label{eq:localized-full-solution-convergence}
\mathbb P\big(\Omega_{\varepsilon,R}\cap\big\{\|u_\varepsilon-u\|_{C([0,T_R];\bC^{-\frac12-\delta})}>\eta\big\}\big)\to0.
\end{equation}
The localization is asymptotically exhaustive:
\begin{equation}\label{eq:localized-events-tight}
\lim_{R\to\infty}\limsup_{\varepsilon\to0}\mathbb P(\Omega_{\varepsilon,R}^{\mathrm c})=0.
\end{equation}
If, in addition, $v_{\varepsilon,0}\to v_0$ in $L^p(\Omega;\bC^\beta)$, extend $v_\varepsilon$ and $v$ by zero on $\Omega_{\varepsilon,R}^{\mathrm c}$. By continuity of the fixed-point map on the bounded data ball, these extensions are measurable. Then
\begin{align*}
\left\|\mathbf 1_{\Omega_{\varepsilon,R}}(v_\varepsilon-v)\right\|_{L^p(\Omega;\mathcal L_{T_R}^{\beta,\sigma})}\leq C_{R,\lambda_*}(\|\mathbb Z_\varepsilon^0-\mathbb Z^0\|_{L^p(\Omega;\mathfrak X_{T,\alpha}^0)}+\|v_{\varepsilon,0}-v_0\|_{L^p(\Omega;\bC^\beta)})\to0.
\end{align*}

On $\Omega_{\varepsilon,R}$ and $[0,T_R]$, $u_\varepsilon$ is the pathwise local mild solution of the original approximating equation
\begin{align}\label{eq:approximating-equation-recovered}
\partial_tu_\varepsilon=\Delta u_\varepsilon-\varepsilon^\alpha u_\varepsilon^5+\xi_\varepsilon+C_\varepsilon u_\varepsilon+\wt C_\varepsilon u_\varepsilon^3,\ u_\varepsilon(0)=u_{\varepsilon,0},
\end{align}
where $\wt C_\varepsilon:=10\varepsilon^\alpha C_\varepsilon^{(1)}-\lambda_\varepsilon,\ C_\varepsilon:=C_\varepsilon^{\mathrm{can}}$. On the same localized interval, the limit $u=X+v$ is the renormalized dynamical $\Phi^4_3(\lambda)$ solution associated with $\mathbb Z$, with initial condition $u_0=X(0)+v_0$.
\end{corollary}

\begin{proof}
Theorem~\ref{thm:deterministic-fixed-point} and \eqref{eq:deterministic-solution-stability} give, on $\Omega_{\varepsilon,R}$,
\begin{align*}
d_{\mathcal D_{T_R}}(\mathbf v_\varepsilon,\mathbf v)+\|v_\varepsilon-v\|_{\mathcal L_{T_R}^{\beta,\sigma}}\leq C_{R,\lambda_*}(\|\mathbb Z_\varepsilon^0-\mathbb Z^0\|_{\mathfrak X_{T_R,\alpha}^0}+\|v_{\varepsilon,0}-v_0\|_\beta).
\end{align*}
Equation~\eqref{eq:finite-time-model-convergence} and \eqref{eq:well-prepared-convergence-assumption} prove \eqref{eq:localized-remainder-convergence}.  Since the first coordinate of the finite-time model is $X$, the same model convergence also gives
\begin{equation*}
X_\varepsilon\to X\ \text{in probability in }C([0,T];\bC^{-\frac12-\delta}).
\end{equation*}
Together with the continuous embedding of $\mathcal L_{T_R}^{\beta,\sigma}$ into
$C([0,T_R];\bC^{-\frac12-\delta})$, this proves \eqref{eq:localized-full-solution-convergence}.  The convergence of the models in $L^p$ and the convergence of the initial remainders in probability imply tightness of the four norms in \eqref{eq:localized-data-event}, which proves \eqref{eq:localized-events-tight}.  The localized $L^p$ estimate follows by applying \eqref{eq:deterministic-solution-stability} after multiplication by $\mathbf 1_{\Omega_{\varepsilon,R}}$, and using \eqref{eq:finite-time-model-convergence}.

It remains to identify the equations.  On the canonical approximating lift, Proposition~\ref{prop:smooth-consistency-mild-map} gives
\begin{align*}
\widetilde C=10\mu_\varepsilon C_\varepsilon^{(1)}-\lambda_\varepsilon=\wt C_\varepsilon,\ C^{\mathrm{can}}=A_{\varepsilon,\alpha}^{\mathrm{loc}}-9\lambda_\varepsilon^2C_\varepsilon^{(2)}-9B_{\varepsilon,\alpha}-6D_{\varepsilon,\alpha}=C_\varepsilon^{\mathrm{can}}.
\end{align*}
Although $X_\varepsilon$ is not smooth in time and $v_{\varepsilon,0}$ need not be smooth in space, one may first regularize $X_\varepsilon$ in time and $v_{\varepsilon,0}$ in space.  If $X_\varepsilon^{(m)}$ denotes the time-regularized path, then the corresponding complete canonical lifts converge to $\mathbb Z_\varepsilon$ in $\mathfrak X_{T,\alpha}$ and hence, after applying $\mathscr Q_0$, in $\mathfrak X_{T,\alpha}^0$.  Moreover,
\begin{equation*}
\xi_\varepsilon^{(m)}:=(\partial_t-\Delta+1)X_\varepsilon^{(m)}\to\xi_\varepsilon
\end{equation*}
in the distributional topology.  The regularized initial remainders converge to $v_{\varepsilon,0}$ in $\bC^\beta$.  The stability part of Theorem~\ref{thm:deterministic-fixed-point} therefore allows the classical mild identity from Proposition~\ref{prop:smooth-consistency-mild-map} to pass to the limit in both regularization parameters.  This proves \eqref{eq:approximating-equation-recovered} on $\Omega_{\varepsilon,R}$ and $[0,T_R]$.  Applying the same density and continuity argument to the limiting admissible datum identifies $u=X+v$ with the renormalized $\Phi^4_3(\lambda)$ solution on the same localized interval.
\end{proof}

\br
\eqref{eq:localized-remainder-convergence}-- \eqref{eq:localized-events-tight} give convergence of the random local solution germs in the bounded-data sense supplied by the present theory. More explicitly, for every $\rho>0$ one may choose $R_\rho$ so that the limiting upper probability of $\Omega_{\varepsilon,R_\rho}^{\mathrm c}$ is at most $\rho$; on the corresponding deterministic interval $[0,T_{R_\rho}]$, the localized error converges to zero in probability.  This statement does not assert a common deterministic lifetime independent of the localization radius.
\er

\section{Renormalization and convergence of the enhanced data} \label{sec:rough-distribution}

Having established the deterministic solution theory for abstract enhanced data in Section~\ref{sec03}, we now prove the stochastic estimates required for Theorem~\ref{thm:enhanced-data-convergence}. The standard $\Phi^4_3$ coordinates are constructed by the usual Wiener-chaos arguments, so we focus on the additional coordinates generated by the quintic perturbation. Throughout this section, we assume that $\alpha\in(\frac56,1)$ and that $\widetilde C_\varepsilon$ is chosen so that $\sup_{\varepsilon\in(0,1]}|\lambda_\varepsilon|<\infty$.

Our aim is to show that, after subtracting the appropriate local terms, all additional higher-order coordinates vanish in the required enhanced-data topologies. The only non-vanishing contractions are encoded by the scalar constants $D_{\varepsilon,\alpha}$ and $B_{\varepsilon,\alpha}$ defined in \eqref{eq:D-def} and \eqref{eq:B5-explicit}, together with their associated first-chaos multiples of $X_\varepsilon$. The analysis proceeds in three steps. We first establish dyadic estimates for the Wick powers and their stationary heat convolutions. We then treat the higher-order resonant products that vanish without further renormalization. Finally, we isolate the zeroth- and first-chaos contributions producing the local terms associated with $D_{\varepsilon,\alpha}$ and $B_{\varepsilon,\alpha}$.

We retain the notation $C_T^{\zeta,\rho}$ introduced in Section~\ref{subsec:admissible-enhanced-data}, and denote by $\Pi_n$ the projection onto the $n$-th homogeneous Wiener chaos. For a dyadic block $\Delta_q$, we write $K:=2^q\vee1$, with $|k|\sim K$ interpreted as $|k|\lesssim1$ when $q=-1$. All auxiliary time H\"older and logarithmic summation losses are chosen sufficiently small relative to the prescribed loss parameter $\eta$ and are absorbed into factors of the form $\varepsilon^{-\eta}K^\eta$.

\subsection{Renormalization for $\eps^\a X_\eps^{\diamond m}$ and $\eps^\a \cI(X_\eps^{\diamond m})$}

The following dyadic estimate serves as the starting point for the higher-order Wick terms generated by the quintic expansion. We state it only for $m=2,3,4,5$, which are precisely the orders needed. Since $f$ is compactly supported, there is a constant $R_f<\infty$ such that $\Delta_q(X_\eps^{\diamond m})=0$ for all $m\le5$ whenever $2^q>R_f\eps^{-1}$.

\begin{lemma}\label{lem:wick-dyadic-vanishing}
Let $m\in\{2,3,4,5\}$, $T>0$, $p\in[1,\infty)$, $0<\zeta<\ff14$, and $\eta>0$ sufficiently small.  Uniformly in $\eps\in(0,1]$ and $q\ge-1$,
\begin{align}\label{wick-dyadic-01}
\big\|\Delta_q X_\eps^{\diamond m}\|_{L^p(\Omega;C^\zeta([0,T];L^\infty))}\lesssim\eps^{-2\zeta-\eta}\Big((2^q\vee1)^{\ff m2+\eta}+\bbone_{\{m\ge3\}}\eps^{-\ff{m-3}{2}}(2^q\vee1)^{\ff32+\eta}\Big),
\end{align}
and
\begin{align}\label{wick-dyadic-02}
\big\|\Delta_q \cI(X_\eps^{\diamond m})\big\|_{L^p(\Omega;C^\zeta([0,T];L^\infty))}\lesssim \eps^{-\eta} (2^q\vee1)^{\ff m2-2+2\zeta+\eta}.
\end{align}
Here the implicit constants may depend on $m,p,T,\zeta,\eta$ and $f$, but not on $q$ or $\eps$.
\end{lemma}

\begin{proof}
Write $a_k=1+4\pi^2|k|^2$. The Wick formula gives
\begin{align*}
\mE\Big[\widehat{X_\eps^{\diamond m}}(t,k)\widehat{X_\eps^{\diamond m}}(s,\ell)\Big]=m!\,\bbone_{ \{k+\ell=0\} }\sum_{k_1+\cdots+k_m=k}\prod_{i=1}^m\frac{|f(\eps k_i)|^2\e^{-a_{k_i}|t-s|}}{2a_{k_i}}.
\end{align*}

Since $a_k\asymp1+|k|^2$ and $|k_i|\lesssim\eps^{-1}$ on the support of $f$, the bound $1-\e^{-x}\lesssim x^{2\zeta+\vt}$, valid whenever $2\zeta+\vt<1$, yields
\begin{align*}
\mE\big|\widehat{X_\eps^{\diamond m}}(t,k)-\widehat{X_\eps^{\diamond m}}(s,k)\big|^2\lesssim|t-s|^{2\zeta+\vt}\eps^{-4\zeta-2\vt}\sum_{k_1+\cdots+k_m=k}\prod_{i=1}^m\frac{|f(\eps k_i)|^2}{1+|k_i|^2}.
\end{align*}
Set $K:=2^q\vee1$ and
\begin{align*}
S_{m,\eps}(K):=\sum_{|k|\sim K}\sum_{k_1+\cdots+k_m=k}\prod_{i=1}^m\frac{|f(\eps k_i)|^2}{1+|k_i|^2}.
\end{align*}
We claim that, for every $\vt>0$,
\begin{align}\label{eq:dyadic-conv-renorm}
S_{m,\eps}(K)\lesssim_{\vt} K^{m}+\bbone_{ \{m\ge3\} }K^{3}\eps^{3-m-\vt}.
\end{align}

To prove \eqref{eq:dyadic-conv-renorm}, decompose the sum according to dyadic $N$ such that $N\sim 1+\max_i|k_i|$. If $N\lesssim K$, then
\begin{align*}
S_{m,\eps}(K)\lesssim\prod_{i=1}^m\sum_{|k_i|\lesssim K}\frac1{1+|k_i|^2}\lesssim K^m.
\end{align*}
If $N\gg K$, the condition $|k|\sim K$ implies that at least two input frequencies have size comparable to $N$. Assume, without loss of generality, that $|k_1|\sim|k_2|\sim N$. For fixed $k,k_3,\ldots,k_m$, 
\begin{align*}
\sum_{\substack{k_1+k_2=k-k_3-\cdots-k_m\\|k_1|\sim|k_2|\sim N}}\frac1{(1+|k_1|^2)(1+|k_2|^2)}\lesssim N^{-1}.
\end{align*}
Therefore, the contribution of this scale is bounded by
\begin{align*}
K^3N^{-1}\prod_{i=3}^m\sum_{|k_i|\lesssim N}\frac1{1+|k_i|^2}\lesssim K^3N^{m-3}.
\end{align*}
Summing over $K\ll N\lesssim\eps^{-1}$ gives
\begin{align*}
\sum_{K\ll N\lesssim\eps^{-1}}K^3N^{m-3}
\lesssim
\begin{cases}
K^2, & m=2,\\
K^3\log(2+\eps^{-1}K^{-1}), & m=3,\\
K^3\eps^{3-m}, & m=4,5.
\end{cases}
\end{align*}
The first term is absorbed by $K^m$, while $\log(2+\eps^{-1}K^{-1})\lesssim_{\vt}\eps^{-\vt}$. This proves \eqref{eq:dyadic-conv-renorm}.

We next derive the $L^\infty$ block estimate. Fix $\vt>0$ and $r,\bar p<\infty$ such that
$$2\zeta+\vt<1,\ 2\vt\leq\eta,\ \frac3r\leq\vt,\ \bar p\geq p\vee r,\ \frac{\bar p\vt}{2}>1.$$
By the Bernstein inequality, the Minkowski inequality and hypercontractivity on the $m$-th homogeneous Wiener chaos,
\begin{align*}
\left\|\Delta_q\big(X_\eps^{\diamond m}(t)-X_\eps^{\diamond m}(s)\big)\right\|_{L^{\bar p}(\Omega;L^\infty)}&\lesssim K^{3/r}\left\|\Delta_q\big(X_\eps^{\diamond m}(t)-X_\eps^{\diamond m}(s)\big)\right\|_{L^{\bar p}(\Omega;L^r)}\\
&\lesssim K^\vt\sup_{x\in\mT^3}\Big(\mE\big|\Delta_q\big(X_\eps^{\diamond m}(t)-X_\eps^{\diamond m}(s)\big)(x)\big|^2\Big)^{\frac12}.
\end{align*}
By \eqref{eq:dyadic-conv-renorm},
\begin{align*}
\left\|\Delta_q\big(X_\eps^{\diamond m}(t)-X_\eps^{\diamond m}(s)\big)\right\|_{L^{\bar p}(\Omega;L^\infty)}\lesssim|t-s|^{\zeta+\frac\vt2}\eps^{-2\zeta-\eta}\big(K^{\frac m2+\eta}+\bbone_{\{m\geq3\}}\eps^{\frac{3-m}{2}}K^{\frac32+\eta}\big).
\end{align*}
The corresponding fixed-time estimate follows in the same way and is bounded by the same right-hand side without the time factor. Since $\ff{\bar p\vt}2>1$, the Banach-space-valued Garsia--Rodemich--Rumsey lemma proves \eqref{wick-dyadic-01}.

We next consider the integrated object. For fixed $k$, set $A:=\sum_{i=1}^m a_{k_i}$. A direct calculation gives
\begin{align*}
\int_{-\infty}^t\int_{-\infty}^t\e^{-a_k(t-r)}\e^{-a_k(t-r')}\e^{-A|r-r'|}\dif r\dif r'=\frac1{a_k(a_k+A)}.
\end{align*}
Consequently, for $|k|\sim K$,
\begin{align*}
\mE\big|\widehat{\cI(X_\eps^{\diamond m})}(t,k)\big|^2\lesssim\sum_{k_1+\cdots+k_m=k}\prod_{i=1}^m\frac{|f(\eps k_i)|^2}{1+|k_i|^2}\frac1{(1+K^2)(1+K^2+\sum_i|k_i|^2)}.
\end{align*}
Define
\begin{align*}
\widetilde S_{m,\eps}(K):=\sum_{|k|\sim K}\sum_{k_1+\cdots+k_m=k}\prod_{i=1}^m\frac{|f(\eps k_i)|^2}{1+|k_i|^2}\frac1{(1+K^2)(1+K^2+\sum_i|k_i|^2)}.
\end{align*}
We claim that, for every $\vt>0$,
\begin{align}\label{eq:dyadic-conv-integrated-renorm}
\widetilde S_{m,\eps}(K)\lesssim_{\vt}\eps^{-\vt}K^{m-4},\ m=2,\ldots,5.
\end{align}
Indeed, if $N\lesssim K$, the denominator in the definition of $\widetilde S_{m,\eps}(K)$ is bounded below by a constant multiple of $K^4$, and the preceding low-frequency estimate gives $K^{m-4}$. If $N\gg K$, the two-large-frequency estimate above gives $K^3N^{m-3}$, while the denominator is bounded below by a constant multiple of $K^2N^2$. Consequently,
\begin{align*}
\widetilde S_{m,\eps}(K)\lesssim K^{m-4}+K\sum_{K\ll N\lesssim\eps^{-1}}N^{m-5}\lesssim_{\vt}\eps^{-\vt}K^{m-4}.
\end{align*}
For $m=5$, the last sum produces $K\log(2+\eps^{-1}K^{-1})$, which is absorbed into $\eps^{-\vt}K$. This proves \eqref{eq:dyadic-conv-integrated-renorm}.

For the time increment of $\cI(X_\eps^{\diamond m})$, write $h=t-s>0$ and split
\begin{align*}
\widehat{\cI(X_\eps^{\diamond m})}(t,k)-\widehat{\cI(X_\eps^{\diamond m})}(s,k)
=(\e^{-a_kh}-1)\widehat{\cI(X_\eps^{\diamond m})}(s,k)+\int_s^t\e^{-a_k(t-r)}\widehat{X_\eps^{\diamond m}}(r,k)\dif r.
\end{align*}
The first term is controlled by the fixed-time estimate and the bound
\begin{align*}
|1-\e^{-a_kh}|\lesssim h^{\zeta+\ff\vt2}a_k^{\zeta+\ff\vt2}\lesssim h^{\zeta+\ff\vt2}K^{2\zeta+\vt}.
\end{align*}
For the second term we use the following elementary kernel estimate. Let $a\ge1$, $A>0$, and $\gamma:=2\zeta+\vt\in(0,1)$. Then
\begin{align}
\int_0^h\int_0^h \e^{-a(u+v)}\e^{-A|u-v|}\dif u\dif v\lesssim \frac{a^\gamma h^\gamma}{a(a+A)}.\no
\end{align}
Applying this estimate with $a=a_k$ and $A=\sum_i a_{k_i}$ gives
\begin{align*}
\sup_x\mE\left|\Delta_q\big(\cI(X_\eps^{\diamond m})(t)-\cI(X_\eps^{\diamond m})(s)\big)(x)\right|^2\lesssim |t-s|^{2\zeta+\vt}K^{4\zeta+2\vt}\widetilde S_{m,\eps}(K).
\end{align*}
Combining this bound with \eqref{eq:dyadic-conv-integrated-renorm}, and using the same Bernstein and hypercontractivity argument as above, yields
\begin{align*}
\left\|\Delta_q\big(\cI(X_\eps^{\diamond m})(t)-\cI(X_\eps^{\diamond m})(s)\big)\right\|_{L^{\bar p}(\Omega;L^\infty)}\lesssim |t-s|^{\zeta+\ff\vt2}\eps^{-\eta}K^{\ff m2-2+2\zeta+\eta}.
\end{align*}
The fixed-time estimate is identical without the factor $|t-s|^{\zeta+\ff\vt2}$. Another application of the Banach-space-valued Garsia--Rodemich--Rumsey lemma proves \eqref{wick-dyadic-02}.
\end{proof}

We now apply the dyadic bounds to the terms which carry an explicit prefactor $\varepsilon^\alpha$ in the renormalized expansion.

\begin{proposition}\label{prop:eps-wick-vanishing}
Let $T>0$, $p\in[1,\infty)$, and let $\delta>0$ and $\delta'\in(0,\ff\a2-\ff14)$ be sufficiently small. Then, for $m=2,3,4$,
\begin{align*}
\eps^\alpha X_\eps^{\diamond m}\to0\ \text{in }L^p\big(\Omega;C_T^{\delta',\,\alpha-\ff m2-\delta-2\delta'}\big).
\end{align*}
Moreover, for $m=2,3,4,5$,
\begin{align*}
\eps^\alpha \cI(X_\eps^{\diamond m})\to0\ \text{in }L^p\big(\Omega;C_T^{\delta',\,\alpha+2-\ff m2-\delta-2\delta'}\big).
\end{align*}
\end{proposition}

\begin{proof}
Choose $\eta>0$ so small that $\eta<\ff\delta2$ and $2\delta'+\eta<\alpha-\ff12$. Set $K:=2^q\vee1$. By the Fourier cutoff, it suffices to consider
$1\leq K\leq R_f\eps^{-1}$. Moreover,
$$\|F_\eps\|_{L^p(\Omega;C_T^{\delta',\rho})}\lesssim\Big(\sum_{2^q\lesssim\eps^{-1}}K^{p\rho}\|\Delta_qF_\eps\|_{L^p(\Omega;C_t^{\delta'}L^\infty)}^p\Big)^{1/p},$$
where $F_\eps$ denotes any of the random fields considered below. Consequently, a weighted dyadic bound of the form $A_\eps K^\beta$ contributes at most $A_\eps\eps^{-\beta_+}$ after geometric summation.

For the first claim, take $\zeta=\delta'$ in \eqref{wick-dyadic-01}. For $m=2,3,4$,
\begin{align*}
&\big\|\eps^\alpha X_\eps^{\diamond m}\big\|_{L^p(\Omega;C_T^{\delta',\alpha-\ff m2-\delta-2\delta'})}\\
&\qquad \lesssim\eps^{\alpha-2\delta'-\eta}\sup_{1\le K\le R_f\eps^{-1}}K^{\alpha-\delta-2\delta'+\eta}+\bbone_{\{m\ge3\}}\eps^{\alpha-2\delta'-\eta+\ff{3-m}2}\sup_{1\le K\le R_f\eps^{-1}}K^{\alpha+\ff32-\ff m2-\delta-2\delta'+\eta}\\
&\qquad \lesssim\eps^{\alpha-2\delta'-\eta-(\alpha-\delta-2\delta'+\eta)_+}+\bbone_{\{m\ge3\}}\eps^{\alpha-2\delta'-\eta+\ff{3-m}2-(\alpha+\ff32-\ff m2-\delta-2\delta'+\eta)_+}.
\end{align*}
The two exponents on the right-hand side are bounded from below, respectively, by $\min\{\alpha-2\delta'-\eta,\delta-2\eta\}$ and $\min\left\{\alpha-\frac12-2\delta'-\eta,\delta-2\eta\right\}$, which are positive by the choice of $\eta$. This proves the first convergence.

For the integrated terms, take $\zeta=\delta'$ in
\eqref{wick-dyadic-02}. For $m=2,\ldots,5$,
\begin{align*}
\big\|\eps^\alpha \cI(X_\eps^{\diamond m})\big\|_{L^p(\Omega;C_T^{\delta',\alpha+2-\ff m2-\delta-2\delta'})}&\lesssim\eps^{\alpha-\eta}\sup_{1\le K\le R_f\eps^{-1}}K^{\alpha+2-\ff m2-\delta-2\delta'}K^{\ff m2-2+2\delta'+\eta}\\
&=\eps^{\alpha-\eta}\sup_{1\le K\le R_f\eps^{-1}}K^{\alpha-\delta+\eta}\lesssim\eps^{\alpha-\eta-(\alpha-\delta+\eta)_+}.
\end{align*}
The last exponent is bounded from below by $\min\{\alpha-\eta,\delta-2\eta\}>0$, which proves the second convergence.
\end{proof}

\begin{remark}\label{rem:m5-not-standalone}
The non-integrated estimate is not stated for $m=5$, and this omission is essential. Indeed, for every fixed $t$, $\mE|\widehat{\eps^\alpha X_\eps^{\diamond5}}(t,0)|^2\asymp \eps^{2\alpha-2}$, which diverges as $\eps\to0$ since $\alpha<1$. Thus, $\eps^\alpha X_\eps^{\diamond5}$ does not vanish even at the level of its zero spatial Fourier mode. What is needed in the mild formulation is instead the integrated estimate $\eps^\alpha\cI(X_\eps^{\diamond5})\longrightarrow0$, which is precisely the second assertion of Proposition~\ref{prop:eps-wick-vanishing}.
\end{remark}

The preceding proposition gives the required vanishing estimates for the terms involving a single Wick power and an explicit factor $\eps^\alpha$. We now turn to resonant products. These require more care, as contractions between the two Wick factors can partially compensate for the smallness induced by $\eps^\alpha$. The following lemmas show that the relevant higher-order resonant products nevertheless vanish in the spaces required by the paracontrolled formulation.

\subsection{Renormalization for $\eps^{2\a}\cI(X_\eps^{\diamond5})\circ X_\eps^{\diamond3}$}

For the remainder of this section, we use the following notation. Recall that $a_k:=1+4\pi^2|k|^2$. After enlarging $R_f$ by a fixed factor if necessary, set
\begin{align*}
\varrho(k):=1+|k|,\ \Lambda:=1+R_f\eps^{-1},\ w_\eps(k):=\frac{|f(\eps k)|^2}{2a_k}.
\end{align*}
Then
\begin{align*}
w_\eps(k)=0\ \text{for }\varrho(k)>\Lambda;\ w_\eps(k)\lesssim\varrho(k)^{-2}.
\end{align*}

To encode the resonant product, define
\begin{align*}
\varphi_{-1}:=\chi,\ \varphi_j(\xi):=\theta(2^{-j}\xi),\ j\geq0;\ \Phi(k,\ell):=\sum_{\substack{i,j\geq-1\\|i-j|\leq1}}\varphi_i(k)\varphi_j(\ell).
\end{align*}
For $\mathbf p=(p_1,\ldots,p_m)$, write
\begin{align*}
A_{\mathbf p}:=\sum_{i=1}^m a_{p_i}.
\end{align*}
For $m\geq1$, define
\begin{align}\label{definitioncJ}
G_{m,\eps}(k):=\sum_{p_1+\cdots+p_m=k}\prod_{i=1}^m w_\eps(p_i),\ \cJ_{m,\eps}(k):=\sum_{p_1+\cdots+p_m=k}\frac{\prod_{i=1}^m w_\eps(p_i)}{a_k(a_k+A_{\mathbf p})}.
\end{align}
Finally, for $r\geq0$, $P\in\mZ^3$ and $u\geq0$, let
\begin{align*}
H_{\eps,r}(P,u):=\sum_{\boldsymbol\xi\in(\mZ^3)^r}\prod_{\nu=1}^r w_\eps(\xi_\nu)\e^{-\left(a_{P+\Xi}+\sum_{\nu=1}^r a_{\xi_\nu}\right)u},\ \Xi:=\sum_{\nu=1}^r\xi_\nu,
\end{align*}
where $H_{\eps,0}(P,u)=\e^{-a_Pu}$.

We first consider $\eps^{2\alpha}\cI(X_\eps^{\diamond5})\circ X_\eps^{\diamond3}$. Both factors carry the prefactor $\eps^\alpha$, and no local counterterm is needed. We treat this case in some detail, since the same kernel estimates will be used repeatedly below.

\begin{lemma}\label{lem:X5-X3-vanishing}
Let $T>0$, $p\in[1,\infty)$, $\de,\de'>0$ sufficiently small such that $\a-\frac12-\delta-2\delta'>0$. Then
\begin{align}\label{eq:X5-X3-vanishing}
\eps^{2\alpha}\,\cI(X_\eps^{\diamond5})\circ X_\eps^{\diamond3}\to0\ \text{in}\ L^p\big(\Omega;C_T^{\delta',\,2\alpha-2-\delta-2\delta'}\big).
\end{align}
More precisely, if $K:=2^q\vee1$, then, for every sufficiently small $\eta>0$,
\begin{align}\label{eq:X5-X3-dyadic}
\left\|\Delta_q\left(\eps^{2\alpha}\,\cI(X_\eps^{\diamond5})\circ X_\eps^{\diamond3}\right)\right\|_{L^p(\Omega;C^{\delta'}([0,T];L^\infty))}\lesssim_{\eta,\delta'}\eps^{2\alpha-1-2\delta'-\eta}K^{1+\eta}.
\end{align}
\end{lemma}

\begin{proof}
By the Wick formula,
\begin{align}\label{eq:X5-X3-wick-expansion}
\cI(X_\eps^{\diamond5})\circ X_\eps^{\diamond3}=\sum_{r=0}^3c_r\cT_{\eps,r},\ c_r:=\binom5r\binom3r r!,
\end{align}
where $\cT_{\eps,r}$ denotes the contribution arising from exactly $r$ cross-contractions between $X_\eps^{\diamond5}$ and $X_\eps^{\diamond3}$. More explicitly,
\begin{align}\label{eq:X5-X3-Tr-Fourier}
\widehat{\cT_{\eps,r}}(t,n)={}\sum_{k+\ell=n}\Phi(k,\ell)\int_0^\infty\e^{-a_ku}\sum_{\substack{k_1+\cdots+k_5=k\\\ell_1+\ell_2+\ell_3=\ell}}\prod_{\nu=1}^r\left(\bbone_{\{k_\nu+\ell_\nu=0\}}w_\eps(k_\nu)\e^{-a_{k_\nu}u}\right)&\\
\times:\prod_{\nu=r+1}^5\widehat X_\eps(t-u,k_\nu)\prod_{\mu=r+1}^3\widehat X_\eps(t,\ell_\mu):\,\dif u&.\notag
\end{align}
The term $\cT_{\eps,r}$ belongs to the homogeneous Wiener chaos of order $8-2r$. Since $f$ is compactly supported, there is a constant $R_f<\infty$ such that all the blocks in \eqref{eq:X5-X3-wick-expansion} vanish whenever $K>R_f\eps^{-1}$.

We claim that, for every $\vt>0$,
\begin{align}\label{eq:X5-X3-r012-second-moment}
\sup_{t\in[0,T]}\sum_{|n|\sim K}\mE\left|\widehat{\Delta_q\cT_{\eps,r}}(t,n)\right|^2\lesssim_\vt\eps^{-1-\vt}K^{3+\vt},\ r=0,1,2.
\end{align}
To prove this, we first record the convolution estimates used below. The same dyadic argument as in the proof of \eqref{eq:dyadic-conv-renorm} gives, for every $\vt>0$,
\begin{align}\label{eq:X5-X3-basic-convolutions}
G_{2,\eps}(b)\lesssim\varrho(b)^{-1},\ G_{3,\eps}(b)\lesssim_\vt\Lambda^\vt,\ G_{4,\eps}(b)\lesssim_\vt\Lambda^{1+\vt}.
\end{align}
where the logarithmic borderline in the case $m=3$ is absorbed into $\Lambda^\vt$. Likewise, summing \eqref{eq:dyadic-conv-integrated-renorm} for $m=5$ over the dyadic  blocks yields
\begin{align}\label{eq:X5-X3-integrated-five-sum}
\sum_{b\in\mZ^3}\cJ_{5,\eps}(b)\lesssim_\vt\Lambda^{1+\vt}.
\end{align}

Set $m_1:=5-r$, $m_2:=3-r$, and write $\mathbf p=(p_1,\ldots,p_{m_1})$, $\mathbf z=(z_1,\ldots,z_{m_2})$ with $P:=\sum_i p_i$ and $Q:=\sum_jz_j$. By the Wiener isometry applied to the ordered kernel, together with the contractivity of symmetrization and the evenness of $a_\xi$ and $w_\eps(\xi)$, for each fixed $n$ we obtain,
$$\mE\big|\widehat{\Delta_q\cT_{\eps,r}}(t,n)\big|^2\lesssim\cS_{\eps,r}(n),$$
where
$$\cS_{\eps,r}(n):={}\sum_{\substack{\mathbf p\in(\mZ^3)^{m_1},\mathbf z\in(\mZ^3)^{m_2}\\P+Q=n}}\prod_{i=1}^{m_1}w_\eps(p_i)\prod_{j=1}^{m_2}w_\eps(z_j)\times\int_0^\infty\int_0^\infty\e^{-A_{\mathbf p}|u-v|}H_{\eps,r}(P,u)H_{\eps,r}(P,v)\,\dif u\,\dif v.$$

For $r=0$,
\begin{align*}
\int_0^\infty\int_0^\infty\e^{-a_P(u+v)}\e^{-A_{\mathbf p}|u-v|}\,\dif u\,\dif v=\frac1{a_P(a_P+A_{\mathbf p})}.
\end{align*}
For $r=1,2$, the inequality
\begin{align*}
a_{P+\Xi}+\sum_{\nu=1}^ra_{\xi_\nu}\gtrsim a_P+\sum_{\nu=1}^ra_{\xi_\nu}
\end{align*}
implies
\begin{align*}
H_{\eps,r}(P,u)\lesssim\e^{-ca_Pu}\Theta_\eps(u)^r,\qquad\Theta_\eps(u):=\sum_{\xi\in\mZ^3}w_\eps(\xi)\e^{-ca_\xi u}\lesssim\Lambda\wedge u^{-\frac12}.
\end{align*}
Consequently,
\begin{align*}
\int_0^\infty H_{\eps,1}(P,u)\,\dif u\lesssim a_P^{-\frac12},\ \int_0^\infty H_{\eps,2}(P,u)\,\dif u\lesssim\log(2+\Lambda).
\end{align*}
Using \eqref{eq:X5-X3-basic-convolutions} and \eqref{eq:X5-X3-integrated-five-sum}, we obtain, after absorbing logarithmic losses into $\Lambda^\vt$,
$$\cS_{\eps,0}(n)\lesssim\sup_{b\in\mZ^3}G_{3,\eps}(b)\sum_{P\in\mZ^3}\cJ_{5,\eps}(P)\lesssim_\vt\Lambda^{1+\vt},$$
$$\cS_{\eps,1}(n)\lesssim\sum_{P\in\mZ^3}a_P^{-1}G_{4,\eps}(P)G_{2,\eps}(n-P)\lesssim_\vt\Lambda^{1+\vt}\sum_{\varrho(P)\lesssim\Lambda}\varrho(P)^{-2}\varrho(n-P)^{-1}\lesssim_\vt\Lambda^{1+\vt},$$
$$\cS_{\eps,2}(n)\lesssim\log^2(2+\Lambda)\sum_{P\in\mZ^3}G_{3,\eps}(P)G_{1,\eps}(n-P)=\log^2(2+\Lambda)G_{4,\eps}(n)\lesssim_\vt\Lambda^{1+\vt}.$$
Since $\Lambda\asymp\eps^{-1}$ and the block $|n|\sim K$ contains $O(K^3)$ frequencies, this proves \eqref{eq:X5-X3-r012-second-moment}.

We next consider $r=3$. All three factors in $X_\eps^{\diamond3}$ are then contracted, leaving a second-chaos term. Writing $\xi_{1,2,3}:=\xi_1+\xi_2+\xi_3$, define
$$J_{\eps,n}(u):={}\sum_{\xi_1,\xi_2,\xi_3\in\mZ^3}\Phi(n+\xi_{1,2,3},-\xi_{1,2,3})\prod_{\nu=1}^3w_\eps(\xi_\nu)\times\e^{-\left(a_{n+\xi_{1,2,3}}+a_{\xi_1}+a_{\xi_2}+a_{\xi_3}\right)u}.$$
Then
\begin{align}\label{eq:X5-X3-r3-kernel-representation}
\widehat{\cT_{\eps,3}}(t,n)=\sum_{p_1+p_2=n}\int_0^\infty J_{\eps,n}(u):\widehat X_\eps(t-u,p_1)\widehat X_\eps(t-u,p_2):\,\dif u.
\end{align}
Since $|\Phi|\lesssim1$ and $a_{n+\xi_{1,2,3}}+a_{\xi_1}+a_{\xi_2}+a_{\xi_3}\gtrsim1+|\xi_1|^2+|\xi_2|^2+|\xi_3|^2$, a dyadic decomposition gives
\begin{align}\label{eq:X5-X3-r3-kernel-L1}
\sup_{n\in\mZ^3}\int_0^\infty|J_{\eps,n}(u)|\,\dif u\lesssim\sum_{\xi_1,\xi_2,\xi_3}\frac{\prod_{\nu=1}^3|f(\eps\xi_\nu)|^2}{\prod_{\nu=1}^3(1+|\xi_\nu|^2)\left(1+|\xi_1|^2+|\xi_2|^2+|\xi_3|^2\right)}\lesssim\eps^{-1}.
\end{align}
Indeed, on a dyadic scale $N\sim\max_\nu\varrho(\xi_\nu)$, the contribution is bounded by
\begin{align*}
N^{-2}\prod_{\nu=1}^3\sum_{\varrho(\xi_\nu)\lesssim N}\varrho(\xi_\nu)^{-2}\lesssim N.
\end{align*}
Summing over $N\lesssim\Lambda$ proves \eqref{eq:X5-X3-r3-kernel-L1}.

Since $\mE|\widehat{X_\eps^{\diamond2}}(t,n)|^2\lesssim G_{2,\eps}(n)$, the Minkowski inequality and \eqref{eq:X5-X3-r3-kernel-L1} give
\begin{align*}
\mE\big|\widehat{\cT_{\eps,3}}(t,n)\big|^2\lesssim\eps^{-2}\varrho(n)^{-1}.
\end{align*}
Therefore,
\begin{align}\label{eq:X5-X3-r3-second-moment}
\sup_{t\in[0,T]}\sum_{|n|\sim K}\mE\big|\widehat{\Delta_q\cT_{\eps,3}}(t,n)\big|^2\lesssim\eps^{-2}K^2.
\end{align}

We next estimate the time increments. Fix $\vt>0$ sufficiently small that $2\delta'+\vt<1$, and let $h:=t-s>0$. Then
\begin{align}\label{eq:X5-X3-first-chaos-increment}
\mE\big|\widehat X_\eps(t,k)-\widehat X_\eps(s,k)\big|^2=2w_\eps(k)(1-\e^{-a_kh})\lesssim h^{2\delta'+\vt}\eps^{-4\delta'-2\vt}w_\eps(k).
\end{align}
Since the heat kernels above depend only on $u$, replacing $t$ by $s$ changes only the factors involving $X_\eps$. Expanding the difference of the products, each resulting term contains one increment of $X_\eps$, while all remaining factors are estimated as in the fixed-time case. The Wiener isometry and the Cauchy--Schwarz inequality then allow us to apply \eqref{eq:X5-X3-first-chaos-increment} to this increment.

For $r=0$, each increment contributes the factor $h^{2\delta'+\vt}\eps^{-4\delta'-2\vt}$. When the increment occurs in a factor inside $\cI(X_\eps^{\diamond5})$ with frequency $p_i$, the identity $P=\sum_jp_j$ gives $a_{p_i}\lesssim a_P+\sum_{j\ne i}a_{p_j}$, so the same time-integral estimate remains valid. For $r=1,2$, no additional argument is needed, since the fixed-time estimate already used $\e^{-A_{\mathbf p}|u-v|}\leq1$. Hence
\begin{align}\label{eq:X5-X3-r012-time-increment}
\sum_{|n|\sim K}\mE\big|\widehat{\Delta_q\big(\cT_{\eps,r}(t)-\cT_{\eps,r}(s)\big)}(n)\big|^2\lesssim_\vt h^{2\delta'+\vt}\eps^{-1-4\delta'-3\vt}K^{3+\vt},\ r=0,1,2.
\end{align}

For $r=3$, the kernel $J_{\eps,n}(u)$ is independent of $t$. Moreover,
\begin{align*}
\mE\left|\widehat{X_\eps^{\diamond2}}(t,n)-\widehat{X_\eps^{\diamond2}}(s,n)\right|^2\lesssim h^{2\delta'+\vt}\eps^{-4\delta'-2\vt}G_{2,\eps}(n).
\end{align*}
Thus, by \eqref{eq:X5-X3-r3-kernel-representation} and \eqref{eq:X5-X3-r3-kernel-L1},
\begin{align}\label{eq:X5-X3-r3-time-increment}
\sum_{|n|\sim K}\mE\left|\widehat{\Delta_q\big(\cT_{\eps,3}(t)-\cT_{\eps,3}(s)\big)}(n)\right|^2\lesssim h^{2\delta'+\vt}\eps^{-2-4\delta'-2\vt}K^2.
\end{align}

Combining the preceding fixed-time and increment bounds with the Bernstein inequality, hypercontractivity, and the Banach-space-valued Garsia--Rodemich--Rumsey lemma, as in the proof of Lemma~\ref{lem:wick-dyadic-vanishing}, we obtain, after absorbing all auxiliary losses into $\eta$,
\begin{align*}
\left\|\Delta_q\cT_{\eps,r}\right\|_{L^p(\Omega;C^{\delta'}([0,T];L^\infty))}\lesssim_{\eta,\delta'}
\begin{cases}
\eps^{-\frac12-2\delta'-\eta}K^{\frac32+\eta},& r=0,1,2,\\[2mm]
\eps^{-1-2\delta'-\eta}K^{1+\eta},& r=3.
\end{cases}
\end{align*}
All blocks vanish unless $K\lesssim\eps^{-1}$. Therefore
\begin{align*}
\eps^{-\frac12-2\delta'-\eta}K^{\frac32+\eta}\lesssim\eps^{-1-2\delta'-\eta}K^{1+\eta},
\end{align*}
and multiplication by $\eps^{2\alpha}$ proves \eqref{eq:X5-X3-dyadic}.

Finally, choose $\eta<\ff\delta2$. Since $2\alpha-1-\delta-2\delta'>0$, the dyadic characterization of $C_T^{\delta',\,2\alpha-2-\delta-2\delta'}$ gives
\begin{align*}
\left\|\eps^{2\alpha}\cI(X_\eps^{\diamond5})\circ X_\eps^{\diamond3}\right\|_{L^p(\Omega;C_T^{\delta',\,2\alpha-2-\delta-2\delta'})}\ \lesssim\eps^{2\alpha-1-2\delta'-\eta}\sup_{1\le K\lesssim\eps^{-1}}K^{2\alpha-1-\delta-2\delta'+\eta}\lesssim\eps^{\delta-2\eta}\longrightarrow0.
\end{align*}
This proves \eqref{eq:X5-X3-vanishing}.
\end{proof}

\subsection{Renormalization for $\eps^\a\,\cI(X_\eps^{\diamond5})\circ X_\eps^{\diamond2}$}

We next consider the analogous resonant product with $X_\eps^{\diamond2}$. Its proof is a direct variant of Lemma~\ref{lem:X5-X3-vanishing}, so we record only the estimates that change.

\begin{lemma}\label{lem:X5-X2-vanishing}
Let $T>0$, $p\in[1,\infty)$, $\de,\de'>0$ sufficiently small such that $\a-\frac12-\delta-2\delta'>0$. Then
\begin{align*}
\lambda_\eps\eps^\a\,\cI(X_\eps^{\diamond5})\circ X_\eps^{\diamond2}\to0\ \text{in}\ L^p\big(\Omega;C_T^{\delta',\,\a-\frac32-\delta-2\delta'}\big).
\end{align*}
More precisely, if $K:=2^q\vee1$, then, for every sufficiently small $\eta>0$,
\begin{align}\label{eq:X5-X2-dyadic}
\left\|\Delta_q\Big(\lambda_\eps\eps^\a\,\cI(X_\eps^{\diamond5})\circ X_\eps^{\diamond2}\Big)\right\|_{L^p(\Omega;C^{\delta'}([0,T];L^\infty))}\lesssim_{\eta,\delta'}\eps^{\a-\frac12-2\delta'-\eta}K^{1+\eta}.
\end{align}
\end{lemma}

\begin{proof}
Since $\sup_{\eps\in(0,1]}|\lambda_\eps|<\infty$, we suppress $\lambda_\eps$ throughout the proof. The Wick formula gives
$$\cI(X_\eps^{\diamond5})\circ X_\eps^{\diamond2}=\sum_{r=0}^2g_r\cV_{\eps,r},\ g_r:=\binom5r\binom2r r!,$$
where $\cV_{\eps,r}$ is the contribution with exactly $r$ contractions and belongs to the homogeneous Wiener chaos of order $7-2r$.

Repeating the Wiener-isometry estimate from the proof of Lemma~\ref{lem:X5-X3-vanishing}, now with $m_2=2-r$, it remains only to estimate the three contraction cases. For $r=0$, the support of $\Phi$ restricts the two resonant frequencies to $|k|\sim|\ell|\sim L$ with $L\gtrsim K$; hence \eqref{eq:dyadic-conv-integrated-renorm} with $m=5$, together with $G_{2,\eps}(\ell)\lesssim\varrho(\ell)^{-1}$, leaves only a logarithmic sum over $L$. For $r=1$, we use
\begin{align*}
\int_0^\infty H_{\eps,1}(P,u)\,\dif u\lesssim a_P^{-\frac12},\ \sum_{P\in\mZ^3}\varrho(P)^{-2}\varrho(n-P)^{-2}\lesssim\varrho(n)^{-1}.
\end{align*}
The case $r=2$ follows from
$$G_{3,\eps}(n)\lesssim_\eta\Lambda^\eta,\ \int_0^\infty H_{\eps,2}(n,u)\,\dif u
\lesssim\log(2+\Lambda).$$

After absorbing the logarithmic losses, the preceding estimates give
\begin{align*}
\sup_{t\in[0,T]}
\mE\big|\widehat{\Delta_q\cV_{\eps,r}}(t,n)\big|^2\lesssim_\eta
\begin{cases}
\eps^{-\eta},&r=0,2,\\
\eps^{-1-\eta}\varrho(n)^{-1},&r=1.
\end{cases}
\end{align*}
Since the block $|n|\sim K$ contains $O(K^3)$ frequencies and vanishes unless $K\lesssim\eps^{-1}$, it follows that
\begin{align}\label{eq:X5-X2-second-moments}
\sup_{t\in[0,T]}\sum_{|n|\sim K}\mE\big|\widehat{\Delta_q\cV_{\eps,r}}(t,n)\big|^2\lesssim_\eta\eps^{-1-\eta}K^{2+\eta},\ r=0,1,2.
\end{align}

For $0\le s<t\le T$, set $h:=t-s$. Choose $\nu>0$ sufficiently small relative to $\eta$ so that $2\delta'+\nu<1$. The time-increment argument leading to \eqref{eq:X5-X3-r012-time-increment} then gives
\begin{align}\label{eq:X5-X2-time-increments}
\sum_{|n|\sim K}\mE\big|\widehat{\Delta_q\big(\cV_{\eps,r}(t)-\cV_{\eps,r}(s)\big)}(n)\big|^2\lesssim_\eta h^{2\delta'+\nu}\eps^{-1-4\delta'-\eta}K^{2+\eta},\ r=0,1,2.
\end{align}

The same Bernstein, hypercontractivity, and Banach-space-valued Garsia--Rodemich--Rumsey argument as in the preceding proof yields
\begin{align*}
\|\Delta_q\cV_{\eps,r}\|_{L^p(\Omega;C^{\delta'}([0,T];L^\infty))}\lesssim_{\eta,\delta'}\eps^{-\frac12-2\delta'-\eta}K^{1+\eta},\ r=0,1,2.
\end{align*}
Multiplying by $\eps^\alpha$, summing over $r$, and using the boundedness of $\lambda_\eps$ proves \eqref{eq:X5-X2-dyadic}.

Finally, choose $0<\eta<\frac\delta2$. Since $\alpha-\frac12-\delta-2\delta'>0$, the dyadic characterization gives
\begin{align*}
\left\|\lambda_\eps\eps^\a\cI(X_\eps^{\diamond5})\circ X_\eps^{\diamond2}\right\|_{L^p(\Omega;C_T^{\delta',\,\a-\frac32-\delta-2\delta'})}\lesssim\eps^{\a-\frac12-2\delta'-\eta}\sup_{1\le K\lesssim\eps^{-1}}K^{\a-\frac12-\delta-2\delta'+\eta}\lesssim\eps^{\delta-2\eta}\to0.
\end{align*}
\end{proof}

\subsection{Renormalization for $\eps^\a\cI(X_\eps^{\diamond5})\circ X_\eps$}

We next consider the corresponding resonant product with $X_\eps$. Its proof is a simpler variant of Lemma~\ref{lem:X5-X3-vanishing}, so we record only the estimates that differ.

\begin{lemma}\label{lem:X5-X1-vanishing}
Let $T>0$, $p\in[1,\infty)$, $\de,\de'>0$ sufficiently small such that $\a-\frac12-\delta-2\delta'>0$. Then
\begin{align*}
\lambda_\eps\eps^\a\,\cI(X_\eps^{\diamond5})\circ X_\eps\to0\ \text{in}\ L^p\big(\Omega;C_T^{\delta',\,\a-1-\delta-2\delta'}\big).
\end{align*}
More precisely, if $K:=2^q\vee1$, then, for every sufficiently small $\eta>0$,
\begin{align}\label{eq:X5-X1-dyadic}
\left\|\Delta_q\Big(\lambda_\eps\eps^\a\,\cI(X_\eps^{\diamond5})\circ X_\eps\Big)\right\|_{L^p(\Omega;C^{\delta'}([0,T];L^\infty))}\lesssim_{\eta,\delta'}\eps^{\a-\frac12-2\delta'-\eta}K^{\frac12+\eta}.
\end{align}
\end{lemma}

\begin{proof}
We suppress the uniformly bounded factor $\lambda_\eps$. By the Wick formula,
$$\cI(X_\eps^{\diamond5})\circ X_\eps=\cT_{\eps,0}+5\cT_{\eps,1},$$
where $\cT_{\eps,r}$ is the contribution with exactly $r$ contractions and belongs to the homogeneous Wiener chaos of order $6-2r$.

For $r=0$, the support of $\Phi$ restricts $\varrho(k)$ and $\varrho(\ell)$ to a common dyadic scale $L\gtrsim K$. The kernel estimate from the proof of Lemma~\ref{lem:X5-X3-vanishing}, together with \eqref{eq:dyadic-conv-integrated-renorm}, \eqref{eq:X5-X3-basic-convolutions}, and
\begin{align*}
\int_0^\infty H_{\eps,1}(n,u)\,\dif u\lesssim a_n^{-\frac12},
\end{align*}
gives
$$\sup_{t\in[0,T]}\sum_{|n|\sim K}
\mE\big|\widehat{\Delta_q\cT_{\eps,0}}(t,n)\big|^2\lesssim_\eta\sum_{K\lesssim L\lesssim\eps^{-1}}K^3L^{-2}\sum_{|k|\sim L}\cJ_{5,\eps}(k)\lesssim_\eta\eps^{-\eta}K^2,$$
$$\sup_{t\in[0,T]}\sum_{|n|\sim K}\mE\big|\widehat{\Delta_q\cT_{\eps,1}}(t,n)\big|^2\lesssim_\eta\sum_{|n|\sim K}G_{4,\eps}(n)a_n^{-1}\lesssim_\eta\eps^{-1-\eta}K^{1+\eta}.$$
Since the blocks vanish unless $K\lesssim\eps^{-1}$, these estimates imply
\begin{align*}
\sup_{t\in[0,T]}\sum_{|n|\sim K}\mE\big|\widehat{\Delta_q\cT_{\eps,r}}(t,n)\big|^2\lesssim_\eta\eps^{-1-\eta}K^{1+\eta},\ r=0,1.
\end{align*}

For $0\le s<t\le T$, set $h:=t-s$ and choose $\nu>0$ sufficiently
small relative to $\eta$ so that $2\delta'+\nu<1$. The time-increment
argument leading to \eqref{eq:X5-X3-r012-time-increment} gives
\begin{align}
\sum_{|n|\sim K}\mE\big|\widehat{\Delta_q\big(\cT_{\eps,r}(t)-\cT_{\eps,r}(s)\big)}(n)\big|^2\lesssim_\eta h^{2\delta'+\nu}\eps^{-1-4\delta'-\eta}K^{1+\eta},\ r=0,1.\no
\end{align}
The same Bernstein, hypercontractivity, and Banach-space-valued Garsia--Rodemich--Rumsey argument as in Lemma~\ref{lem:X5-X3-vanishing} now yields
$$\|\Delta_q\cT_{\eps,r}\|_{L^p(\Omega;C^{\delta'}([0,T];L^\infty))}\lesssim\eps^{-\frac12-2\delta'-\eta}K^{\frac12+\eta},\ r=0,1,$$
Multiplying by $\lambda_\eps\eps^\alpha$ and summing over $r$ proves \eqref{eq:X5-X1-dyadic}.

Finally, choose $0<\eta<\frac\delta2$. Since $\alpha-\frac12-\delta-2\delta'>0$,
\begin{align*}
\left\|\lambda_\eps\eps^\a\cI(X_\eps^{\diamond5})\circ X_\eps\right\|_{L^p(\Omega;C_T^{\delta',\,\a-1-\delta-2\delta'})}\lesssim\eps^{\a-\frac12-2\delta'-\eta}\sup_{1\le K\lesssim\eps^{-1}}K^{\a-\frac12-\delta-2\delta'+\eta}\lesssim\eps^{\delta-2\eta}\to0.
\end{align*}
\end{proof}

\subsection{Renormalization for $\eps^\a\,\cI(X_\eps^{\diamond2})\circ X_\eps^{\diamond4}$ and $\eps^\a\,\cI(X_\eps^{\diamond4})\circ X_\eps^{\diamond2}$}

We next treat these two mixed resonant products together because they have the same contraction structure and differ only in which Wick factor is integrated. Both are controlled by the kernel estimates from Lemma~\ref{lem:X5-X3-vanishing}, so we record only the estimates that distinguish the two cases.

\begin{lemma}\label{lem:mixed-Z2-Z2-vanishing}
Let $T>0$, $p\in[1,\infty)$, $\de,\de'>0$ sufficiently small such that $\a-\frac12-\delta-2\delta'>0$. Then
\begin{align}\label{eq:mixed-Z2-Z2-vanishing}
\eps^\a\,\cI(X_\eps^{\diamond2})\circ X_\eps^{\diamond4}\to0,\ \eps^\a\,\cI(X_\eps^{\diamond4})\circ X_\eps^{\diamond2}\to0\ \text{in}\ L^p\big(\Omega;C_T^{\delta',\,\a-1-\delta-2\delta'}\big).
\end{align}
More precisely, if $K:=2^q\vee1$, then, for every sufficiently small $\eta>0$,
\begin{align}\label{eq:Z2-Z2-dyadic-24}
\left\|\Delta_q\Big(\eps^\a\,\cI(X_\eps^{\diamond2})\circ X_\eps^{\diamond4}\Big)\right\|_{L^p(\Omega;C^{\delta'}([0,T];L^\infty))}\lesssim_{\eta,\delta'}\eps^{\a-\frac12-2\delta'-\eta}K^{\frac12+\eta},
\end{align}
\begin{align}\label{eq:Z2-Z2-dyadic-42}
\left\|\Delta_q\Big(\eps^\a\,\cI(X_\eps^{\diamond4})\circ X_\eps^{\diamond2}\Big)\right\|_{L^p(\Omega;C^{\delta'}([0,T];L^\infty))}\lesssim_{\eta,\delta'}\eps^{\a-2\delta'-\eta}K^{1+\eta}.
\end{align}
\end{lemma}

\begin{proof}
By the Wick formula,
$$\cI(X_\eps^{\diamond2})\circ X_\eps^{\diamond4}=\sum_{r=0}^2c_r\cU_{\eps,r}^{(2,4)},\ \cI(X_\eps^{\diamond4})\circ X_\eps^{\diamond2}=\sum_{r=0}^2c_r\cU_{\eps,r}^{(4,2)},\ c_r:=\binom2r\binom4r r!.$$
Each $\cU_{\eps,r}^{(m,n)}$ belongs to the homogeneous Wiener chaos of order $6-2r$. Moreover, all dyadic blocks vanish unless $K\lesssim\Lambda\asymp\eps^{-1}$.

For $m=2,4$, the largest-frequency decomposition used in \eqref{eq:dyadic-conv-integrated-renorm} gives
$$\cJ_{m,\eps}(k)\lesssim_\eta\varrho(k)^{-2}\sum_{\substack{N\gtrsim\varrho(k),N\text{ dyadic}}}\frac{N^{m-3+\eta}}{\varrho(k)^2+N^2}\lesssim_\eta\varrho(k)^{m-7+\eta}.$$
Thus,
\begin{align}
\cJ_{2,\eps}(k)\lesssim_\eta\varrho(k)^{-5+\eta},\ \cJ_{4,\eps}(k)\lesssim_\eta\varrho(k)^{-3+\eta}.\no
\end{align}

For $|n|\sim K$, the support of $\Phi$ implies
\begin{align*}
\sup_{P\in\mZ^3}\sum_{\xi\in\mZ^3}|\Phi(P+\xi,n-P-\xi)|\frac{w_\eps(\xi)}{a_{P+\xi}+a_\xi}\lesssim K^{-1}.
\end{align*}
Indeed, the support of $\Phi$ restricts $P+\xi$ and $n-P-\xi$ to comparable frequencies of dyadic size $L\gtrsim K$. For each such $L$, the corresponding sum is bounded by $L^{-1}$: if $\varrho(P)\lesssim L$, this follows from $L^{-2}\sum_{\varrho(\xi)\lesssim L}\varrho(\xi)^{-2}\lesssim L^{-1}$; otherwise, $\varrho(\xi)\asymp\varrho(P)$ and $L^3\varrho(P)^{-4}\lesssim L^{-1}$. Summing over $L$ proves the estimate.

The Wiener-isometry argument from the proof of Lemma~\ref{lem:X5-X3-vanishing}, together with the support of $\Phi$, the preceding bounds for $\cJ_{2,\eps}$ and $\cJ_{4,\eps}$, and \eqref{eq:X5-X3-basic-convolutions}, gives, uniformly in $t\in[0,T]$ and $|n|\sim K$,
\begin{align*}
\mE\big|\widehat{\Delta_q\cU_{\eps,0}^{(2,4)}}(t,n)\big|^2\lesssim_\eta\Lambda^{1+\eta}K^{-2+\eta},\ \mE\big|\widehat{\Delta_q\cU_{\eps,0}^{(4,2)}}(t,n)\big|^2\lesssim_\eta K^{-1+\eta}.
\end{align*}
For the contracted terms, the bounds for $H_{\eps,1}$ and $H_{\eps,2}$ from the same proof, the preceding estimate, and
\begin{align*}
\sum_{P\in\mZ^3}\varrho(P)^{-2}\varrho(n-P)^{-2}\lesssim\varrho(n)^{-1}
\end{align*}
yield
$$\mE\big|\widehat{\Delta_q\cU_{\eps,1}^{(2,4)}}(t,n)\big|^2\lesssim K^{-2}G_{4,\eps}(n)\lesssim_\eta\Lambda^{1+\eta}K^{-2},$$
$$\mE\big|\widehat{\Delta_q\cU_{\eps,2}^{(2,4)}}(t,n)\big|^2
+\sum_{r=1}^2\mE\big|\widehat{\Delta_q\cU_{\eps,r}^{(4,2)}}(t,n)\big|^2\lesssim_\eta\Lambda^\eta\varrho(n)^{-1}.$$
Since $K\lesssim\Lambda$ and the block contains $O(K^3)$ frequencies, the preceding estimates give, uniformly for $r=0,1,2$,
\begin{align*}
\sup_{t\in[0,T]}\sum_{|n|\sim K}\mE\big|\widehat{\Delta_q\cU_{\eps,r}^{(2,4)}}(t,n)\big|^2\lesssim_\eta\eps^{-1-\eta}K^{1+\eta},\ \sup_{t\in[0,T]}\sum_{|n|\sim K}\mE\big|\widehat{\Delta_q\cU_{\eps,r}^{(4,2)}}(t,n)\big|^2\lesssim_\eta\eps^{-\eta}K^{2+\eta}.
\end{align*}

For $0\le s<t\le T$, set $h:=t-s$ and choose $\nu>0$ sufficiently small relative to $\eta$ so that $2\delta'+\nu<1$. Repeating the time-increment argument leading to \eqref{eq:X5-X3-r012-time-increment}, we obtain, uniformly for $r=0,1,2$,
$$\sum_{|n|\sim K}\mE\big|\widehat{\Delta_q\big(\cU_{\eps,r}^{(2,4)}(t)-\cU_{\eps,r}^{(2,4)}(s)\big)}(n)\big|^2\lesssim_\eta h^{2\delta'+\nu}\eps^{-1-4\delta'-\eta}K^{1+\eta},$$
$$\sum_{|n|\sim K}\mE\big|\widehat{\Delta_q\big(\cU_{\eps,r}^{(4,2)}(t)-\cU_{\eps,r}^{(4,2)}(s)\big)}(n)\big|^2\lesssim_\eta h^{2\delta'+\nu}\eps^{-4\delta'-\eta}K^{2+\eta}.$$
The same Bernstein, hypercontractivity, and Banach-space-valued Garsia--Rodemich--Rumsey argument as in Lemma~\ref{lem:X5-X3-vanishing} yields, uniformly for $r=0,1,2$,
$$\left\|\Delta_q\cU_{\eps,r}^{(2,4)}\right\|_{L^p(\Omega;C^{\delta'}([0,T];L^\infty))}\lesssim_{\eta,\delta'}\eps^{-\frac12-2\delta'-\eta}K^{\frac12+\eta},$$
$$\left\|\Delta_q\cU_{\eps,r}^{(4,2)}\right\|_{L^p(\Omega;C^{\delta'}([0,T];L^\infty))}\lesssim_{\eta,\delta'}\eps^{-2\delta'-\eta}K^{1+\eta}.$$
Multiplying by $\eps^\alpha$ and summing over $r$ proves \eqref{eq:Z2-Z2-dyadic-24}--\eqref{eq:Z2-Z2-dyadic-42}.

Finally, choose $0<\eta<\frac\delta2$. Since $\alpha-\frac12-\delta-2\delta'>0$, 
$$\left\|\eps^\alpha\cI(X_\eps^{\diamond2})\circ X_\eps^{\diamond4}\right\|_{L^p(\Omega;C_T^{\delta',\,\alpha-1-\delta-2\delta'})}\lesssim\eps^{\alpha-\frac12-2\delta'-\eta}\sup_{1\le K\lesssim\eps^{-1}}K^{\alpha-\frac12-\delta-2\delta'+\eta}\lesssim\eps^{\delta-2\eta},$$
$$\left\|\eps^\alpha\cI(X_\eps^{\diamond4})\circ X_\eps^{\diamond2}\right\|_{L^p(\Omega;C_T^{\delta',\,\alpha-1-\delta-2\delta'})}\lesssim\eps^{\alpha-2\delta'-\eta}\sup_{1\le K\lesssim\eps^{-1}}K^{\alpha-\delta-2\delta'+\eta}\lesssim\eps^{\delta-2\eta}.$$
Since $\delta-2\eta>0$, this proves \eqref{eq:mixed-Z2-Z2-vanishing}.
\end{proof}

The next two products are different from the previous ones: their full contractions do not vanish. After the zeroth-chaos projection is removed, the remaining non-scalar parts vanish; the scalar projections are kept as local constants.

\subsection{Renormalization for $\eps^\a\cI(X_\eps^{\diamond3})\circ X_\eps^{\diamond3}$}

The full contraction of the cubic--cubic product yields the first new local scalar term $D_{\eps,\a}$, while all its nonzero-chaos components vanish.

\begin{lemma}\label{lem:X3-X3-nonscalar-vanishing}
Let $T>0$, $p\in[1,\infty)$, $\de,\de'>0$ sufficiently small such that $\a-\frac12-\delta-2\delta'>0$. Then
$$
\eps^\a\Big(\cI(X_\eps^{\diamond3})\circ X_\eps^{\diamond3}-\Pi_0\big(\cI(X_\eps^{\diamond3})\circ X_\eps^{\diamond3}\big)\Big)\to0\ \text{in}\ L^p\big(\Omega;C_T^{\delta',\a-1-\delta-2\delta'}\big).
$$
More precisely, if $K:=2^q\vee1$, then, for every sufficiently small
$\eta>0$,
$$
\left\|\Delta_q\left[\eps^\a\Big(\cI(X_\eps^{\diamond3})\circ X_\eps^{\diamond3}-\Pi_0\big(\cI(X_\eps^{\diamond3})\circ X_\eps^{\diamond3}\big)\Big)\right]\right\|_{L^p(\Omega;C_T^{\delta'}L^\infty)}\lesssim_\eta\eps^{\a-2\de'-\eta}K^{1+\eta}.
$$
\end{lemma}

\begin{proof}
By the Wick formula,
$$\cI(X_\eps^{\diamond3})\circ X_\eps^{\diamond3}=\sum_{r=0}^3d_r\widetilde\cT_{\eps,r},\ d_r:=\binom{3}{r}^2r!,$$
where $\widetilde\cT_{\eps,r}$ is the contribution with exactly $r$ contractions and belongs to the homogeneous Wiener chaos of order $6-2r$. Since $d_3\widetilde\cT_{\eps,3}=\Pi_0\big(\cI(X_\eps^{\diamond3})\circ X_\eps^{\diamond3}\big)$, it remains to estimate $r=0,1,2$.

The same Wiener-isometry argument as in Lemma~\ref{lem:X5-X3-vanishing} gives, uniformly in $t\in[0,T]$ and $|n|\sim K$,
$$\mE\big|\widehat{\Delta_q\widetilde\cT_{\eps,0}}(t,n)\big|^2\lesssim\sum_{P+Q=n}|\Phi(P,Q)|^2\cJ_{3,\eps}(P)G_{3,\eps}(Q),$$
$$\mE\big|\widehat{\Delta_q\widetilde\cT_{\eps,1}}(t,n)\big|^2\lesssim\sum_{P+Q=n}a_P^{-1}G_{2,\eps}(P)G_{2,\eps}(Q),$$
$$\mE\big|\widehat{\Delta_q\widetilde\cT_{\eps,2}}(t,n)\big|^2\lesssim\log^2(2+\Lambda)G_{2,\eps}(n).$$
For $r=0$, the support of $\Phi$ restricts $\varrho(P)$ and $\varrho(Q)$ to comparable dyadic frequencies $L\gtrsim K$. Hence \eqref{eq:dyadic-conv-integrated-renorm} with $m=3$ and \eqref{eq:X5-X3-basic-convolutions} give
\begin{align*}
\sup_{t\in[0,T]}\sum_{|n|\sim K}\mE\big|\widehat{\Delta_q\widetilde\cT_{\eps,0}}(t,n)\big|^2\lesssim_\eta\eps^{-\eta}K^3\sum_{L\gtrsim K}L^{-1}\lesssim_\eta\eps^{-\eta}K^2.
\end{align*}
For $r=1$, the convolution estimate
\begin{align*}
\sum_{P\in\mZ^3}\varrho(P)^{-3}\varrho(n-P)^{-1}\lesssim_\eta\varrho(n)^{-1+\eta}
\end{align*}
gives the same bound. The case $r=2$ follows directly from $G_{2,\eps}(n)\lesssim\varrho(n)^{-1}$, with the logarithmic loss absorbed into $\eps^{-\eta}$. Consequently,
\begin{align*}
\sup_{t\in[0,T]}\sum_{|n|\sim K}\mE\big|\widehat{\Delta_q\widetilde\cT_{\eps,r}}(t,n)\big|^2\lesssim_\eta\eps^{-\eta}K^{2+\eta},\ r=0,1,2.
\end{align*}

For $0\le s<t\le T$, set $h:=t-s$. Repeating the time-increment argument leading to \eqref{eq:X5-X3-r012-time-increment}, we obtain, for some $\nu>0$ with $2\de'+\nu<1$ and after absorbing the auxiliary losses into $\eta$,
\begin{align*}
\sum_{|n|\sim K}\mE\big|\widehat{\Delta_q\big(\widetilde\cT_{\eps,r}(t)-\widetilde\cT_{\eps,r}(s)\big)}(n)\big|^2\lesssim_{\eta,\de'}h^{2\de'+\nu}\eps^{-4\de'-\eta}K^{2+\eta},\ r=0,1,2.
\end{align*}
Combining the fixed-time and increment estimates with the Bernstein inequality, hypercontractivity, and the Banach-space-valued Garsia--Rodemich--Rumsey lemma, as in Lemma~\ref{lem:X5-X3-vanishing}, yields
\begin{align*}
\left\|\Delta_q\widetilde\cT_{\eps,r}\right\|_{L^p(\Omega;C^{\de'}([0,T];L^\infty))}\lesssim_{\eta,\de'}\eps^{-2\de'-\eta}K^{1+\eta},\ r=0,1,2.
\end{align*}
Multiplying by $\eps^\a$ and summing over $r$ proves the stated dyadic estimate.

Finally, all blocks vanish unless $K\lesssim\eps^{-1}$. Choosing $0<\eta<\ff\de2$, the dyadic characterization gives
\begin{align*}
&\left\|\eps^\a\Big(\cI(X_\eps^{\diamond3})\circ X_\eps^{\diamond3}-\Pi_0\big(\cI(X_\eps^{\diamond3})\circ X_\eps^{\diamond3}\big)\Big)\right\|_{L^p(\Omega;C_T^{\de',\,\a-1-\de-2\de'})}\\
&\qquad\qquad\qquad\qquad\qquad\lesssim\eps^{\a-2\de'-\eta}\sup_{1\le K\lesssim\eps^{-1}}K^{\a-\de-2\de'+\eta}\lesssim\eps^{\de-2\eta}\to0.
\end{align*}
\end{proof}

Set
\begin{align}\label{eq:D-def}
D_{\eps,\a}:=\frac{10}{3}\lambda_\eps\eps^\a\,\Pi_0\Big(\cI(X_\eps^{\diamond3})\circ X_\eps^{\diamond3}\Big).
\end{align}
Note that
$$\sum_{\substack{ |i-j|\le1}}\varphi_i(k)\varphi_j(-k)=1,$$
since non-neighbouring blocks have disjoint supports. Then for  $a_k=1+4\pi^2|k|^2$ and $k_{1,2,3}:=k_1+k_2+k_3$,
\begin{align*}
D_{\eps,\a}=\frac{10}{3}\lambda_\eps\eps^\a\frac{3!}{2^3}\sum_{k_1,k_2,k_3\in\mZ^3}\frac{\prod_{i=1}^3|f(\eps k_i)|^2}{a_{k_1}a_{k_2}a_{k_3}\big(a_{k_{1,2,3}}+a_{k_1}+a_{k_2}+a_{k_3}\big)}.\no
\end{align*}
The sum has size $\eps^{-1}$, which means $D_{\eps,\a}$ is divergent unless $\a\ge1$.

By Lemma~\ref{lem:X3-X3-nonscalar-vanishing},
\begin{align*}
\frac{10}{3}\lambda_\eps\eps^\a\,\cI(X_\eps^{\diamond3})\circ X_\eps^{\diamond3}-D_{\eps,\a}\to0\ \hbox{in }L^p\big(\Omega;C_T^{\delta',\a-1-\delta-2\delta'}\big).\no
\end{align*}
Thus, at the level of the enhanced noise, the whole cubic--cubic contribution is reduced to the single scalar quantity $D_{\eps,\alpha}$.

\subsection{Renormalization for $\eps^{2\a}\cI(X_\eps^{\diamond4})\circ X_\eps^{\diamond4}$}

The proof is parallel to the preceding cubic--cubic case. The nonzero-chaos components again vanish, while the full contraction defines the scalar constant $B_{\eps,\a}$.

\begin{lemma}\label{lem:X4-X4-nonscalar-vanishing}
Let $T>0$, $p\in[1,\infty)$, $\de,\de'>0$ sufficiently small such that $\a-\frac12-\delta-2\delta'>0$. Then
$$
\eps^{2\a}\Big(\cI(X_\eps^{\diamond4})\circ X_\eps^{\diamond4}-\Pi_0\big(\cI(X_\eps^{\diamond4})\circ X_\eps^{\diamond4}\big)\Big)\to0\ \text{in}\ L^p\big(\Omega;C_T^{\delta',2\a-2-\delta-2\delta'}\big).
$$
More precisely, if $K:=2^q\vee1$, then, for every sufficiently small $\eta>0$,
\begin{align*}
\left\|\Delta_q\left[\eps^{2\a}\Big(\cI(X_\eps^{\diamond4})\circ X_\eps^{\diamond4}-\Pi_0\big(\cI(X_\eps^{\diamond4})\circ X_\eps^{\diamond4}\big)\Big)\right]\right\|_{L^p(\Omega;C^{\delta'}([0,T];L^\infty))}\lesssim_{\eta,\delta'}\eps^{2\a-1-2\delta'-\eta}K^{1+\eta}.\end{align*}
\end{lemma}

\begin{proof}
By the Wick formula,
$$\cI(X_\eps^{\diamond4})\circ X_\eps^{\diamond4}=\sum_{r=0}^4e_r\cS_{\eps,r},\ e_r:=\binom4r^2r!,$$
where $\cS_{\eps,r}$ is the contribution with exactly $r$ contractions and belongs to the homogeneous Wiener chaos of order $8-2r$. Since $e_4\cS_{\eps,4}=\Pi_0\big(\cI(X_\eps^{\diamond4})\circ X_\eps^{\diamond4}\big)$, it remains to estimate $r=0,1,2,3$.

The same Wiener-isometry argument as in Lemma~\ref{lem:X5-X3-vanishing} gives, uniformly in $t\in[0,T]$ and $|n|\sim K$,
$$\mE\big|\widehat{\Delta_q\cS_{\eps,0}}(t,n)\big|^2\lesssim\sum_{P+Q=n}|\Phi(P,Q)|^2\cJ_{4,\eps}(P)G_{4,\eps}(Q),$$
$$\mE\big|\widehat{\Delta_q\cS_{\eps,1}}(t,n)\big|^2\lesssim\sum_{P+Q=n}a_P^{-1}G_{3,\eps}(P)G_{3,\eps}(Q),$$
$$\mE\big|\widehat{\Delta_q\cS_{\eps,2}}(t,n)\big|^2\lesssim\log^2(2+\Lambda)G_{4,\eps}(n).$$
For $r=0$, the support of $\Phi$ restricts $\varrho(P)$ and $\varrho(Q)$ to comparable dyadic frequencies $L\gtrsim K$, so we use \eqref{eq:dyadic-conv-integrated-renorm} with $m=4$. For $r=1$, we use $G_{3,\eps}\lesssim_\eta\Lambda^\eta$ and $\sum_{\varrho(P)\lesssim\Lambda}a_P^{-1}\lesssim\Lambda$, while the case $r=2$ follows from $G_{4,\eps}\lesssim_\eta\Lambda^{1+\eta}$. After absorbing the logarithmic summations and convolution losses into $\eta$, we obtain
$$\sup_{t\in[0,T]}\sum_{|n|\sim K}\mE\big|\widehat{\Delta_q\cS_{\eps,r}}(t,n)\big|^2\lesssim_\eta\eps^{-1-\eta}K^3,\ r=0,1,2.$$

For $r=3$, the estimate
$$H_{\eps,3}(P,u)\lesssim\e^{-ca_Pu}\big(\Lambda^3\wedge u^{-\frac32}\big)$$
implies
$$\int_0^\infty H_{\eps,3}(P,u)\,\dif u\lesssim\Lambda.$$
Consequently,
$$\mE\big|\widehat{\Delta_q\cS_{\eps,3}}(t,n)\big|^2\lesssim\Lambda^2G_{2,\eps}(n),$$
and hence
$$\sup_{t\in[0,T]}\sum_{|n|\sim K}\mE\big|\widehat{\Delta_q\cS_{\eps,3}}(t,n)\big|^2\lesssim\eps^{-2}K^2.$$
Since all blocks vanish unless $K\lesssim\eps^{-1}$, the estimates for $r=0,1,2$ are also bounded by $\eps^{-2-\eta}K^2$. Therefore,
$$\sup_{t\in[0,T]}\sum_{|n|\sim K}\mE\big|\widehat{\Delta_q\cS_{\eps,r}}(t,n)\big|^2\lesssim_\eta\eps^{-2-\eta}K^2,\ r=0,1,2,3.$$

Let $h:=t-s>0$. Using \eqref{eq:X5-X3-first-chaos-increment} and repeating the time-increment argument from the proof of Lemma~\ref{lem:X5-X3-vanishing}, we obtain, for some sufficiently small $\nu>0$ satisfying $2\delta'+\nu<1$,
$$\sum_{|n|\sim K}\mE\big|\widehat{\Delta_q\big(\cS_{\eps,r}(t)-\cS_{\eps,r}(s)\big)}(n)\big|^2\lesssim_{\eta,\delta'}h^{2\delta'+\nu}\eps^{-2-4\delta'-\eta}K^2,\ r=0,1,2,3.$$
Combining the fixed-time and increment estimates with hypercontractivity, the Bernstein inequality, and the Banach-space-valued Garsia--Rodemich--Rumsey lemma yields
$$\left\|\Delta_q\cS_{\eps,r}\right\|_{L^p(\Omega;C_T^{\delta'}L^\infty)}\lesssim_{\eta,\delta'}\eps^{-1-2\delta'-\eta}K^{1+\eta},\ r=0,1,2,3.$$
Multiplying by $\eps^{2\a}$ and summing over $r$ proves the stated dyadic estimate.

Finally, choose $0<\eta<\ff\delta2$. Since all blocks vanish unless $K\lesssim\eps^{-1}$ and $2\a-1-\delta-2\delta'+\eta>0$, 
\begin{align*}
&\left\|\eps^{2\a}\Big(\cI(X_\eps^{\diamond4})\circ X_\eps^{\diamond4}-\Pi_0\big(\cI(X_\eps^{\diamond4})\circ X_\eps^{\diamond4}\big)\Big)\right\|_{L^p(\Omega;C_T^{\delta',\,2\a-2-\delta-2\delta'})}\\
&\qquad\qquad\qquad\qquad\qquad\lesssim\eps^{2\a-1-2\delta'-\eta}\sup_{1\le K\lesssim\eps^{-1}}K^{2\a-1-\delta-2\delta'+\eta}\lesssim\eps^{\delta-2\eta}\to0.
\end{align*}
\end{proof}

Set
\begin{align}\label{eq:B5-explicit}
\begin{split}
B_{\eps,\a}&:=\frac{25}{9}\eps^{2\a}\Pi_0\Big(\cI(X_\eps^{\diamond4})\circ X_\eps^{\diamond4}\Big)\\
&=\frac{25}{9}\eps^{2\a}\ff{4!}{2^4}\sum_{k_1,k_2,k_3,k_4\in\mZ^3}\frac{\prod_{i=1}^4|f(\eps k_i)|^2}{a_{k_1}a_{k_2}a_{k_3}a_{k_4}\big(a_{k_{1,2,3,4}}+a_{k_1}+a_{k_2}+a_{k_3}+a_{k_4}\big)},
\end{split}
\end{align}
where $k_{1,2,3,4}:=k_1+k_2+k_3+k_4$. The sum has size $\eps^{-2}$, which means $B_{\eps,\a}$ is divergent unless $\a\ge1$.

By Lemma~\ref{lem:X4-X4-nonscalar-vanishing},
$$\frac{25}{9}\eps^{2\a}\cI(X_\eps^{\diamond4})\circ X_\eps^{\diamond4}-B_{\eps,\a}\to0\  \text{in}\ L^p\big(\Omega;C_T^{\delta',2\a-2-\delta-2\delta'}\big).$$
Consequently, the fourth--fourth product contributes only the scalar constant $B_{\eps,\alpha}$.

The same constants reappear when one of the Wick powers contains one additional factor of $X_\eps$. Contracting all factors in the lower-order Wick power then leaves one copy of $X_\eps$, so the required local term is a multiple of $X_\eps$ rather than a scalar.

\subsection{Renormalization for $\lambda_\eps\eps^\a\cI(X_\eps^{\diamond3})\circ X_\eps^{\diamond4}$}

The first-chaos part of $\cI(X_\eps^{\diamond3})\circ X_\eps^{\diamond4}$ produces the local counterterm $2D_{\eps,\a}X_\eps$, while all remaining chaos components vanish.

\begin{lemma}\label{lem:X3-X4-D-localization}
Let $T>0$, $p\in[1,\infty)$, $\de,\de'>0$ sufficiently small such that $\a-\frac12-\delta-2\delta'>0$. Then
\begin{align*}
\frac53\lambda_\eps\eps^\a\,\Pi_1\big(\cI(X_\eps^{\diamond3})\circ X_\eps^{\diamond4}\big)-2D_{\eps,\a}X_\eps\to0\ \hbox{in }L^p\big(\Omega;C_T^{\delta',\a-\frac32-\delta-2\delta'}\big),
\end{align*}
and
\begin{align*}
\frac53\lambda_\eps\eps^\a\,(1-\Pi_1)\big(\cI(X_\eps^{\diamond3})\circ X_\eps^{\diamond4}\big)\to0\ \hbox{in }L^p\big(\Omega;C_T^{\delta',\a-\frac32-\delta-2\delta'}\big).
\end{align*}
More precisely, if $K:=2^q\vee1$, then, for every sufficiently small $\eta>0$,
\begin{align}\label{eq:X3-X4-D-first-chaos-dyadic}
\left\|\Delta_q\left(\frac53\lambda_\eps\eps^\a\,\Pi_1\big(\cI(X_\eps^{\diamond3})\circ X_\eps^{\diamond4}\big)-2D_{\eps,\a}X_\eps\right)\right\|_{L^p(\Omega;C_T^{\delta'}L^\infty)}\lesssim_{\eta,\delta'}\eps^{\a-2\delta'-\eta}K^{\frac32+\eta},
\end{align}
\begin{align}\label{eq:X3-X4-D-higher-chaos-dyadic}
\left\|\Delta_q\left(\frac53\lambda_\eps\eps^\a(1-\Pi_1)\big(\cI(X_\eps^{\diamond3})\circ X_\eps^{\diamond4}\big)\right)\right\|_{L^p(\Omega;C_T^{\delta'}L^\infty)}\lesssim_{\eta,\delta'}\eps^{\a-\frac12-2\delta'-\eta}K^{1+\eta}.
\end{align}
\end{lemma}

\begin{proof}
Since $(\lambda_\eps)_{\eps\in(0,1]}$ is bounded, we suppress this factor in the estimates below. By the Wick formula,
$$\cI(X_\eps^{\diamond3})\circ X_\eps^{\diamond4}=\sum_{r=0}^3h_r\cW_{\eps,r}^{(3,4)},\ h_r:=\binom3r\binom4r r!,$$
where $\cW_{\eps,r}^{(3,4)}$ belongs to the homogeneous Wiener chaos of order $7-2r$. The summand corresponding to $r=3$ is the first-chaos projection.

For $n,k\in\mZ^3$, set
$$\Phi_n(k):=\Phi(k,n-k),\ \Gamma_{\eps,3}(n):=\sum_{k\in\mZ^3}\Phi_n(k)a_k\cJ_{3,\eps}(k).$$
Then
$$\Phi_0(k)=1,\ |\Phi_n(k)-1|\lesssim\bbone_{\{\varrho(k)\lesssim\varrho(n)\}}+\frac{\varrho(n)}{\varrho(k)}\bbone_{\{\varrho(k)\gg\varrho(n)\}}.$$
Moreover, the largest-frequency decomposition used in the proof of \eqref{eq:dyadic-conv-integrated-renorm} gives
$$\sum_{\varrho(k)\sim L}a_k\cJ_{3,\eps}(k)\lesssim L,\ 1\le L\lesssim\Lambda.$$
Decomposing the sum at $L\sim\varrho(n)$ therefore yields, for every $\kappa>0$,
\begin{align*}
|\Gamma_{\eps,3}(n)-\Gamma_{\eps,3}(0)|\lesssim\sum_{L\lesssim\varrho(n)}L+\sum_{\varrho(n)\ll L\lesssim\Lambda}\varrho(n)\lesssim\varrho(n)\log(2+\Lambda)\lesssim_\kappa\eps^{-\kappa}\varrho(n).
\end{align*}

A direct Fourier computation gives
$$\widehat{\Pi_1\big(\cI(X_\eps^{\diamond3})\circ X_\eps^{\diamond4}\big)}(t,n)=24\Gamma_{\eps,3}(n)\widehat X_\eps(t,n).$$
Moreover, \eqref{eq:D-def} and $\Phi_0(k)=1$ give
$$D_{\eps,\a}=\frac{10}{3}\lambda_\eps\eps^\a\,3!\,\Gamma_{\eps,3}(0).$$
Thus, setting
$$F_\eps^{(1)}:=\frac53\lambda_\eps\eps^\a\Pi_1\big(\cI(X_\eps^{\diamond3})\circ X_\eps^{\diamond4}\big)-2D_{\eps,\a}X_\eps,$$
we obtain
$$\widehat F_\eps^{(1)}(t,n)=40\lambda_\eps\eps^\a\big(\Gamma_{\eps,3}(n)-\Gamma_{\eps,3}(0)\big)\widehat X_\eps(t,n).$$

Let $h:=t-s>0$ and choose $\nu>0$ such that $\gamma:=2\delta'+\nu<1$. Using
\eqref{eq:X5-X3-first-chaos-increment} and $\sum_{|n|\sim K}w_\eps(n)\lesssim K$, we obtain
$$\sup_{t\in[0,T]}\sum_{|n|\sim K}\mE\big|\widehat{\Delta_qF_\eps^{(1)}}(t,n)\big|^2\lesssim_\kappa\eps^{2\a-\kappa}K^3,$$
$$\sum_{|n|\sim K}\mE\big|\widehat{\Delta_q\big(F_\eps^{(1)}(t)-F_\eps^{(1)}(s)\big)}(n)\big|^2\lesssim_\kappa h^\gamma\eps^{2\a-2\gamma-\kappa}K^3.$$
Hypercontractivity, the Bernstein inequality, and the Banach-space-valued Garsia--Rodemich--Rumsey lemma then yield
$$\left\|\Delta_qF_\eps^{(1)}\right\|_{L^p(\Omega;C_T^{\delta'}L^\infty)}\lesssim_{\eta,\delta'}\eps^{\a-2\delta'-\eta}K^{\frac32+\eta}.$$
This proves \eqref{eq:X3-X4-D-first-chaos-dyadic}.

It remains to estimate $r=0,1,2$. The same Wiener-isometry argument as in Lemma~\ref{lem:X5-X3-vanishing} gives, uniformly in $t\in[0,T]$ and $|n|\sim K$,
$$\mE\big|\widehat{\Delta_q\cW_{\eps,0}^{(3,4)}}(t,n)\big|^2\lesssim\sum_{P+Q=n}|\Phi(P,Q)|^2\cJ_{3,\eps}(P)G_{4,\eps}(Q),$$
$$\mE\big|\widehat{\Delta_q\cW_{\eps,1}^{(3,4)}}(t,n)\big|^2\lesssim\sum_{P+Q=n}a_P^{-1}G_{2,\eps}(P)G_{3,\eps}(Q),$$
$$\mE\big|\widehat{\Delta_q\cW_{\eps,2}^{(3,4)}}(t,n)\big|^2\lesssim\log^2(2+\Lambda)G_{3,\eps}(n).$$
For $r=0$, the support of $\Phi$ restricts $\varrho(P)$ and $\varrho(Q)$ to comparable dyadic frequencies $L\gtrsim K$. Using \eqref{eq:dyadic-conv-integrated-renorm} with $m=3$ and $G_{4,\eps}\lesssim_\eta\eps^{-1-\eta}$ gives a bound of order $\eps^{-1-\eta}K^2$. For $r=1$, we use
$$G_{3,\eps}\lesssim_\eta\Lambda^\eta,\ \sum_{\varrho(P)\lesssim\Lambda}a_P^{-1}G_{2,\eps}(P)\lesssim\log(2+\Lambda),$$
while the case $r=2$ follows directly from the bound on $G_{3,\eps}$. Since $K\lesssim\Lambda\asymp\eps^{-1}$, the last two bounds are also dominated by $\eps^{-1-\eta}K^2$. Hence
$$\sup_{t\in[0,T]}\sum_{|n|\sim K}\mE\big|\widehat{\Delta_q\cW_{\eps,r}^{(3,4)}}(t,n)\big|^2\lesssim_\eta\eps^{-1-\eta}K^{2+\eta},\ r=0,1,2.$$

For $0\le s<t\le T$, set $h:=t-s$. Repeating the time-increment argument leading to \eqref{eq:X5-X3-r012-time-increment}, we obtain, for some $\nu>0$ with $2\delta'+\nu<1$,
$$\sum_{|n|\sim K}\mE\big|\widehat{\Delta_q\big(\cW_{\eps,r}^{(3,4)}(t)-\cW_{\eps,r}^{(3,4)}(s)\big)}(n)\big|^2\lesssim_{\eta,\delta'}h^{2\delta'+\nu}\eps^{-1-4\delta'-\eta}K^{2+\eta},\ r=0,1,2.$$
Together with the fixed-time estimate, hypercontractivity, the Bernstein inequality, and the Banach-space-valued Garsia--Rodemich--Rumsey lemma, this gives
$$\left\|\Delta_q\cW_{\eps,r}^{(3,4)}\right\|_{L^p(\Omega;C_T^{\delta'}L^\infty)}\lesssim_{\eta,\delta'}\eps^{-\frac12-2\delta'-\eta}K^{1+\eta},\ r=0,1,2.$$
Since
$$(1-\Pi_1)\big(\cI(X_\eps^{\diamond3})\circ X_\eps^{\diamond4}\big)=\sum_{r=0}^2h_r\cW_{\eps,r}^{(3,4)},$$
the boundedness of $\lambda_\eps$ proves \eqref{eq:X3-X4-D-higher-chaos-dyadic}.

Finally, choose $0<\eta<\ff\delta2$. Since all blocks vanish unless $K\lesssim\eps^{-1}$ and the relevant powers of $K$ are positive, the two dyadic estimates give
$$
\left\|F_\eps^{(1)}\right\|_{L^p(\Omega;C_T^{\delta',\,\a-\frac32-\delta-2\delta'})}\lesssim\eps^{\a-2\delta'-\eta}\sup_{1\le K\lesssim\eps^{-1}}K^{\a-\delta-2\delta'+\eta}\lesssim\eps^{\delta-2\eta},$$
$$\left\|\frac53\lambda_\eps\eps^\a(1-\Pi_1)\big(\cI(X_\eps^{\diamond3})\circ X_\eps^{\diamond4}\big)\right\|_{L^p(\Omega;C_T^{\delta',\,\a-\frac32-\delta-2\delta'})}\lesssim\eps^{\a-\frac12-2\delta'-\eta}\sup_{1\le K\lesssim\eps^{-1}}K^{\a-\frac12-\delta-2\delta'+\eta}\lesssim\eps^{\delta-2\eta}.$$
Both terms therefore converge to zero.
\end{proof}

\subsection{Renormalization for $\eps^{2\a}\cI(X_\eps^{\diamond5})\circ X_\eps^{\diamond4}$}

The first-chaos part of $\cI(X_\eps^{\diamond5})\circ X_\eps^{\diamond4}$ produces the local counterterm $3B_{\eps,\a}X_\eps$.

\begin{lemma}\label{lem:X5-X4-B-localization}
Let $T>0$, $p\in[1,\infty)$, $\de,\de'>0$ sufficiently small such that $\a-\frac12-\delta-2\delta'>0$. Then
\begin{align}\label{eq:X5-X4-B-first-chaos}
\frac53\eps^{2\a}\,\Pi_1\big(\cI(X_\eps^{\diamond5})\circ X_\eps^{\diamond4}\big)-3B_{\eps,\a}X_\eps\to0\ \text{in }L^p\big(\Omega;C_T^{\delta',\,2\a-\frac52-\delta-2\delta'}\big),
\end{align}
and
\begin{align}\label{eq:X5-X4-B-higher-chaos}
\frac53\eps^{2\a}\,(1-\Pi_1)\big(\cI(X_\eps^{\diamond5})\circ X_\eps^{\diamond4}\big)\to0\ \text{in }L^p\big(\Omega;C_T^{\delta',\,2\a-\frac52-\delta-2\delta'}\big).
\end{align}
More precisely, if $K:=2^q\vee1$, then, for every sufficiently small $\eta>0$,
\begin{align}\label{eq:X5-X4-B-first-chaos-dyadic}
\left\|\Delta_q\left(\frac53\eps^{2\a}\,\Pi_1\big(\cI(X_\eps^{\diamond5})\circ X_\eps^{\diamond4}\big)-3B_{\eps,\a}X_\eps\right)\right\|_{L^p(\Omega;C_T^{\delta'}L^\infty)}\lesssim_{\eta,\delta'}\eps^{2\a-1-2\delta'-\eta}K^{\frac32+\eta},
\end{align}
and
\begin{align}\label{eq:X5-X4-B-higher-chaos-dyadic}
\left\|\Delta_q\left(\frac53\eps^{2\a}(1-\Pi_1)\big(\cI(X_\eps^{\diamond5})\circ X_\eps^{\diamond4}\big)\right)\right\|_{L^p(\Omega;C_T^{\delta'}L^\infty)}\lesssim_{\eta,\delta'}\eps^{2\a-1-2\delta'-\eta}K^{\frac32+\eta}.
\end{align}
\end{lemma}

\begin{proof}
For $\mathbf k=(k_1,\ldots,k_4)\in(\mZ^3)^4$, write
$$k:=k_1+\cdots+k_4,\ W_\eps(\mathbf k):=\prod_{i=1}^4w_\eps(k_i),\ R_n(\mathbf k):=a_{k+n}+\sum_{i=1}^4a_{k_i},$$
and set
$$\Psi(k,n):=\sum_{|i-j|\le1}\varphi_i(k+n)\varphi_j(-k).$$
Note that $\Psi(k,0)=1$. The term with four cross-contractions is
$$\widehat{\Pi_1\big(\cI(X_\eps^{\diamond5})\circ X_\eps^{\diamond4}\big)}(t,n)=120\sum_{\mathbf k}\Psi(k,n)W_\eps(\mathbf k)\int_0^\infty\e^{-R_n(\mathbf k)u}\widehat X_\eps(t-u,n)\,\dif u.$$
Set
$$F_\eps^{(1)}:=\frac53\eps^{2\a}\Pi_1\big(\cI(X_\eps^{\diamond5})\circ X_\eps^{\diamond4}\big)-3B_{\eps,\a}X_\eps.
$$
Equation \eqref{eq:B5-explicit} gives
\begin{align*}
\frac{\widehat{F_\eps^{(1)}}(t,n)}{200\eps^{2\a}}={}&\sum_{\mathbf k}W_\eps(\mathbf k)\left(\frac{\Psi(k,n)}{R_n(\mathbf k)}-\frac1{R_0(\mathbf k)}\right)\widehat X_\eps(t,n)\\
&+\sum_{\mathbf k}\Psi(k,n)W_\eps(\mathbf k)\int_0^\infty\e^{-R_n(\mathbf k)u}\big(\widehat X_\eps(t-u,n)-\widehat X_\eps(t,n)\big)\,\dif u.
\end{align*}
The first term measures the change in the resonant multiplier and heat denominator when $n$ is set to zero; the second is the error from replacing $\widehat X_\eps(t-u,n)$ by $\widehat X_\eps(t,n)$.

We first estimate the coefficient in the first term. For every $\kappa>0$,
\begin{align}\label{eq:X5-X4-localization-sum}
\sum_{\varrho(k)\sim L}\sum_{k_1+\cdots+k_4=k}\frac{W_\eps(\mathbf k)}{R_n(\mathbf k)}\lesssim_\kappa L^{2+\kappa},\ 1\le L\lesssim\Lambda,
\end{align}
uniformly in $n$. Indeed, for fixed $k$, decompose the sum according to $N:=\max_i\varrho(k_i)\gtrsim L$. The same largest-frequency argument as in \eqref{eq:X5-X3-basic-convolutions} bounds the corresponding convolution sum by $N^{1+\kappa}$, whereas $R_n(\mathbf k)\gtrsim N^2$. Summing over dyadic $N\gtrsim L$ and over the $O(L^3)$ possible values of $k$ proves \eqref{eq:X5-X4-localization-sum}.

Moreover,
$$|\Psi(k,n)-1|\lesssim\bbone_{\{\varrho(k)\lesssim\varrho(n)\}}+\frac{\varrho(n)}{\varrho(k)}\bbone_{\{\varrho(k)\gg\varrho(n)\}},\ |a_{k+n}-a_k|\lesssim\varrho(n)\big(\varrho(k)+\varrho(n)\big).$$
Splitting at $\varrho(k)\sim\varrho(n)$ and using $R_n(\mathbf k)\asymp R_0(\mathbf k)$ when $\varrho(k)\gg\varrho(n)$ gives, for the only relevant range $\varrho(n)\lesssim\Lambda$,
$$\sum_{\mathbf k}W_\eps(\mathbf k)\Big|\frac{\Psi(k,n)}{R_n(\mathbf k)}-\frac1{R_0(\mathbf k)}\Big|\lesssim_\kappa\Lambda^{1+\kappa}\varrho(n).$$

Choose $\gamma>2\delta'$ sufficiently close to $2\delta'$ and then $\kappa>0$ sufficiently small that $\gamma+\kappa<1$. Writing $h:=|t-s|$, we have
$$\mE\left|\widehat X_\eps(t,n)-\widehat X_\eps(s,n)\right|^2\lesssim h^\gamma a_n^\gamma w_\eps(n),\ \sum_{|n|\sim K}w_\eps(n)\lesssim K.$$
This controls the first term in the decomposition of $F_\eps^{(1)}$.

For the second term, set $\rho:=\ff{1-\gamma-\kappa}2$. Since $R_n(\mathbf k)\gtrsim\max_i\varrho(k_i)^2$ and $\sum_{\varrho(k)\lesssim N}w_\eps(k)\lesssim N$,
$$\sum_{\mathbf k}W_\eps(\mathbf k)R_n(\mathbf k)^{-1-\rho}\lesssim\sum_{\substack{N\lesssim\Lambda\\N\text{ dyadic}}}N^{2-2\rho}\lesssim\Lambda^{2-2\rho}.$$
The covariance formula for the Ornstein--Uhlenbeck modes gives
$$\left\|\widehat X_\eps(t-u,n)-\widehat X_\eps(t,n)\right\|_{L^2(\Omega)}\lesssim u^\rho a_n^\rho w_\eps(n)^{\frac12},$$
$$\left\|\big(\widehat X_\eps(t-u,n)-\widehat X_\eps(t,n)\big)-\big(\widehat X_\eps(s-u,n)-\widehat X_\eps(s,n)\big)\right\|_{L^2(\Omega)}\lesssim h^{\frac\gamma2}u^\rho a_n^{\frac\gamma2+\rho}w_\eps(n)^{\frac12}.$$
Using the Minkowski inequality and $K\lesssim\Lambda$, the two terms in
the decomposition of $F_\eps^{(1)}$ therefore satisfy the common bounds
$$\sup_{t\in[0,T]}\sum_{|n|\sim K}\mE\big|\widehat{\Delta_qF_\eps^{(1)}}(t,n)\big|^2\lesssim_\kappa\eps^{4\a}\Lambda^{2+2\gamma+2\kappa}K^3,$$
$$\sum_{|n|\sim K}\mE\big|\widehat{\Delta_q\big(F_\eps^{(1)}(t)-F_\eps^{(1)}(s)\big)}(n)
\big|^2\lesssim_\kappa h^\gamma\eps^{4\a}\Lambda^{2+2\gamma+2\kappa}K^3.
$$
Hypercontractivity, the Bernstein inequality, and the Banach-space-valued Garsia--Rodemich--Rumsey lemma now yield, after absorbing the auxiliary losses into $\eta$,
$$\left\|\Delta_qF_\eps^{(1)}\right\|_{L^p(\Omega;C_T^{\delta'}L^\infty)}\lesssim_{\eta,\delta'}\eps^{2\a-1-2\delta'-\eta}K^{\frac32+\eta}.$$
This proves \eqref{eq:X5-X4-B-first-chaos-dyadic}.

It remains to control the higher-chaos part. By the Wick formula,
$$(1-\Pi_1)\big(\cI(X_\eps^{\diamond5})\circ X_\eps^{\diamond4}\big)=\sum_{r=0}^3\binom5r\binom4r r!\,\cQ_{\eps,r}^{(5,4)},$$
where $\cQ_{\eps,r}^{(5,4)}$ is the contribution with exactly $r$ cross-contractions and belongs to the homogeneous Wiener chaos of order $9-2r$.

Besides \eqref{eq:X5-X3-basic-convolutions}, we only need
$$G_{5,\eps}(b)\lesssim_\kappa\Lambda^{2+\kappa},\ \int_0^\infty H_{\eps,3}(P,u)\,\dif u\lesssim\int_0^\infty\e^{-ca_Pu}\big(\Lambda^3\wedge u^{-\frac32}\big)\,\dif u\lesssim\Lambda.$$
The same Wiener-isometry estimates as in Lemma~\ref{lem:X5-X3-vanishing} give, uniformly in $t\in[0,T]$ and $|n|\sim K$,
$$\mE\big|\widehat{\Delta_q\cQ_{\eps,0}^{(5,4)}}(t,n)\big|^2\lesssim\sum_{P+Q=n}|\Phi(P,Q)|^2\cJ_{5,\eps}(P)G_{4,\eps}(Q),$$
$$\mE\big|\widehat{\Delta_q\cQ_{\eps,1}^{(5,4)}}(t,n)\big|^2\lesssim\sum_{P+Q=n}a_P^{-1}G_{4,\eps}(P)G_{3,\eps}(Q),$$
$$\mE\big|\widehat{\Delta_q\cQ_{\eps,2}^{(5,4)}}(t,n)\big|^2\lesssim\log^2(2+\Lambda)G_{5,\eps}(n),$$
$$\mE\big|\widehat{\Delta_q\cQ_{\eps,3}^{(5,4)}}(t,n)\big|^2\lesssim\Lambda^2G_{3,\eps}(n).$$
Together with \eqref{eq:X5-X3-integrated-five-sum}, \eqref{eq:X5-X3-basic-convolutions}, and $\sum_{\varrho(P)\lesssim\Lambda}a_P^{-1}\lesssim\Lambda$, this yields
$$\sup_{t\in[0,T]}\sum_{|n|\sim K}\mE\big|\widehat{\Delta_q\cQ_{\eps,r}^{(5,4)}}(t,n)\big|^2\lesssim_\kappa\Lambda^{2+\kappa}K^3,\ r=0,1,2,3.$$
Using \eqref{eq:X5-X3-first-chaos-increment} and the time-increment argument from Lemma~\ref{lem:X5-X3-vanishing}, we similarly obtain
$$\sum_{|n|\sim K}\mE\big|\widehat{\Delta_q\big(\cQ_{\eps,r}^{(5,4)}(t)-\cQ_{\eps,r}^{(5,4)}(s)\big)}(n)\big|^2\lesssim_\kappa h^\gamma\Lambda^{2+2\gamma+\kappa}K^3,\ r=0,1,2,3.$$
Consequently,
$$\left\|\Delta_q\cQ_{\eps,r}^{(5,4)}\right\|_{L^p(\Omega;C_T^{\delta'}L^\infty)}\lesssim_{\eta,\delta'}\eps^{-1-2\delta'-\eta}K^{\frac32+\eta},\ r=0,1,2,3.$$
Multiplying by $\frac53\eps^{2\a}$ and summing over $r$ proves
\eqref{eq:X5-X4-B-higher-chaos-dyadic}.

Finally, choose $0<\eta<\ff\delta2$. Both expressions in \eqref{eq:X5-X4-B-first-chaos} and \eqref{eq:X5-X4-B-higher-chaos} have no dyadic blocks above $K\lesssim\eps^{-1}$ and satisfy the same dyadic estimate. Since $2\a-1-\delta-2\delta'+\eta>0$, their norms in  $L^p\big(\Omega;C_T^{\delta',\,2\a-\frac52-\delta-2\delta'}\big)$ are bounded by
$$\eps^{2\a-1-2\delta'-\eta}\sup_{1\le K\lesssim\eps^{-1}}K^{2\a-1-\delta-2\delta'+\eta}\lesssim\eps^{\delta-2\eta}\to0.$$
This proves both convergence statements.
\end{proof}

Proposition~\ref{prop:eps-wick-vanishing} and Lemmas~\ref{lem:X5-X3-vanishing}--\ref{lem:X5-X4-B-localization} verify the convergence of all higher-order coordinates in the canonical lift. Together with the standard $\Phi^4_3$ construction, these estimates complete the stochastic input used in Theorem~\ref{thm:enhanced-data-convergence}.


\begin{thebibliography}{99}

\bibitem{BCD11} H. Bahouri, J.-Y. Chemin and R. Danchin, \emph{Fourier analysis and nonlinear partial differential equations}, Grundlehren der mathematischen Wissenschaften [Fundamental Principles of Mathematical Sciences], Heidelberg, Springer, 2011. \url{https://doi.org/10.1007/978-3-642-16830-7}

\bibitem{BFS83} D. C. Brydges, J. Fr\"ohlich and A. D. Sokal, A new proof of the existence and nontriviality of the continuum $\varphi^4_2$ and $\varphi^4_3$ quantum field theories, \emph{Comm. Math. Phys.} \textbf{91} (1983), no.~2, 141--186. \url{https://doi.org/10.1007/BF01211157}

\bibitem{BG20} N. Barashkov and M. Gubinelli, A variational method for $\Phi^4_3$, \emph{Duke Math. J.} \textbf{169} (2020), no.~17, 3339--3415. \url{https://doi.org/10.1215/00127094-2020-0029}

\bibitem{BHZ19} Y. Bruned, M. Hairer and L. Zambotti, Algebraic renormalisation of regularity structures, \emph{Invent. Math.} \textbf{215} (2019), no.~3, 1039--1156. \url{https://doi.org/10.1007/s00222-018-0841-x}

\bibitem{CC18} R. Catellier and K. Chouk, Paracontrolled distributions and the three-dimensional stochastic quantization equation, \emph{Ann. Probab.} \textbf{46} (2018), no.~5, 2621--2679. \url{https://doi.org/10.1214/17-AOP1235}

\bibitem{CGW22} A. Chandra, T. S. Gunaratnam and H. Weber, Phase transitions for $\phi^4_3$, \emph{Comm. Math. Phys.} \textbf{392} (2022), no.~2, 691--782. \url{https://doi.org/10.1007/s00220-022-04353-6}

\bibitem{CW17} A. Chandra and H. Weber, Stochastic PDEs, regularity structures, and interacting particle systems, \emph{Ann. Fac. Sci. Toulouse Math. (6)} \textbf{26} (2017), no.~4, 847--909. \url{https://doi.org/10.5802/afst.1555}

\bibitem{DPD03} G. Da Prato and A. Debussche, Strong solutions to the stochastic quantization equations, \emph{Ann. Probab.} \textbf{31} (2003), no.~4, 1900--1916. \url{https://doi.org/10.1214/aop/1068646370}

\bibitem{EX22} D. Erhard and W. Xu, Weak universality of dynamical $\Phi^4_3$: polynomial potential and general smoothing mechanism, \emph{Electron. J. Probab.} \textbf{27} (2022), article no.~112, 1--43. \url{https://doi.org/10.1214/22-EJP833}

\bibitem{FG19} M. Furlan and M. Gubinelli, Weak universality for a class of 3d stochastic reaction--diffusion models, \emph{Probab. Theory Related Fields} \textbf{173} (2019), nos.~3--4, 1099--1164. \url{https://doi.org/10.1007/s00440-018-0849-6}

\bibitem{GJ87} J. Glimm and A. Jaffe, \emph{Quantum physics: a functional integral point of view}, 2nd ed., Springer-Verlag, New York, 1987. \url{https://doi.org/10.1007/978-1-4612-4728-9}

\bibitem{GH19} M. Gubinelli and M. Hofmanov\'a, Global solutions to elliptic and parabolic $\Phi^4$ models in Euclidean space, \emph{Comm. Math. Phys.} \textbf{368} (2019), no.~3, 1201--1266. \url{https://doi.org/10.1007/s00220-019-03398-4}

\bibitem{GH21} M. Gubinelli and M. Hofmanov\'a, A PDE construction of the Euclidean $\Phi^4_3$ quantum field theory, \emph{Comm. Math. Phys.} \textbf{384} (2021), no.~1, 1--75. \url{https://doi.org/10.1007/s00220-021-04022-0}

\bibitem{GIP15} M. Gubinelli, P. Imkeller and N. Perkowski, Paracontrolled distributions and singular PDEs, \emph{Forum Math. Pi}, vol. 3, 2015, Art. no. e6. \url{https://doi.org/10.1017/fmp.2015.2}

\bibitem{GMW25} P. Grazieschi, K. Matetski and H. Weber, The dynamical Ising--Kac model in 3D converges to $\Phi^4_3$, \emph{Probab. Theory Related Fields} \textbf{191} (2025), nos.~1--2, 671--778. \url{https://doi.org/10.1007/s00440-024-01316-x}

\bibitem{Hairer14} M. Hairer, A theory of regularity structures, \emph{Invent. Math.} \textbf{198} (2014), no.~2, 269--504. \url{https://doi.org/10.1007/s00222-014-0505-4}

\bibitem{HM18} M. Hairer and K. Matetski, Discretisations of rough stochastic PDEs, \emph{Ann. Probab.} \textbf{46} (2018), no.~3, 1651--1709. \url{https://doi.org/10.1214/17-AOP1212}

\bibitem{HX18} M. Hairer and W. Xu, Large-scale behavior of three-dimensional continuous phase coexistence models, \emph{Comm. Pure Appl. Math.} \textbf{71} (2018), no.~4, 688--746. \url{https://doi.org/10.1002/cpa.21738}

\bibitem{Kup16} A. Kupiainen, Renormalization group and stochastic PDEs, \emph{Ann. Henri Poincar'e} \textbf{17} (2016), no.~3, 497--535. \url{https://doi.org/10.1007/s00023-015-0408-y}

\bibitem{MW17} J.-C. Mourrat and H. Weber, The dynamic $\Phi^4_3$ model comes down from infinity, \emph{Comm. Math. Phys.} \textbf{356} (2017), no.~3, 673--753. \url{https://doi.org/10.1007/s00220-017-2997-4}

\bibitem{MW17IK} J.-C. Mourrat and H. Weber, Convergence of the two-dimensional dynamic Ising--Kac model to $\Phi^4_2$, \emph{Comm. Pure Appl. Math.} \textbf{70} (2017), no.~4, 717--812. \url{https://doi.org/10.1002/cpa.21655}

\bibitem{MWX17} J.-C. Mourrat, H. Weber and W. Xu, Construction of $\Phi^4_3$ diagrams for pedestrians, in \emph{From Particle Systems to Partial Differential Equations}, Springer Proc. Math. Stat., vol.~209, Springer, Cham, 2017, pp.~1--46. \url{https://doi.org/10.1007/978-3-319-66839-0_1}

\bibitem{PW81} G. Parisi and Y.-S. Wu, Perturbation theory without gauge fixing, \emph{Sci. Sinica} \textbf{24} (1981), no.~4, 483--496. \url{https://doi.org/10.1360/ya1981-24-4-483}

\bibitem{SX18} H. Shen and W. Xu, Weak universality of dynamical $\Phi^4_3$: non-Gaussian noise, \emph{Stoch. Partial Differ. Equ. Anal. Comput.} \textbf{6} (2018), no.~2, 211--254. \url{https://doi.org/10.1007/s40072-017-0107-4}

\bibitem{ZZ18} R. Zhu and X. Zhu, Lattice approximation to the dynamical $\Phi^4_3$ model, \emph{Ann. Probab.} \textbf{46} (2018), no.~1, 397--455. \url{https://doi.org/10.1214/17-AOP1188}

\bibitem{ZZ23WU} R. Zhu and X. Zhu, Weak universality of the dynamical $\Phi^4_3$ model on the whole space, \emph{Potential Anal.} \textbf{58} (2023), no.~2, 295--330. \url{https://doi.org/10.1007/s11118-021-09941-0}

\end{thebibliography}
\end{document}